\documentclass[11pt]{amsart}
\usepackage{enumitem}
\usepackage{color}
\usepackage{amssymb,verbatim}
\usepackage{mathrsfs}
\usepackage{dsfont}
\usepackage[colorlinks,citecolor=blue,urlcolor=black,linkcolor=black]{hyperref}

\newtheorem{theorem}{Theorem}[section]
\newtheorem*{theorem*}{Theorem}
\newtheorem*{definition*}{Definition}
\newtheorem{prop}[theorem]{Proposition}
\newtheorem{conjecture}[theorem]{Conjecture}
\newtheorem{claim}{Claim}[theorem]
\newtheorem{subclaim}{Subclaim}[claim]

\newtheorem{lemma}[theorem]{Lemma}
\newtheorem{cor}[theorem]{Corollary}

\newtheorem{question}{Question}

\makeatother

\theoremstyle{definition}
\newtheorem{definition}[theorem]{Definition}
\newtheorem{notation}[theorem]{Notation}
\newtheorem{conv}[theorem]{Convention}

\newtheorem*{setup}{Setup}

\theoremstyle{remark}
\newtheorem{remark}[theorem]{Remark}

\def\s{\subseteq}
\def\sq{\sqsubseteq}
\def\forces{\Vdash}
\def\br{\blacktriangleright}

    \def\LE{\le}
    
\newcommand{\one}{\mathop{1\hskip-3pt {\rm l}}} 
\newcommand{\Ult}{\mathrm{Ult}}% \newcommand{\one}{\mathbbm{1}}
\renewcommand{\restriction}{\mathbin\upharpoonright}    	% by default in amssymb it's mathrel
\renewcommand{\mid}{\mathrel{|}\allowbreak}

\newcommand{\diagonal}{\bigtriangleup}

\newcommand{\cat}{{}^{\curvearrowright}}
\DeclareMathOperator{\stem}{stem}

\newcommand\old[1]{}

\DeclareMathOperator{\range}{range}
\DeclareMathOperator{\supp}{supp}

\DeclareMathOperator{\crit}{crit}

\DeclareMathOperator{\dom}{dom}
\DeclareMathOperator{\ran}{ran}
\DeclareMathOperator{\gch}{GCH}

\DeclareMathOperator{\otp}{otp}
\DeclareMathOperator{\ob}{OB}
\DeclareMathOperator{\rng}{Im}
\DeclareMathOperator{\acc}{acc}

\DeclareMathOperator{\cf}{cf}
\DeclareMathOperator{\refl}{Refl}
\DeclareMathOperator{\ord}{Ord}

\DeclareMathOperator{\id}{id}
\DeclareMathOperator{\image}{''}

\newcommand{\ZFC}{\mathrm{ZFC}}
\newcommand{\GCH}{\mathrm{GCH}}
\newcommand{\Ord}{\mathrm{Ord}}
\newcommand\ale[1]{\marginpar{Alejandro: #1}}

\author[Hayut]{Yair Hayut}
\address{Einstein Institute of Mathematics, Hebrew University of Jerusalem, Givat-Ram, 91904, Israel.}
\email{yair.hayut@mail.huji.ac.il}
\author[Poveda]{Alejandro Poveda}
\address{Harvard University, Department of Mathematics and Center of Mathematical Sciences and Applications, Cambridge, MA 02138, USA
	\newline
	\href{https://scholar.harvard.edu/apoveda/home}{Website: https://alejandropovedaruzafa.com/} }
\email{alejandro@cmsa.fas.harvard.edu}

\title[The gluing property]{The gluing property}
\begin{document}
\begin{abstract}
   We introduce a new compactness principle which we call the gluing property. % for measurable cardinals. 
   For a measurable cardinal $\kappa$ and a cardinal $\lambda$, we say that $\kappa$ has the $\lambda$-gluing property if every sequence of $\lambda$-many $\kappa$-complete ultrafilters on $\kappa$ can be glued into an %obtained as projections of an 
   extender. We show that every $\kappa$-compact cardinal has the $2^\kappa$-gluing property, yet non-necessarily the full gluing property. % and that this is the best possible. 
   Finally, we compute the exact consistency-strength for $\kappa$ to have the $\omega$-gluing property - this being $o(\kappa)=\omega_1$.
\end{abstract}

\maketitle

%\tableofcontents

\section{Introduction}
Suppose that $\mathfrak{M}$ is a structure  all of whose \emph{small}  substructures $\mathfrak{N}\s \mathfrak{M}$ have a property $\Phi$. Is it true that  $\mathfrak{M}$ satisfies $\Phi$? %This is arguably a prevalent question in modern Mathematics.  
An affirmative answer to the above is  an instance of a  pheno\-menon called \emph{compactness.} %Experience   indicates that 
There is a natural expectation for compactness when small means finite. %Indeed, several classical results, in various areas, indicate that this is the case when \emph{small} means finite.
%When the answer to this question is affirmative one says, loosely speaking, that \emph{compactness holds.}
For instance,  a classical theorem of De Brujin and Erd\H{o}s \cite{BrujinErdos} establishes that if $\mathcal{G}$ is an infinite graph all of whose  finite subgraphs $\mathcal{H}\s \mathcal{G}$ have chromatic number ${\leq}n$ then $\mathcal{G}$ itself has chromatic number ${\leq}n$. Another example %in this vein 
is \emph{G\"{o}del's compactness theorem} for First Order Logic (FOL). %$\mathcal{L}_{\omega,\omega}$. %\footnote{$\mathcal{L}_{\omega,\omega}$ is the usual language of Mathematics. Namely, there is a primitive symbol $\in$ and formulas are defined recursively via conjuntions, negations and existential quantifications.} This latter says that an arbitrary set of sentences $T$ is satisfiable  provided all of its finite subcollections $S\s T$ are satisfiable.   
%which, in turn, 
Compactness has had several applications in  areas such as Ramsey Theory, Algebra or Topology. %To mention just one, %it is a consequence of 
%the compactness theorem for FOL can be used to show that the \emph{finitary Ramsey's theorem}  follows from its infinitary version (\cite[p. 25]{TodRam}).%that every one-to-one polynomial map $f\colon \mathbb{C}^n\rightarrow\mathbb{C}^n$ is in fact a bijection.\yair{This is more a consequence of completeness of Algebraically closed field theory, right?}% \cite{Marker}.
%\ale{After checking it, I think you are right - it is a consequence of completeness. Please check the new example I gave..}

\smallskip

The above results evidence that $\aleph_0$ is a \emph{compact cardinal}, which invites to the following natural question.  How about if small stands for %, e.g., countable (i.e., 
``cardinality ${<}\aleph_1$''? Should we  expect %such a garden-variety of instances of 
compactness at $\aleph_1$? It turns out  that this is not the case. %To mention just one of the various counter-examples -
For instance, the De Brujin-Erd\H{o}s' theorem does not apply in this  context \cite{ErdosHajnal}. %Indeed, a theorem of Erd\H{o}s and Hajnal \cite{ErdosHajnal} showed that, if the \emph{Continuum Hypothesis} holds, there is a graph $\mathcal{G}$ of size $\aleph_2$ whose chromatic number is $\aleph_1$ but all of its subgraphs $\mathcal{H}$ with $|\mathcal{H}|<\aleph_1$ do have chromatic number $\leq\aleph_0$.
Similarly, ``finite'' can neither be replaced by ``countable'' in G\H{o}del's compactness theorem for FOL.
%\smallskip
%Since compactness  for FOL entails the De Brujin-Erd\H{o}s' theorem,  the above counter-example  is in fact a proof that ``finite'' can neither be replaced by ``countable'' in G\H{o}del's compactness theorem for FOL. % G\H{o}del's compactness theorem for FOL. %First Order Logic. 
Inspired by this intriguing phenomenon Keisler and Tarski investigated for which cardinals $\kappa>\aleph_0$ the infinitary logic $\mathcal{L}_{\kappa,\kappa}$ %-- the extension of $\mathcal{L}_{\omega,\omega}$  in which one allows  conjunctions and quantifications over ${<}\kappa$-many formulae -- 
does satisfy an analogue of the compactness theorem for FOL   \cite{KeislerTarski}. This work lead to the discovery of the so-called  \emph{strongly compact cardinals}, a cornerstone concept in the theory of large cardinals. %We shall come back to these concept shortly during the introduction.

An uncountable cardinal $\kappa$ is  \emph{strongly compact} if any $\kappa$-complete filter extends to a $\kappa$-complete ultrafilter. %\footnote{A filter $\mathcal{F}$ is called $\kappa$-complete if the intersection of ${<}\kappa$-many of its members is yet again a member of $\mathcal{F}$. \S\ref{SectionUltrafilters} provides a self-contained account on this property.} 
This is a natural generalization of the well-known fact that $\aleph_0$ has the \emph{filter extension property}; namely,  that every filter (on any set) extends to an ultrafilter.  %Even though our understanding of the large-cardinal hierarchy has vastly improved since the 60's (see \cite{Kan}) some natural problems regarding strongly compact cardinals are among the most stubborn open questions in set theory.  \smallskip

The motivation behind this work is to isolate a combinatorial principle complying with the following delicate balance.  Namely, the property holds at strongly compact cardinals, but it has  low consistency strength %\footnote{In particular, much weaker than the one corresponding to strongly compact cardinals.} 
and does not follow from large cardinals  weaker than strongly compacts. % cardinal. %\ale{This paragraph is excellent but we should convey this information in more lay-man terms. At this point the reader do not know what consistency strength means.}
%A strongly compact cardinal 
 %generalized the notion of compactness to infinitary logics, defining the \emph{strongly compact cardinals} which are a type of large cardinal. 
%Mimicking the tight connection between G\H{o}del's compactness theorem  and the \emph{Boolean Prime Ideal theorem}, one can characterize the strong compactness of a cardinal $\kappa$ as follows: every $\kappa$-complete filter $\mathscr{F}$ extends to a  $\kappa$-complete ultrafilter $\mathscr{U}$.  
There is a weakening of strong compactness almost fulfilling these -- \emph{$\kappa$-compactness.}\footnote{Recall that an uncountable cardinal $\kappa$ is $\kappa$-compact if every $\kappa$-complete filter \emph{over $\kappa$} can be extended to a $\kappa$-complete ultrafilter.}

In \cite{MitHyper} Mitchell asked the following: Suppose that $\kappa$ is a $\kappa$-compact cardinal. Is it also $\kappa$-compact in a model of the form $L[{\mathcal{U}}]$? %\footnote{%$L[{\mathcal{U}}]$ is the canonical inner model for the anti-large cardinal axiom $``\forall \kappa\, (o(\kappa) \leq \kappa^{++})$''  \S\ref{sectionInnerModel}. 
%At this point is not necessary for the reader to know the exact definition of the model $L[{\mathcal{U}}]$ nor the meaning of the expression $``\forall \kappa\, (o(\kappa) \leq \kappa^{++})$''. It suffices to know that $L[{\mathcal{U}}]$ is a canonical universe of $\ZFC$ which can contain any  normal ultrafilter over $\kappa$ (Def. \ref{Normal}). } 
The meaning of this question is the following. On one hand, the very nature of $L[\mathcal{U}]$ allows to capture normal ultrafilters over $\kappa$. On the other hand, since $\kappa$ is $\kappa$-compact, every $\kappa$-complete filter extends to a $\kappa$-complete ultrafilter. So, perhaps, this  property of $\kappa$ can be  translated into a problem on the existence of a normal ultrafilter, and then the witnessing ultrafilter is  absorbed by $L[\mathcal{U}]$. 
%catching all the normal ultrafilters enables us to extend any $\kappa$-complete filter to a $\kappa$-complete ultrafilter by translating this latter problem to an existence problem of a normal ultrafilter. In that case this witnessing ultrafilter will  be  absorbed within $L[\mathcal{U}]$. %\yair{I changed the formaultion of this paragraph, for two reasons: first, it is easy to find filters that cannot be exteneded to a normal ultrafilter (even normal ones). Second, extending normal filters to $\kappa$-complete ultrafilters is as difficult as $\kappa$-compactness (but this was unknown to Mitchell).}
%\ale{After a second though I do not quite understand the sentnce ``by translating this problem to ... ''}
%It turns out that this is not possible. In effect, 
A negative response to Mitchell's question  was  provided  by Gitik in \cite{GitikOnMeasurables}.

In \cite[\S2]{GitikOnMeasurables} Gitik showed that the {consistency strength} of $\kappa$-compactness %is at least that of a \emph{strong cardinal} -- which 
super-exceeds that of $``\forall \kappa\, (o(\kappa)\leq\kappa^{++})$''. %\footnote{The terminologies $o(\kappa)$ and the notion of an \emph{extender} will be introduced in \S\ref{Preliminaries}.} %is at least that of the existence of a strong cardinal. 
In his proof, Gitik constructs an \emph{extender} of arbitrary length using the  filter extension property of a $\kappa$-compact cardinal. Unfortunately, Gitik's argument contains a subtle gap\footnote{During the construction, Gitik is constructing an extender by picking measures one at a time. In order to verify that the collection of measures induces an extender, at each step, the chosen ultrafilter extends a filter which is originated from the combination of the ultrafilters that were constructed so far. The problem is that, in the limit points, those filters might fail to have the required closure. In order to ensure this closure, it seems like the extensions to ultrafilters must be chosen in a coherent way.} and in order to get past his construction in \cite[Theorem~2.1]{GitikOnMeasurables} a new compactness property is required -- \emph{The Gluing Property}: 
\begin{definition*}
     Let $\kappa$ be a measurable cardinal. We say that $\kappa$ has the \emph{$\lambda$-gluing property} if for every sequence of $\kappa$-complete ultrafilters on $\kappa$, $\langle U_\gamma \mid \gamma < \lambda\rangle$, 
     there is an elementary embedding $j \colon V \to M$, where ${}^\kappa M \subseteq M$, $\crit j = \kappa$
     and an increasing sequence of ordinals $\langle \eta_\gamma \mid \gamma < \lambda\rangle$ such that
     $$U_\gamma = \{X \subseteq \kappa \mid \eta_\gamma \in j(X)\}.$$
\end{definition*}
%Sections~\ref{SectionUltrafilters} and \ref{subsection:extenders} contain all the background material to put in context this  definition and describe its relationship with the theory of \emph{extenders}.  %As usual, the elementary embedding $j \colon V \to M$ witnessing the gluing can be derived from a certain extender, see Subsection \ref{subsection:extenders} for details. 
%This definition catches the possibility to glue different measures. 
%One may speculate whether it is also possible to glue together various \emph{extenders}. We will not deal with this issue here.

\smallskip

The adoption of the gluing property is thus inspired by Gitik's \cite[Theorem~2.1]{GitikOnMeasurables}. In fact, Gitik's argument shows that if $\kappa$ has the $\lambda$-gluing property for every cardinal $\lambda$ then there is an inner model with a strong cardinal.
%\begin{theorem*}[Gitik]
%If $\kappa$ has the $\lambda$-gluing property for every cardinal $\lambda$ then there is an inner model with a strong cardinal. 
%\end{theorem*}
Namely, the assumption ``$\kappa$ is a $\kappa$-compact cardinal'' in \cite[\S2]{GitikOnMeasurables} should be replaced by  ``$\kappa$ has the $\lambda$-gluing property for all cardinals $\lambda$''. Indeed, the $\lambda$-gluing property for $\kappa$ implies that $\kappa$ is \emph{$\lambda$-tall}, see \cite{Hamkins2009}.

\smallskip

 The gluing property incarnates a weakening of the filter extension property of $\kappa$-compactness. %Instead of trying to 
 Rather than extending an arbitrary $\kappa$-complete filter, one extends a filter on increasing sequences which has the property that each projection is already a $\kappa$-complete ultrafilter (see Lemma~\ref{lemma:omega-gluing-by-a-measure}).  As we will see, this shift in perspective reduces considerably the consistency strength.% i.e., the consistency strength of (say) the $\omega$-gluing property is much weaker than that of $\kappa$-compactness. However, it is open if any large cardinal axioms strictly weaker than $\kappa$-compactness have  the $\omega$-gluing property.

% In \cite{Gitikcompact}, using a similar argument as in \cite{GitikOnMeasurables}, Gitik constructs for a $\kappa$-compact cardinal, a forcing notion that does not collapse cardinals and projects onto every $\kappa$-distributive forcing notion of cardinality $\kappa$. As we remarked above, this argument uses implicitly the $2^\kappa$-gluing property for $\kappa$. Fortunately, we can prove that the required gluing property holds by Theorem \ref{kappacompactnessandgluing}.

\smallskip

Let us now describe the content of this paper. In \S\ref{Preliminaries} we present all the necessary set-theoretic background.
%As our title announces, in this paper we investigate the gluing property. 
In \S\ref{section:gluing-from-large-crdinals} we show that $\kappa$-compact cardinals enjoy the $2^\kappa$-gluing property. Specifically, % More precisely, we prove the following:

\begin{theorem*}
If $\kappa$ has the $\lambda$-filter extension property then $\kappa$ has the $2^\lambda$-gluing property. However, it is consistent that $\kappa$ is a $\kappa$-compact cardinal which does not  have the $\kappa^{+\omega+1}$-gluing property.
\end{theorem*}
This result was motivated by a striking result of Gitik \cite{Gitikcompact} saying that if $\kappa$ is $\kappa$-compact then there is a forcing notion that does not collapse cardinals and projects onto every $\kappa$-distributive forcing of cardinality $\kappa$. In his proof, Gitik uses the $2^\kappa$-gluing property of $\kappa$-compact cardinals in order to define such forcing - but this latter claim was based on the argument of \cite[Theorem~2.1]{GitikOnMeasurables}. Fortunately, we can prove that the required gluing property holds for $\kappa$-compact cardinals. %and therefore that such a forcing exists. 
In \S\ref{section:universal} of this paper we provide the reader with a self-contained  exposition of Gitik's poset.

%In \cite{Gitikcompact}, using a similar argument as in \cite[\S2]{GitikOnMeasurables}, Gitik constructs for a $\kappa$-compact cardinal, a forcing notion that does not collapse cardinals and projects onto every $\kappa$-distributive forcing notion of cardinality $\kappa$. As we remarked above, this argument uses implicitly the $2^\kappa$-gluing property for $\kappa$. Fortunately, we can prove that the required gluing property holds by Theorem \ref{kappacompactnessandgluing}.

%Nevertheless we would like to stress that the $2^\kappa$-gluing property (and hence $\kappa$-compactness) is sufficiently strong to derive another striking result due to Gitik -- the existence of a Prikry type forcing that is universal for all $\kappa$-distributive forcings of size $\kappa$ (see \cite{Gitikcompact}). 

\smallskip

Continuing with our study of the gluing property, in \S\ref{section:consistent-omega-gluing}
 we derive the consistency of the $\omega$-gluing property from the existence of a strong cardinal: 
 
 \begin{theorem*}
Assume that $\kappa$ is a strong cardinal. Then, there is a  cardinal-preserving generic extension in which $\kappa$ has the $\omega$-gluing property. %$\lambda$-gluing property for all $\lambda$. 
\end{theorem*}

%Thus, strongness yields an upper bound for the consistency strength of the $\omega$-gluing property. 
The above fits like a glove with Gitik's theorem that $\lambda$-gluing, for every cardinal $\lambda$, yields an inner model with a strong cardinal. Besides, the techniques used there provide valuable information to pin down the exact consistency strength of the $\omega$-gluing property - this is obtained in \S\ref{section; improving} and \S\ref{section:lower-bound}:

%Following this vein, we conclude the paper with \S\S\ref{section; improving} and \ref{section:lower-bound} where we compute the  consistency strength of the $\omega$-gluing property: 

%However, as we will show in this paper, $\kappa$-compact cardinals do not  enjoy of the \emph{global gluing property}. More precisely, in  % for all of the full gluing property.

%In this paper, we investigate this property. In Section \ref{section:gluing-from-large-crdinals} we show that $\kappa$-compactness implies a bounded degree of the gluing property. This gluing property is sufficiently strong in order to derive a striking result due to Gitik---the existence of a Prikry type forcing which is universal for all $\kappa$-distributive forcings of size $\kappa$. In Section \ref{section:universal} we spell out the details of this construction. 
%In Section \ref{section:consistent-bounds-on-gluing} we show that consistently, the results of Section \ref{section:gluing-from-large-crdinals} are optimal. In Section \ref{section:consistent-omega-gluing}
% we derive the consistency of the $\omega$-gluing property from a strong cardinal. 

\begin{theorem*}
The following statements are equiconsistent:
\begin{enumerate}
    \item $\kappa$ has the $\omega$-gluing property.
    \item $\kappa$ is a measurable cardinal with $o(\kappa)=\omega_1$. 
\end{enumerate}
\end{theorem*}
%The difference between those two results is that in order to prove the second one, we work over a suitable core model, with some anti-large cardinal hypothesis, while the first one is a straight-forward forcing argument. 
%Succintly speaking, our For this we mimic the construction used in the above-mentioned theorem, replacing a Gitik iteration of Tree Prikry forcings by a non-stationary-supported iteration. 
We close the paper with \S\ref{OpenProblems} where we spell out some open problems.

%The  notations in this paper are mostly standard. All the background material  appears either in \S\ref{Preliminaries} or will be introduced at a due moment.

%Nevertheless, we assume our readers to be fluent with the theory of elementary embeddings and Prikry-type forcings. For a comprehensive account in these matters we refer the reader to \cite{Kan} and \cite{Gitik-handbook}, respectively.

\section{Set-theoretic preliminaries}\label{Preliminaries}
%The purpose of this section is to  help the reader get acquainted with the set-theoretic background of the paper. Thus this section should serve as a guide and will be referred in further sections 

%In this section we garner some standard background % the bulk of set-theoretic notions that will be  used in the paper. %The cornerstone concept of the manuscript is that of an \emph{ultrafilter}. Most especially, we shall be interested in \emph{$\kappa$-complete ultrafilters} over $\kappa$ for some uncountable regular cardinal $\kappa$. %\footnote{Recall that a cardinal $\kappa$ is called \emph{regular} if for any sequence $\langle \lambda_\alpha\mid \alpha<\theta\rangle$ of ordinals less than $\kappa$ and with length  $\theta<\kappa$ then $(\sup_{\alpha<\theta}\lambda_\alpha)<\kappa.$} 
%The reader is warned that (unless specified) 
%all cardinals considered will be uncountable and regular.% In spite the section is largely self-contained we refer to \cite{Kun} for basic set-theoretic definitions.
\subsection{Ultrafilters and large cardinals}\label{SectionUltrafilters}
Let $\lambda\leq \kappa$ be cardinals and $\mathcal{U}$  a (ultra)filter over $\kappa$. We say that $\mathcal{U}$ is \emph{$\lambda$-complete} if $\bigcap\mathcal{A}\in \mathcal{U}$ for every family $\mathcal{A}\s \mathcal{U}$ with cardinality ${<}\lambda$.
 The (ultra)filter  $\mathcal{U}$ is called \emph{normal} if for every function $f\colon\kappa\rightarrow\kappa$ with $\{\alpha<\kappa\mid f(\alpha)<\alpha\}\in \mathcal{U}$ there is $\beta<\kappa$ such that $$\{\alpha<\kappa\mid f(\alpha)=\beta\}\in\mathcal{U}.$$
    %$\diagonal_{\alpha<\kappa}A_\alpha\in \mathcal{F}$ for every  sequence $\langle A_\alpha\mid\alpha<\kappa\rangle$ of members of $\mathcal{F}$, being
    %$$\diagonal_{\alpha<\kappa}A_\alpha:=\{\alpha<\kappa\mid \forall\beta<\alpha\,(\beta\in A_\alpha)\}.$$
    A cardinal $\kappa$ is said to be  \emph{measurable} if it carries a $\kappa$-complete  (non-trivial) ultrafilter. 
%\begin{definition}[Tarski]
%    $\kappa$ is \emph{measurable} if it carries a $\kappa$-complete ultrafilter.% over $\kappa$. 
%\end{definition}
  Those ultrafilters  will be  hereafter referred  as \emph{measures} (over $\kappa$). % If, in addition, $\mathcal{U}$ is normal we shall say it is a \emph{normal measure}.
   %If, in addition, $\mathcal{U}$ is a normal 
%Measurable cardinals belong to a broad (and well-studied) category termed \emph{large cardinals} or \emph{large-cardinal axioms}. In this introductory section we shall present some large-cardinal siblings of measurability that will be relevant for our purposes.  %Just to help the reader grasp the  essence of these objects - these are cardinals postulating the existence of \emph{large forms of infinity} and whose existence cannot be established in \textsf{ZFC} alone. 
%We refrain to elaborate more than necessary 
%on this issue. Instead, we refer the interested reader to Kanamori's excellent monograph \cite{Kan}.

\smallskip

There is an obvious generalization of measurable cardinals inspired by the \emph{Ultrafilter Theorem}. This is incarnated by the  \emph{$\lambda$-filter extension property}: %It is a theorem of $\mathsf{ZFC}$ that every filter $\mathcal{F}$ admits an extension $\mathcal{U}\supseteq \mathcal{F}$ to a ultrafilter. Trivially, every filter is $\omega$-complete in the sense of Definition~\ref{DefCompleteness}. Thus the ultrafilter theorem says that every $\omega$-complete filter extends to an $\omega$-complete ultrafilter. 
%Replacing $\omega$ by an uncountable regular cardinal in the statement of the ultrafilter theorem  gives rise to  a new (and stronger) large cardinal notion:
%If in this statement we replace $\omega$ by an uncountable regular cardinal we are yielded a new (and stronger) large cardinal concept.

\begin{definition}
    A cardinal $\kappa$ has the \emph{$\lambda$-filter extension property} (for $\lambda\geq \kappa$)
    if  every $\kappa$-complete filter over $\lambda$ extends to a $\kappa$-complete ultrafilter. 
    
    If $\lambda=\kappa$  it is customary to  say that $\kappa$ is \emph{$\kappa$-compact}.
\end{definition}
In \S\ref{section:gluing-from-large-crdinals} of this paper we shall study the connection between such cardinals and the gluing property. In particular, we will show that if $\kappa$ has the $\lambda$-filter extension property then $\kappa$ has the $(2^\lambda)$-gluing property.

\smallskip

The ultimate form of filter extension property is   \emph{strong compactness}, a classical large-cardinal notion isolated by Keisler and Tarski:

%There is yet another natural generalization of the $\lambda$-filter extension property which arises after relaxing the dependence on $\lambda$. This generalization was introduced by Keisler and Tarksi in \cite{KeislerTarski} aiming to extend the  \emph{Compactness Theorem} of First Order Logic to the infinitary language $\mathcal{L}_{\kappa,\kappa}.$ 
\begin{definition}[Keisler \& Tarski {\cite{KeislerTarski}}]
    A cardinal $\kappa$ is called \emph{strongly compact} if every $\kappa$-complete filter extends to a $\kappa$-complete ultrafilter.
\end{definition}
In what concerns the gluing property, strong compactness is powerful enough to entail  $\lambda$-gluing for every cardinal $\lambda$ (Theorem~\ref{kappacompactnessandgluing}).% and that these have the $\lambda$-filter extension property, for all $\lambda$.

\smallskip

Pushing this vein further one reaches the notion of \emph{supercompact cardinal} \cite[\S22]{Kan}. Despite its fundamental role in the theory of large cardinals we shall not introduce this concept for it is not used in the paper. However, there are local forms of supercompactness which will be relevant for us:
\begin{definition}\label{Subcompactness}
   A cardinal $\kappa$ is \emph{$\lambda$-$\Pi^1_1$-subcompact} (for $\lambda\geq \kappa$) if for every $A\s H(\lambda)$ and every $\Pi^1_1$-statement $\Phi$ such that $\langle H(\lambda),\in, A\rangle\models \Phi$, there are:
   \begin{enumerate}
       \item a pair of cardinals $\bar{\kappa}\leq \bar{\lambda}<\kappa$;
       \item a subset $\bar{A}\s H(\bar{\lambda})$ such that $\langle H(\bar{\lambda}),\in, \bar{A}\rangle\models \Phi$, 
       \item and an elementary embedding%\footnote{Recall that a map $j\colon M\rightarrow N$ between two structures in the same language $\mathcal{L}$ is called elementary if for every $\bar{a}\in M^n$ and every $\mathcal{L}$-formula $\varphi(\bar{x})$, $M\models \varphi(\bar{a})$ iff $N\models \varphi(j(\bar{a}))$. }
       $$ j\colon \langle H(\bar{\lambda}),\in,\bar{A}\rangle\rightarrow\langle H(\lambda),\in, A\rangle$$
       with  $\crit(j)=\bar{\kappa}$ and $j(\bar{\kappa})=\kappa.$%\footnote{In this paper $\crit(j)=\bar{\kappa}$ will serve as a shorthand for $``j$ has critical point $\bar{\kappa}$''.}
   \end{enumerate}
\end{definition}
Recall that a $\Pi^1_1$-statement $\Phi$ is a sentence of the form $``\forall X \varphi(X)$'' where $X$ is a second-order variable and the quantifiers of $\varphi(X)$ range only over first-order variables.
The notion of $\kappa^+$-$\Pi^1_1$-subcompactness first appeared in \cite{NeSt} under the name of \emph{$\Pi^2_1$-subcompactness}.  In \S\ref{section:gluing-from-large-crdinals} we   link $\Pi^1_1$-subcompact\-ness with the study of the gluing property. %For instance, we will use it to show that there are $\kappa$-compact cardinals without the $(2^\kappa)^+$-gluing property.
%To help the reader parse the above definition let us clarify the notation used: %Let us conclude with a few clarifications to parse the above definition: 
%as usual,  $H(\lambda)$ denotes  the collection of sets whose transitive closure is of cardinality $<\lambda$; 
 % finally, $\crit(j)=\bar{\kappa}$ is a shorthand for $``j$ has critical point $\bar{\kappa}$'' or, equivalently, $``\bar{\kappa}$ is the first cardinal such that $j(\bar{\kappa})\neq \bar{\kappa}$''.

\subsection{Extenders}\label{subsection:extenders}

In the  \S\ref{section:consistent-omega-gluing} we will consider the following notion:% streng\-thening of measurability:% defined by Reindhardt, Solovay and Kanamori:
\begin{definition}\label{StrongCardinal}
    A cardinal $\kappa$ is \emph{strong} if for all $\lambda\geq \kappa$ there is an elementary embedding $j\colon V\rightarrow M$ with $\crit(j)=\kappa$, $j(\kappa)\geq \lambda$ and $V_\lambda\s M$.
\end{definition}
There is a technical issue when considering such embeddings - they are too powerful to arise from a single measure on $\kappa$. This leads to the definition of an extender which, in rough terms, is a coherent collection of measures.
\begin{definition}
    Let $\kappa$ be a cardinal and $X$ a set.  We denote by $[X]^{\leq\kappa}$  the collection of sets $a\s X$ with $|a| \leq \kappa$. For a set of ordinals, $y$, let $\pi_y \colon y \to \otp(y)$ be the unique order preserving map between $y$ and $\otp(y)$
\end{definition}
For every non-zero $\alpha < \kappa^{+}$ let us fix $f^\alpha \colon \kappa \to \alpha$ a  surjective function. 

\begin{definition}\label{Definitionextender}
   For cardinals $\omega < \kappa < \lambda$ a \emph{$(\kappa,\lambda)$-extender} $E$ is a sequence $\langle E_a\mid \kappa\in a\in  [\lambda \setminus \kappa]^{\leq\kappa}\rangle$ such that  for each   such $a$ %a\in[\lambda\setminus \kappa]^{\leq\kappa}$ %with $\kappa\in a$ 
   the following hold:
\begin{enumerate}
    \item $E_a$ is a $\kappa$-complete, non-$\kappa^+$-complete, ultrafilter over $[\kappa]^{<\kappa}$. % where
      %$$\Sigma_a:=\{\sigma \subseteq \kappa \mid |\sigma|<\kappa\}.$$
    % Besides, $E_a$ is not $\kappa^+$-complete.
      
%      \ale{In general, I think we should give up from concentrating on increasing sequences. In case $V=M$ we are OK but otherwise that might be probleatic}
  % $$\Sigma_a:=\{\sigma \colon \dom(a)\rightarrow\kappa\mid |\sigma|<\kappa\,\wedge\, \text{$\sigma$ is order-preserving}\}$$\ale{In general, I think we should give up from concentrating on increasing sequences. In case $V=M$ we are OK but otherwise that might be probleatic} 
   \item \textbf{(Normality)} $E_{\{\kappa\}}$ is a normal measure.  
   \item \textbf{(Coherency)} Suppose that $a\s b$ are in $ [\lambda \setminus \kappa]^{\leq \kappa}$ and $\kappa\in a$. Then, $\rho_{b, a}$ 
   is a Rudin-Keisler projection from $E_b$ to $E_a$.  Namely, 
   $$A\in E_a\;\Leftrightarrow\; \{\sigma\in [\kappa]^{<\kappa}\mid \rho_{b,a}(\sigma)\in A\}\in E_b.$$

   Here the map $\rho_{b, a}\colon [\kappa]^{<\kappa}\rightarrow [\kappa]^{<\kappa}$ is defined as follows.
%\ale{ Maybe $f_\alpha$ is just $f\restriction\alpha?$}\yair{I think that we need the more complicated definition. $\range f\restriction \alpha$ is not necessarily contained in $\kappa$.}   
Recall that $f:=f^{\otp b} \colon \kappa \to \otp b$ is a surjection, and let $\hat{a} := f^{-1}(\pi_b \image a) \s \kappa$. For every $\alpha < \kappa$, let $f_\alpha$ be $\pi_{f \image \alpha} \circ f \restriction \alpha$. Finally, for $\sigma \in [\kappa]^{<\kappa}$, let $\kappa_\sigma := \min \sigma$ and define $\rho_{b,a}(\sigma) := \range (\pi^{-1}_{\sigma} \circ f_{\kappa_\sigma} \restriction \hat{a} \cap \kappa_\sigma)$.
   %the restriction map $\restriction\colon \Sigma_b\rightarrow X_a$ witnesses the following:
   %Namely, 
  % $$A\in E_a\;\Leftrightarrow\; \{\sigma\in \Sigma_b\mid \rho_{b,a}(\sigma)\in A\}\in E_b.$$
%   \item (Normality) Suppose that $f\colon \Sigma_a\rightarrow\kappa$ is a function such that
%   $$\{\sigma\in \Sigma_a\mid f(\sigma)\in \sup(\ran (\sigma))\}\in E_a.$$
%   Then, there is $b\supseteq a$ such that $\{\sigma\in \Sigma_{b}\mid f(\sigma\restriction a)\in \ran(\sigma)\}\in E_{b}.$%\ale{Maybe we want to get rid of this. This ensures that $k_a([\id]_{E_{a}})=a$, which is necessary to show that $E$ is the extender derived form the extender ultrapower $j_E$ and that every member of $M_E$ is of the form $j_E(f)(a)$. However, (1) and (2) already guarantee that $j_E\colon V\rightarrow M_E$ exits. (3) is impossible to guarantee in Theorem 4.2.}
%   \item (Well-foundeness) The direct limit $M_E$ is well-founded.
\end{enumerate}
\end{definition}
%Any $(\kappa,\lambda)$-extender $E=\langle E_a\mid a\in [\lambda]^{\leq\kappa}\rangle$ yields an elementary embedding $j_E\colon V\rightarrow M_E$ with $\crit(j)=\kappa$. This is obtained by taking the direct limit of the commutative system of embeddings given by the ultrapowers maps $j\colon V\rightarrow M_a\simeq \Ult(V,E_a)$. For details, see \cite[\S26]{Kan}.
%\begin{remark}
 %   About (4) above: For each $a\in [\lambda]^{\leq \kappa}$  we let  $j_a\colon V\rightarrow\Ult(V,E_a)$ be the usual elementary ultrapower buy $E_a$  \footnote{More precisely, $M_E$ is the direct limit of the system $\langle \langle M_a, j_a, k_{a,b}\rangle\mid a, b\in [\lambda]^{\leq\kappa},\, a\s b\rangle$}
%\end{remark}
Given an elementary embedding $j \colon V \to M$ such that ${}^\kappa M \subseteq M$, one can derive a $(\kappa, \lambda)$-extender for every $\lambda \leq j(\kappa)$ as follows: for  $\kappa\in a\in [\lambda\setminus \kappa]^{\leq\kappa}$, %$a\in {}^{\leq \kappa} \lambda$, 
\begin{equation*}\label{extenderequation}
   \tag{$\dagger$} E_a := \{X \subseteq [\kappa]^{<\kappa} \mid a \in j(X)\}.
\end{equation*}
%where $\Upsilon_a:=\{\langle j(\alpha),a(\alpha)\rangle\mid \alpha\in \dom(a)\}$.

%\ale{If we give up on increasingness in the previous definition, we can say here that another possible exteder is that given by $(j\restriction a)^{-1}$. This will concetrate on increasing sequences, but it is not always available.}
\smallskip

By the closure hypothesis over $M$, $E_a$ is a well-defined $\kappa$-complete ultrafilter over $[\kappa]^{<\kappa}$ for every $\kappa\in a\in [\lambda\setminus \kappa]^{\leq\kappa}$. %If $\kappa \in a$ (for example) then $E_a$ is going to be non-principal. 
On the other hand, given an extender, $E$, there is a canonical directed system of transitive models and embeddings $\langle M_a, i_{a,b} \mid a \subseteq b \in[\lambda]^{\leq\kappa}\rangle$ where $M_a \simeq \Ult(V, E_a)$ and $i_{a,b}\colon M_a\rightarrow M_b$ is the elementary embedding obtained from the Rudin-Keisler projection of Clause~(3) above. The direct limit of this system is again well-founded and called \emph{the extender ultrapower}. So, an elementary embedding $j_E \colon V \to M_E$ is obtained.  The original embedding $j$ \emph{factors through $j_E$} in the sense that there is yet another elementary embedding $k\colon M_E\rightarrow M$ such that $j=k\circ j_E.$ Furthermore, the critical point of the factor map $k$ (if exists) %\ale{What do you mean by ''if exists''?} 
is at least $\lambda$. For further details related to extenders, we refer the reader to \cite[\S26]{Kan}.
\begin{comment}
\begin{remark}
When (as before) the elementary embedding is from the universe of sets $V$ to any other transitive class one can derive an extender sitting on order-preserving functions.  Specifically, for each $a\in{}^{\leq \kappa} \lambda$ define
    \begin{equation*}
    \tag{$\dagger\dagger$} E_a := \{X \subseteq \Sigma^+_a \mid (j\restriction a)^{-1} \in j(X)\},
\end{equation*}
where  $(j\restriction a)^{-1}:=\{\langle j(\alpha),\alpha\rangle\mid \alpha\in\range(a)\}$ and $\Sigma^+_a:=\{\sigma\in \Sigma_a\mid \text{$\sigma$ is order-preserving}\}$. 
As we will see in the proof of Theorem~\ref{kappacompactnessandgluing}, extenders defined according to \eqref{extenderequation} are more appropriate  when the elementary embedding is of the form $j\colon M\rightarrow N$ for two sets $M,N$.
\end{remark}
\end{comment}
\begin{remark} %\ale{I'm not sure the current approach is Merimovich's. It seems the measures are defined in the usual way but the indexing is more flexible.}
Our definition of extender differs slightly   from the classical one (see \cite[\S26]{Kan}) due to instrumental purposes. % Instead, we are inspired by Merimovich's view of extenders \cite[\S2]{MerPrikryOnExt}. 
In the usual definition, the measures $E_a$ are indexed by finite sequences $a\in [\lambda]^{<\omega}$ and there is an additional requirement ensuring that the obtained extender-ultrapower is well founded. Our change of indexing is a simple way to ensure that the extender ultrapower $M_E$ will  be closed under $\kappa$-sequences. An equivalent way to obtain that using the standard definition is to require that the \emph{Rudin-Keisler order}\footnote{Given two measures $\mathcal{U}$ and $\mathcal{W}$ over $\kappa$ one says that $\mathcal{U}$ is \emph{Rudin-Keisler below} $\mathcal{W}$, $\mathcal{U}\leq_{\mathrm{RK}}\mathcal{W}$ if there is a map $f\colon \kappa\rightarrow\kappa$ such that for every $A\s \kappa$, $A\in \mathcal{W}$ iff $f^{-1}(A)\in\mathcal{U}$.\label{RudinKeisler}} restricted to the measures $\langle E_a\mid a\in [\lambda]^{<\omega}\rangle$ is $\kappa^{+}$-directed. %namely, for every $\langle a_\alpha\mid \alpha<\kappa\rangle $ of members of $[\lambda]^{<\omega}$ bere is $b$
\end{remark}

%However this shift is harmless in that both definitions are technically equivalent, whenever the extender ultrapower is closed under $\kappa$-sequences. Indeed, given an elementary embedding $j \colon V \to M$ with critical point $\kappa$, $M^{\kappa}\subseteq M$, we can factor $j$ through an extender ultrapower. 

The key concept of this paper is the \emph{Gluing Property}:
\begin{definition}[The Gluing Property]\label{def:gluing-property}
Let $\kappa$ be a measurable cardinal. We say that $\kappa$ has the \emph{$\lambda$-gluing property} if for every $\lambda$-sequence $\langle U_\gamma \mid \gamma < \lambda\rangle$ of $\kappa$-complete ultrafilters over $\kappa$ %there is a $\kappa$-complete extender $E$ 
     there is $j \colon V \to M$ an elementary embedding  with ${}^\kappa M \subseteq M$, $\crit(j) = \kappa$
   and an increasing sequence of ordinals $\langle \eta_\gamma \mid \gamma < \lambda\rangle$ such that
     $U_\gamma = \{X \subseteq \kappa \mid \eta_\gamma \in j(X)\}.$\footnote{The choice of an increasing sequence of $\langle \eta_\gamma\mid \gamma<\lambda\rangle$ implies, e.g.,  that $\lambda\leq j(\kappa)$.}
\end{definition}

%\ale{Explain that every gluing embedding entails an extender in the sense of the above definition.}

There is (yet again) a duality between the gluing property and extenders \emph{gluing measures}. Suppose that the $\lambda$-gluing property holds. For  each sequence of measures $\langle U_\gamma\mid \gamma<\lambda\rangle$ let  $j\colon V\rightarrow M$ and $\langle \eta_\gamma\mid \gamma<\lambda\rangle$ be witnesses.  Let $E$ be the $(\kappa,j(\kappa))$-extender derived from $j$ as in \eqref{extenderequation} above. It is not hard to check that $E$ \emph{glues} the measures $U_\gamma$; namely,  for each $\gamma<\lambda$\label{DualityGluingExtender}
$$E_{\{\kappa, \eta_\gamma\}}\supseteq\{\{\{\alpha,\beta\}\mid \beta \in A,\, \alpha<\kappa\}\mid A\in U_\gamma\}.$$%\ale{This needs to be ammended. Now we are asuming that all of our indices contain $\kappa$}
Similarly, given $\langle U_\gamma\mid \gamma<\lambda\rangle$ and a gluing $(\kappa,\theta)$-extender $E$ (in the above sense) it can be shown that, for each $\gamma<\lambda$, $U_\gamma=\{X\s \kappa\mid \eta_\gamma\in j_E(X)\}.$

\smallskip

We finish this block with an observation about the gluing property: 
\begin{lemma}\label{lemma:omega-gluing-by-a-measure}
The following are equivalent for %a measurable cardinal $\kappa$ and a cardinal
 cardinals $\lambda<\kappa$:
\begin{enumerate}
    \item $\kappa$ has the $\lambda$-gluing property.
    \item For every sequence $\langle U_\alpha \mid \alpha < \lambda\rangle$ of $\kappa$-complete ultrafilters over $\kappa$  there is a $\kappa$-complete ultrafilter $W$ on the increasing sequences of $\kappa^{\lambda}$ such that $U_\alpha\leq_{\mathrm{RK}}{W}$ as witnessed by $e_\alpha\colon \langle \eta_\alpha\mid \alpha<\lambda\rangle \mapsto \eta_\alpha.$ %concentrating on increasing sequences, such that the projection map $e_n$ sending $\eta \in \kappa^{\omega}$ to $\eta(n)$ is a Rudin-Keisler projection from $W$ to $U_n$.
\end{enumerate}
    %A cardinal $\kappa$ has the $\omega$-gluing property if and only if for every $\omega$-sequence of $\kappa$-complete ultrafilters $\langle U_n \mid n < \omega\rangle$ there is a $\kappa$-complete ultrafilter $W$ on  $\kappa^{\omega}$, concentrating on increasing sequences, such that the projection map $e_n$ sending $\eta \in \kappa^{\omega}$ to $\eta(n)$ is a -Keisler projection from $W$ to $U_n$.
\end{lemma}
\begin{proof}
(1)$\Rightarrow$(2): Let $\langle U_\alpha \mid \alpha < \lambda\rangle$ be as above. Using the $\lambda$-gluing property we find and increasing $\langle \eta_\alpha\mid \alpha<\lambda\rangle$ and an elementary embedding $j\colon V\rightarrow M$ with $\crit(j)=\kappa$ and ${}^\kappa M \subseteq M$, such that  $U_\alpha=\{X\s \kappa\mid \eta_\alpha\in j(X)\}.$ 

Let $\mathcal{I}\s \kappa^\lambda$ be the collection of increasing sequences. Define $W$ as follows:
$$X\in W\;\Longleftrightarrow\; X\s \mathcal{I}\;\wedge\; \langle \eta_\alpha\mid \alpha<\lambda\rangle\in j(X).$$
One can show that $W$ is $\kappa$-complete   and that $U_\alpha=\{e_\alpha``A\mid A\in W\}.$

(2)$\Rightarrow$(1): Let $\langle U_\alpha \mid \alpha < \lambda\rangle$ and $W$ be its corresponding witness for (2). Let $j\colon V\rightarrow M$ be the ultrapower embedding by $W$. It follows from $\kappa$-completeness of $W$ that $\crit(j)=\kappa$ and ${}^\kappa M \subseteq M$. Let's look now at $[\id]_W$. This is, by definition, an increasing $\lambda$-sequence of ordinals ${<}j(\kappa)$. For each $\alpha<\lambda$ let $\eta_\alpha$ be the $\alpha^{\mathrm{th}}$-member of $[\id]_W$.  We would like to show that $$U_\alpha=\{X\s \kappa\mid \eta_\alpha\in j(X)\}.$$
Let $X\s \kappa$. Then, $X\in U_\alpha$ if and only if $e_{\alpha}^{-1}X\in W$, which is equivalent to saying $[\id]_W\in j(e_{\alpha}^{-1}X)\Leftrightarrow j(e_\alpha)[\id]_W\in j(X)\Leftrightarrow \eta_\alpha\in j(X),$
as needed.
   % The backwards direction is clear, as $W$ itself can be represented as an extender. For the forwards direction, let $E$ be a $\kappa$-complete extender such that there is a sequence $\langle \eta_n \mid n < \omega\rangle$, $U_n = \{X \subseteq \kappa^{\{n\}} \mid \{\langle n, \eta_n\rangle\} \in j_E(X)\}$. 
% Let $W := \{X \subseteq \kappa^\omega \mid \langle \eta_n \mid n < \omega\rangle \in j_E(X)\}$. Then, $W$ witnesses the validity of the lemma.
\end{proof}
Thus, in the case of the $\omega$-gluing, the gluing extender is actually a $\kappa$-complete measure over increasing sequences of $\kappa^\omega$. This observation will be used in \S\ref{section:consistent-omega-gluing}, when we prove Theorem~\ref{omegagluingnonoptimal}.

Our definition of the gluing property requires the seed, $\eta_\alpha$, to be increasing. This requirement is quite natural with respect both to the upper bounds for the consistency strength as well as the lower bounds.  

One can define the gluing property without requiring the seeds to be increasing. There are at least two possible alternative definitions. 

\begin{definition}
Let $\kappa, \lambda$ be cardinals. We say that $\kappa$ has the \emph{weak-$\lambda$-gluing property} if for every sequence of measures on $\kappa$, $\langle U_\alpha \mid \alpha < \lambda\rangle$ there is an extender $E$ and a sequence of distinct ordinals (non-necessarily increasing) $\langle \eta_\alpha \mid \alpha < \lambda\rangle$ such that $U_\alpha = \{X \subseteq \kappa \mid \eta_\alpha \in j(X)\}$. 
\end{definition}
We say that the Rudin-Keisler order is $\lambda$-directed at $\kappa$, if for every $\lambda$-sequence of measures $\langle U_\alpha \mid \alpha < \lambda\rangle$ there is a  $(\kappa,\lambda)$-extender $E$ which is Rudin-Keisler above $U_\alpha$ for all $\alpha$.\footnote{We would like to thank the anonymous referee for posing the question on the relations between those different definitions} Clearly, the $\lambda$-gluing property implies the weak $\lambda$-gluing property, which in turn implies the $\lambda$-directness of the Rudin-Keisler order. For $\omega$-gluing the obtained principles are the same: 
\begin{lemma}
If for every countable sequence of $\kappa$-complete measures on $\kappa$ there is a $\kappa$-complete measure on $\kappa$ which is Rudin-Keisler above them, then $\kappa$ has the $\omega$-gluing property.
\end{lemma}
\begin{proof}
Let $\langle U_n \mid n < \omega\rangle$ be a sequence of measures on $\kappa$. For each $n<\omega$ define $W_n:= U_0 \otimes U_1 \cdots \otimes U_n$, a measure on $\kappa^n$. Let $f_n\colon \kappa^n \to \kappa$ be a bijection and let $\tilde{W}_n$ be the ultrafilter on $\kappa$ obtained by stipulating  $$X \in \tilde{W}_n:\Leftrightarrow f_n^{-1}(X) \in W_n.$$

Let $W$ be a measure which is Rudin-Keisler above $\tilde{W}_n$ for all $n<\omega$ and let $j\colon V \to M$ be the ultrapower by $W$.  So, in $M$, we can obtain a seed $s_n$ for $\tilde{W}_n$, for each $n$. By re-packing them using $j(f_n)^{-1}$, we obtain an $\omega$-sequence of finite sequences of ordinals. We wish to filter them into a single increasing sequence of seeds. 

Let us construct inductively infinite sets $A_n \subseteq \omega$ such that $A_{n+1}\subseteq A_n$ and for all $k < \ell \in A_n$, the $n$-th member of $j(f_k)^{-1}(s_k)(n) \leq j(f_{\ell})^{-1}(s_k)(n)$. Indeed, by the infinite Ramsey theorem on $\omega$ there is an infinite subset on which this sequence is either weakly increasing or decreasing. As there is no infinite decreasing sequence of ordinals, the later possibility is ruled out.

Finally, let $\{k_n \mid n < \omega\}$ be such that $k_n \in A_n$. Then, letting $\eta_n = j(f_{n+1})^{-1}(s_{k_n})(n)$ we obtain an increasing sequence of seeds. 
\end{proof}
\subsection{Prikry-type forcings}\label{sectionPrikrytype}
%Prikry-type forcings are a major technique in singular-cardinal combinatorics \cite{MagSing, MagSingII,MagSheGroups}. Perhaps a less known facet is that they have as well been instrumental in the study of regular cardinals. Two relevant examples %for the purposes of this paper are:  %The iterations of Prikry-type forcings were introduced and exploited by 
%Magidor's use of Prikry-type iterations to determine the mutual relationship between strongly compact and measurable cardinals  \cite{MagSuper};\footnote{It must be said that \cite{MagSuper} is the original source for such iterations.} Gitik's version of Magidor's iterations  and their applications to the study of  the %saturation of the 
%non-stationary ideal \cite{GitikNonStaI}.\smallskip

In \S\ref{section:consistent-omega-gluing}  of this paper we shall use Gitik's iteration of Prikry-type forcings to prove the consistency of the $\omega$-gluing property from the existence of a strong cardinal. % (see Definition~\ref{StrongCardinal}). 
Later, in \S\ref{section; improving}, we will employ a modification of this technique aiming to compute the exact consistency-strength of the $\omega$-gluing property. Such modification is borrowed from Ben-Neria and Unger's paper \cite{BenUng}  and it is a Prikry-analogue of the non-stationary-supported iteration considered by Friedman and Magidor in \cite{FriedmanMagidor}. The main part of the proof is the analysis of $\kappa$-complete ultrafilters in such a generic extension. This is a reminiscent of Gitik and Kaplan's work in \cite{GitikKaplan-nonstationary2022}. 

\smallskip

%Let us now pass to introducing the pertinent Prikry-type iterations. 
Let $P$ be a set and $\leq$, $\leq^*$ denote partial orders %\footnote{a.k.a.\ posets.} 
over it.
In this paper we endorse the following view of Prikry-type poset due to Gitik \cite[\S6]{Gitik-handbook}:
\begin{definition}\label{GitikPrikrydef}
    A triple $\langle P,\leq,\leq^*\rangle$ is called a \emph{Prikry-type forcing} if %both $\langle \mathbb{P},\leq\rangle$ and $\langle \mathbb{P},\leq^*\rangle$ are partial orders
    \begin{enumerate}
        %\item both $\leq$ and $\leq^*$ are partial orders over $\mathbb{P}$;
        \item  $\leq^*\subseteq \leq$,
        \item $\leq^*$ has the \emph{Prikry property}; to wit, for every $p\in {P}$ and a statement $\varphi$ in the language of forcing of $\langle {P},\leq\rangle$ there is  $q\leq^* p$ such that $$\text{$q\forces_{\mathbb{P}}\varphi$ or $q\forces_{\mathbb{P}}\neg\varphi.$}$$
    \end{enumerate}
\end{definition}
We will conveniently adopt the more compact notation $\mathbb{P}$ in lieu of $\langle P,\leq\rangle$ and identify the universe of the poset $P$ with $\mathbb{P}$ itself. We will also tend to say that $\mathbb{P}$ is of Prikry-type when the order $\leq^*$ is clear from the context.

Note that Gitik's definition is flexible enough to encompass any poset - just take $\leq^*:=\leq$. Of course, genuine Prikry-type forcings satisfy $\leq^*\subsetneq \leq$, and the direct order $\leq^*$ will have better closure properties. %Since we do not want to elaborate further  we refer the reader to Gitik's handbook chapter \cite{Gitik-handbook} where the reader can find several  examples.

%\smallskip

%As we mention some paragraphs above we shall be interested in Gitik-styled iterations:
\begin{definition}[Gitik's iteration, \cite{GitikNonStaI}]\label{Gitikiteration}
Let $\varrho$ be an ordinal. We define an iteration $\langle\mathbb{P}_\alpha, \dot{\mathbb{Q}}_\beta\mid \beta<\alpha\leq \varrho\rangle$ recursively as follows:  For each $\alpha\leq \varrho$ members of $\mathbb{P}_\alpha$ are sequences $p=\langle \dot{p}_\beta\mid \beta\in \supp(p)\rangle$ such that:
\begin{enumerate}[label=(\alph*)]
    \item $\supp(p)\s \alpha$ has \emph{Easton support}: namely, for every inaccessible cardinal $\beta\leq \alpha$, $|\supp(p)\cap \beta|<\beta$ provided $|\mathbb{P}_\gamma|<\beta$ for all $\gamma<\beta$.
    \item for every $\beta\in \supp(p)$, $p\restriction\beta:=\langle \dot{p}_\beta\mid \beta\in \supp(p)\cap\beta\rangle\in \mathbb{P}_\beta$      $$p\restriction\beta\forces_{\mathbb{P}_\beta}``\dot{p}_\beta\in \dot{\mathbb{Q}}_\beta$$ and, moreover, $$\one\Vdash_{\mathbb{P}_\beta} |\dot{\mathbb{Q}}_\beta|\leq 2^\beta\,\wedge\,\text{$\langle \dot{\mathbb{Q}}_\beta,\leq_\beta, \leq_\beta^*\rangle$ is Prikry-type and $\leq^*$ is $\beta$-closed''}.\footnote{Recall that $\beta$-closure stands for thee following property: For every $\leq^*_\beta$-decreasing sequence $\langle r_\gamma\mid \gamma<\bar{\gamma}\rangle$ of members of $\mathbb{Q}_\beta$ (with $\bar{\gamma}<\beta$) there is $r\in\mathbb{Q}_\beta$ such that $r\leq^*_\beta r_\gamma.$ }$$
   % $$\text{and}\; p\restriction\beta\forces_{\mathbb{P}_\beta}`` |\dot{\mathbb{Q}}_\beta|\leq 2^\beta\,\wedge\,\text{$\langle \dot{\mathbb{Q}}_\beta,\leq^*\rangle$ is $\beta$-closed''}.$$
\end{enumerate}
Let $p=\langle \dot{p}_\beta\mid \beta\in \supp(p)\rangle$ and $q=\langle \dot{q}_\beta\mid \beta\in \supp(q)\rangle$ be elements of $\mathbb{P}_\alpha$.

We write $p\leq_\alpha q$  if and only if the following hold:
\begin{enumerate}
    \item $\supp(p)\supseteq \supp(q)$,
    \item $p\restriction\beta\forces_{\mathbb{P}_\beta}\dot{p}_\beta\leq_\beta \dot{q}_\beta$, for every $\beta\in \supp(q)$;
    \item there is $b\subseteq \supp(q)$, finite, such that for all $\beta\in \supp(q)\setminus b$
    $$p\restriction\beta\forces_{\mathbb{P}_\beta} \dot{p}_\beta\leq^*_\beta \dot{q}_\beta.$$
\end{enumerate}
The Prikry ordering $\leq^*_\alpha$ is defined by the case in which $b$ above is empty. %say that \emph{$p$ is a pure/direct extension of $q$} and write $p\leq^*_\alpha q$
\end{definition}
\begin{remark}
About (a): While members of $p\in \mathbb{P}_\varrho$ are names for the various posets involved note that $\supp(p)$ is assumed to be an actual set in the ground model. About (b): The assumption that $|\dot{\mathbb{Q}}_\beta|\leq 2^\beta$ and $\langle \dot{\mathbb{Q}}_\beta,\leq^*_\beta\rangle$ is forced
to be $\beta$-closed is a technical requirement that allows to carry out \emph{diagonalizations}, such as the Prikry Property or its variants. For a concrete example, see  Lemma~\ref{lemma:representing-ultrafilters} in p.\pageref{lemma:representing-ultrafilters}. Finally, observe that Clause~(3) above does not restrict the possible values for $p$ outside the support of $q$, $\supp(q)$. This is in contrast to the \emph{Magidor's iterations} considered in \cite{MagSuper}. 
\end{remark}
One defines  non-stationary-supported iterations of Prikry-type forcings in a similar fashion: just replace Clause~(a) of Definition~\ref{Gitikiteration} by $``\supp(p)\cap \beta$ is \emph{non-stationary} in $\beta$ for all inaccessible cardinals $\beta\leq \alpha$''.%\footnote{Recall that a set $A\s \beta$ is non-stationary if there is a club set $C\s\beta$ disjoint from it.} 

\smallskip

A fundamental theorem about these iterations is that $\mathbb{P}_\varrho$ is of  Prikry-type.  The result for Easton-supported iteration is due to Gitik \cite[\S1]{GitikNonStaI} while the one concerning non-stationary-supported iterations is due to Ben-Neria and Unger \cite[\S2]{BenUng}. Gitik's iterations have the advantage that if $\varrho$ is sufficiently large (say Mahlo) then the iteration $\mathbb{P}_\varrho$ is $\varrho$-cc.\label{GitikIterationsccc} While this is not the case for non-stationary-supported iterations they still keep some of the pleasant properties of $\varrho$-ccness; e.g., they preserve stationary subsets of $\varrho$ \cite[\S3]{BenUng}. The real deal with non-stationary-supported iterations is that they
permit a fine control upon the measures of the generic extension $V^{\mathbb{P}_\varrho}.$ This feature will be exploited through \S\ref{section; improving}; particularly, in \S\S\ref{SectionInterlude} and \ref{Section;proof of coding}.  %A case of special interest in applications is when the degree of closure of each $\langle \dot{\mathbb{Q}}_\alpha,\leq^*_\alpha\rangle$ is much larger than the cardinality of $\mathbb{P}_\alpha$. For instance, in {BenUng}
%\begin{theorem}[\cite{GitikNonStaI,BenUng}]\label{IterationsAreOfPrikrytype}
%$\langle\mathbb{P}_\varrho,\leq^*_\varrho\rangle$ is a Prikry-type forcing.\footnote{The result for Easton-supported iteration is due to Gitik while the other for non-stationary-supported iterations is due to Ben-Neria and Unger.}
%\end{theorem}

\smallskip

We close this section presenting a  variation of the usual \emph{Tree Prikry-type forcing} from \cite[\S1.2]{Gitik-handbook} using $\omega$-many measures. Iterations of this poset are considered in \S\ref{section:consistent-omega-gluing} and \S\ref{section; improving} when proving the consistency of $\omega$-gluing.% property. %More precisely, we will be considering Gitik and non-stationary-supported iterations of it.

%Hereafter that $\kappa$ is a measurable cardinal.

\begin{definition}\label{Utree}
  Let ${\mathcal{U}}:=\langle U_n\mid n<\omega\rangle$ be a sequence of measures over $\kappa$. A set $T$ consisting of finite increasing sequences in $\kappa$ is called a \emph{$\mathcal{U}$-tree} if:
  \begin{enumerate}
      \item $T$ is a \emph{tree}; namely, $\{t\in T\mid t\sq s\}$ is well-orderable for every $s\in T$.
      \item $T$ is $\mathcal{U}$-fat; namely, for every $s\in T$, $\{\alpha<\kappa\mid s^\smallfrown \langle \alpha\rangle\in T\}\in {U}_{|s|}$.
  \end{enumerate}
  The order $\sq$ above stands for usual end-extension (i.e., $t\sq s$ if $t=s\restriction|t|.$)

  A sequence $s\in T$ is called the \emph{stem} of $T$ if it is the largest $t\in T$ such that $s\sq u$ or $u\sq s$ for all $u\in T.$ For each $n<\omega$, the $n^{\mathrm{th}}$-level of $T$ is 
  $$\mathrm{Lev}_n(T):=\{\alpha<\kappa\mid \exists t\in T\, (|t|=n-1\,\wedge\, t^\smallfrown\langle\alpha\rangle\in T)\}.$$
\end{definition}
%\yair{Why do you call that "diagonal"?}
%\ale{You are right. There is just one cardinal!}
\begin{definition}[$\mathcal{U}$-Tree Prikry]\label{UtreePrikry}
    Assume that ${\mathcal{U}}:=\langle U_n\mid n<\omega\rangle$ is a sequence of measures over a cardinal $\kappa$. The \emph{${\mathcal{U}}$-Tree Prikry forcing}  is the poset $\mathbb{T}(\mathcal{U})$ whose conditions are pairs $p=\langle s^p,T^p\rangle$ where $T^p$ is a $\mathcal{U}$-tree with \emph{stem} $s^p.$ Given two conditions $p,q\in \mathbb{T}(\mathcal{U})$ we write $p\leq q$ iff $T^p\s T^q$. %\footnote{Note that $p\leq q$ automatically implies $s^q\sq s^p.$} 
    In addition, if $s^p=s^q$ then we  write $p\leq^* q.$
\end{definition}
Given a condition $p=\langle s^p,T^p\rangle\in \mathbb{T}(\mathcal{U})$ we will conveniently suppress the superindex $p$ when the context allows it. The poset $\mathbb{T}(\mathcal{U})$ is a standard variation of the usual Tree Prikry forcing. Thus, it is not a surprise that it shares most of the properties of its classical sibling:
\begin{lemma}[Folklore]\label{TreePrikryproperty}
The following hold:
\begin{enumerate}
    \item $\mathbb{T}(\mathcal{U})$ is a Prikry-type forcing.
    \item $\langle \mathbb{T}(\mathcal{U}), \leq^*\rangle$ is $\kappa$-closed.
    \item $\mathbb{T}(\mathcal{U})$ is $\kappa^+$-cc and cardinal-preserving.
    \item $\one\forces_{\mathbb{T}(\mathcal{U})}``\cf(\kappa)=\omega$''.
\end{enumerate}
\end{lemma}
We refer to Benhamou's paper \cite{BenTree} for an extensive study of this type of \emph{tree-like} Prikry forcings forcing. 

\subsection{Inner Model theory}\label{sectionInnerModel}
%Some of the results in this paper (particularly, those in \S\ref{section; improving} and \S\ref{section:lower-bound}) involve some inner-model-theoretic concepts. 
We will use this section to garner some  inner-model-theoretic notions and set some relevant notations. % following closely the presentation in \cite{MitChap}.% e.g., to \cite{MitChap,Steel} for more about this topic. %We will use this section to present the reader with the corresponding inner-model notation we shall bear on and recall some relevant concepts
%\smallskip
%The first concept that will be used in this paper is the \emph{Mitchell order}:
\begin{definition}\label{MitchellIncreasing}
    For normal measures ${U}$ and ${W}$ over the same cardinal  one writes $ {U}\lhd {W}$ iff $ {U}\in \Ult(V, {W})$. 
\end{definition}
It is a theorem of Mitchell \cite[\S2]{MitChap} that $\lhd$ is well-founded, hence it makes perfect sense to speak about the rank of a normal measure $ {U}$ in $\lhd$.
\begin{definition}
    The \emph{Mitchell order of a normal measure} $ {U}$ (in symbols, $o( {U})$) is its rank in $\lhd$; namely, $o( {U}):=\sup\{o( {W})+1\mid  {W}\lhd  {U}\}.$

    Similarly, the \emph{Mitchell order of a cardinal $\kappa$} is $$o(\kappa):=\sup\{o( {U})+1\mid \text{$ {U}$ is a normal measure over $\kappa$}\}.$$
\end{definition}
Some few basic observations. A cardinal $\kappa$ is not measurable if and only if $o(\kappa)=0$. Similarly if $o(\kappa)=1$ then $\kappa$ is measurable but any normal measure $ {U}$ over it concentrates on $\{\alpha<\kappa\mid \text{$\alpha$ is not measurable}\}.$ Solovay observed (see \cite[\S2]{MitChap}) that $o(\kappa)\leq (2^\kappa)^+$ and thus, under GCH, $o(\kappa)\leq \kappa^{++}$.

\smallskip
%\yair{Do we really need that? Are we going to use the model $L[\mathcal{U}]$ at some point?}
%\ale{No - I suppresed it. But we need the notion of coherent sequence.}
A related concept, also due to Mitchel, is that of a \emph{coherent sequence of measures}. This will appear later in Theorem~\ref{MainTheorem}.
\begin{definition}
A \emph{coherent sequence of measures} is a function ${\mathcal{U}}$ with:
\begin{enumerate}
    \item $\dom({\mathcal{U}})=\{(\alpha,\beta)\mid \alpha<\mathrm{len}({\mathcal{U}})\;\text{and}\;\beta<o^{{\mathcal{U}}}(\alpha)\}$, where $\mathrm{len}({\mathcal{U}})$ is a cardinal and $o^{{\mathcal{U}}}$ is a map between cardinals $\alpha<\mathrm{len}({\mathcal{U}})$ to ordinals;
    \item if $(\alpha,\beta)\in \dom({\mathcal{U}})$ then $\mathcal{U}(\kappa,\beta)$ is a normal measure on $\kappa$;
    \item  for all $(\alpha,\beta)\in\dom({\mathcal{U}})$,  $j_{\mathcal{U}(\alpha,\beta)}({\mathcal{U}})\restriction\alpha+1={\mathcal{U}}\restriction(\alpha,\beta)$ where
    $${\mathcal{U}}\restriction(\alpha,\beta):=\{\mathcal{U}(\gamma,\bar{\beta})\mid (\gamma< \alpha\,\wedge\,\bar{\beta}<o^{{\mathcal{U}}}(\gamma))\; \vee\; (\gamma=\alpha\,\wedge\,\bar{\beta}<\beta)\},$$
    and 
    $j_{\mathcal{U}(\alpha,\beta)}({\mathcal{U}})\restriction\alpha+1:=j_{\mathcal{U}(\alpha,\beta)}({\mathcal{U}})\restriction(\alpha+1,0)$.%\footnote{This is the coherency assumption which implies, in particular, that $\mathcal{U}(\alpha,\bar{\beta})\lhd\mathcal{U}(\alpha,{\beta})$ for all $\bar{\beta}<\beta$ such that $(\alpha,\bar{\beta}),(\alpha,\beta)\in\dom({\mathcal{U}}).$}
\end{enumerate}
%domain $\{(\alpha,\beta)\mid \alpha\leq \kappa,\,\beta<o^{\vec{\mathcal{U}}}(\alpha)\}$ such that 
\end{definition}
%Clause~(3) above is the coherency condition. This latter implies that $\mathcal{U}(\alpha,\bar{\beta})\lhd\mathcal{U}(\alpha,{\beta})$ for all $\bar{\beta}<\beta$ such that $(\alpha,\bar{\beta}),(\alpha,\beta)\in\dom({\mathcal{U}}).$  Thus, $o(\mathcal{U}(\alpha,{\beta}))\geq \beta$. It is a  theorem of Mitchell (see e.g. \cite[Theorem~2.7]{MitChap})
%that the only measures in $L[\mathcal{U}]$\footnote{Namely, the minimal inner model containing $\mathcal{U}$} are the measures $\mathcal{U}(\alpha,\beta)$ appearing in $\mathcal{U}$. In particular, $o(\mathcal{U}(\alpha,{\beta}))=\beta$ holds in $L[\mathcal{U}].$

A major avenue of research in inner model theory is the construction of \emph{core models}. Roughly speaking, for certain anti-large cardinal hypotheses\footnote{For example, hypotheses such as $``0^{\#}$ does not exists", ``There is no inner model with a measurable cardinal", ``There is no inner model with $\exists \alpha, o(\alpha) = \alpha^{++}$", ``There is no sharp for an inner model with a strong cardinal" or ``There is no class inner model with a Woodin cardinal".}, these are canonical transitive inner models, containing all ordinals, which are not changed under forcing and maximal with respect to the anti-large cardinal hypothesis. Moreover, assuming an anti-large cardinal assumption, the corresponding core model truthfully resembles the universe of sets. 

One of the most important features of the core model $\mathcal{K}$ is its maximality. Namely, every elementary embedding $j\colon \mathcal{K} \to M$, where $M$ is transitive, can be represented using objects in $\mathcal{K}$. The following definition is a presice formulation of this idea for the core model that we use in the paper.

\begin{definition}\label{Iteration}
A system of elementary embeddings $\langle \iota_{\alpha, \beta} \mid \alpha \leq \beta \leq \delta\rangle$ between transitive models is called a \emph{linear iteration} if the following hold:%\ale{Do we want to plug this in here?}
\begin{enumerate}
    \item $\iota_{\alpha,\beta}\colon \mathcal{M}_{\alpha} \to \mathcal{M}_{\beta}$ has $\crit(\iota_{\alpha,\beta})=\mu_\alpha$,%\footnote{}
    \item for all $\alpha \leq \beta \leq \gamma \leq \delta$, $\iota_{\beta,\gamma}\circ \iota_{\alpha,\beta} = \iota_{\alpha,\gamma}$,
    \item for all $\alpha \leq \delta$, $\iota_{\alpha,\alpha} = \id$, and
    \item for every limit ordinal $\gamma \leq \delta$, $$\text{$\langle \mathcal{M}_\gamma, \langle \iota_{\alpha,\gamma} \mid \alpha < \gamma\rangle\rangle$ is the direct limit of $\langle \iota_{\alpha,\beta} \mid \alpha \leq \beta < \gamma\rangle$.}$$
\end{enumerate}
%For each $\alpha\leq\delta$, $\iota_\alpha$ will be a shorthand for $\iota_{0,\alpha}$.

A linear iteration is a \emph{normal iteration, using internal measures} if:
\begin{enumerate}\addtocounter{enumi}{4}
    \item %For all $\alpha$, 
    $\iota_{\alpha, \alpha+1}$ is an ultrapower embedding by a measure $\mathcal{U}_\alpha \in \mathcal{M}_\alpha$ on $\mu_\alpha$,   
    \item $\langle \mu_\alpha \mid \alpha < \delta\rangle$ is a strictly increasing sequence.\end{enumerate}
\end{definition}
\begin{conv}
For each $\alpha\leq \delta$, $\iota_\alpha$ will serve as a shorthand for $\iota_{0,\alpha}$. Often times we will be a bit careless and say that $\iota_\delta$ is an iteration rather than referring to the sequence $\langle \iota_{\alpha, \beta} \mid \alpha \leq \beta \leq \delta\rangle$.
\end{conv}

A remarkable theorem that we shall use is due to Mitchell \cite{MitIter}:
\begin{theorem}\label{CoreModelTheorem}
Assume there is no inner model for $``\exists\alpha\,(o(\alpha)=\alpha^{++})$''. Then there is a unique %coherence sequence $\mathcal{U}$ with an associated 
inner model $\mathcal{K}$ with the following properties:
\begin{enumerate}
    \item\label{Maximality}  If $\mathcal{U}$ is a $\mathcal{K}$-normal and $\kappa$-complete $\mathcal{K}$-ultrafilter over $\kappa$ then $\mathcal{U}\in \mathcal{K}$.\footnote{$\mathcal{K}$-normal means that $\mathcal{U}$ is normal relative to every function $f\colon \kappa\rightarrow \kappa$ in $\mathcal{K}$.}%\yair{Where does this formulation come from?}
   % \ale{You mean how does this connect to our definition of normality? If so, you are right, we need to formulate $\mathcal{K}$-normality in terms of functions $f\in\mathcal{K}$. This is a reminiscence from when we talked about diagonal intersections.}
    \item\label{coremodelIterations} Any non-trivial elementary embedding $j\colon \mathcal{K}\rightarrow M$ into a transitive class $M$ is a normal linear iteration using normal measures in $\mathcal{K}$.
    \item\label{GenericAbsoluteness}  Let $\mathbb{P}$ be a set-sized forcing and  $G\s \mathbb{P}$ be generic. Then $\mathcal{K}^{V[G]}=\mathcal{K}$.
    \item\label{weakcovering} If $\delta$ is a singular strong limit cardinal then $\delta^+=(\delta^+)^{\mathcal{K}}$.
\end{enumerate}
\end{theorem}
The above is the so-called \emph{core model up $o(\kappa)=\kappa^{++}$} but in this paper we shall refer to it simply as the \emph{core model}. Aside from the above-displayed properties, $\mathcal{K}$ has some other gentle combinatorial features - for instance, it satisfies $\GCH$.\label{GCHinK} For more see Mitchell's handbook chapter \cite{Mithand2}. We will not need the precise definition of the core model in this paper. 

In \S\ref{section; improving} and \S\ref{section:lower-bound}  we shall exploit (1)--(4) to prove the equiconsitency of the $\omega$-gluing property modulo a measurable cardinal with $o(\kappa)=\omega_1$. %We will timely refer to Mitchell's theorem when this are used.

\section{Gitik's universal extender-based Prikry forcing}\label{section:universal}
%An uncountable cardinal $\kappa$ is said to be \emph{$\kappa$-compact} if every $\kappa$-complete filter on $\kappa$ can be extended to a $\kappa$-complete ultrafilter.
In this section we would like to provide a full-detailed account of the following striking theorem of Gitik:
\begin{theorem}[Gitik]\label{GitiksTheorem}
	Assume that $\kappa$ is a $\kappa$-compact cardinal. Then, there is a Prikry-type poset $\mathbb{P}$ that does not add bounded subsets of $\kappa$ and forces the following: for every poset $\mathbb{Q}\in V$ with $|\mathbb{Q}|=\kappa$ and  $V\models ``\text{$\mathbb{Q} $ is $\kappa$-distributive}$''   there is a $\mathbb{Q}$-generic filter over $V$.
\end{theorem}
Gitik's result provides a substantial improvement of other  \emph{universality theorems} concerning Prikry-type forcings. For instance, it was well known that if $\kappa$ is ${2}^{\kappa}$-supercompact then the usual \emph{Supercompact Prikry forcing} with respect  a measure on $\mathcal{P}_\kappa(2^\kappa)$ projects onto any $\kappa$-distributive forcing of size $\kappa$ (see \cite[\S6.4]{Gitik-handbook}). Gitik improved that to a $\kappa$-compact cardinal, while the exact consistency strength is still unclear. More recent developments in this vein have been obtained by Benhamou, Hayut and Gitik in \cite{BHG}.

In this paper, however, our interest in Theorem~\ref{GitiksTheorem} relies on the fact that it hints the gluing property. Succinctly speaking the idea is the following.

Firstly, it suffices to establish the theorem for  $\kappa$-distributive forcing posets (in $V$) whose underlying set is  $\kappa$. So, let $\langle \mathbb{Q}_\alpha\mid \kappa\leq \alpha<2^\kappa \rangle$ be an injective enumeration of all such forcings and  $\langle\mathcal{F}_\alpha\mid \kappa\leq \alpha<2^\kappa\rangle$ be the corresponding filters generated by the dense open sets. All of these are $\kappa$-complete filters on $\kappa$. Hence, by $\kappa$-compactness of $\kappa$,  each $\mathcal{F}_\alpha$ can be extended to a $\kappa$-complete ultrafilter $U_\alpha$. Since $\kappa$-compact cardinals have the $2^\kappa$-gluing property (see Theorem~\ref{kappacompactnessandgluing}) we can glue all these measures into a $(\kappa,\theta)$-extender $E$.\footnote{Recall the equivalence Gluing property/Gluing Extender discussed in page~\pageref{DualityGluingExtender}.} Afterwards we will define an extender-based forcing (\emph{Gitik's Universal Extender-based Prikry forcing}) that will simultaneously project onto every poset $\mathbb{Q}_\alpha.$ %That way one gets the result of Theorem~\ref{GitiksTheorem}.

\smallskip

We wish to provide a full-detailed exposition of Gitik's theorem, which first  appeared with fewer details in \cite{Gitikcompact}. For this we will use Merimovich's framework for extender-based forcings. The core concept in Merimovich setup is the the notion of object (see e.g.,  \cite{MerExtender,MerPrikryOnExt, MerSuper}).

 %put $\mathcal{I}:=[2^\kappa\setminus \kappa]^{\leq \kappa}$ and  
Appealing to the the $2^\kappa$-gluing property of $\kappa$ we obtain an elementary embedding $j\colon V\rightarrow M$, with $\crit(j)=\kappa$, ${}^\kappa M\s M$ and an increasing sequence of ordinals $\langle \gamma_\alpha\mid \alpha\in [\kappa, 2^\kappa)\rangle$ such that $U_\alpha=\{X\s\kappa\mid \gamma_\alpha\in j(X)\}$.
%$(\kappa,\theta)$-extender  $E=\langle E_a\mid a\in {}^{\leq\kappa}\theta\rangle$ and an increasing sequence of ordinals $\langle \gamma_\alpha\mid \alpha\in [\kappa, 2^\kappa)\rangle$ such that, for each $a\in {}^{\leq\kappa}\theta$, $E_a$ concentrates on $$\Sigma^+_a:=\{\nu\colon \dom(a)\rightarrow\kappa\mid \text{$\nu$ is order-preserving and $|\nu|<\kappa$}\}$$
%and $\{e_{\gamma_\alpha}``X\mid X\in E_{\{\langle 0,\gamma_\alpha\rangle \}}\}=U_\alpha$, where $e_{\gamma_\alpha}$ is the  map $e_{\gamma_\alpha}(\nu):=\nu(\gamma_\alpha).$

%for every $\alpha\in [\kappa, 2^\kappa)$,  $\mathcal{F}_\alpha\s E_\alpha$, where  $E_\alpha:=\pi_\alpha``E_{\{\alpha\}}$.\footnote{Here $\pi_\alpha$ stands for the natural projection between $X_{\{\alpha\}}$ and $\kappa$:  namely, $\sigma\mapsto \sigma(\alpha)$.} \ale{I need to integrate this in a better way.}

	%Formally speaking, the measure $E_{\{\alpha\}}$ does not concentrate on ordinals, but rather on maps in $X_{\{\alpha\}}$. Nonetheless, both  $E_{\{\alpha\}}$ and $E_\alpha$ are Rudin-Keisler equivalent, so we may freely assume that the latter is the  $\alpha^{\mathrm{th}}$-measure of $E$. 

 \begin{notation}
     For each $\alpha\in [\kappa,2^\kappa)$ the seed of  $U_\alpha$ (i.e., $[\id]_{U_\alpha}$) will be denoted by $\sigma_\alpha$. Also, we set $\mathcal{I}=\{\gamma_\alpha\mid \alpha\in [\kappa,2^\kappa)\}.$%; namely,  $\sigma_\alpha:=[\id]_{U_\alpha}$.
 \end{notation}
%In the current setting we are just interested in adding Prikry sequences with respect to the measures in $E$ indexed by members of $\mathcal{I}$ - the rest will be simply ignored. This motivates the following definition: %The rest of coordinates will be just ignored as they do not 

%For each $\kappa\leq \alpha<2^\kappa$, denote by .%\footnote{Recall that $\sigma_\alpha$ is formally not an ordinal; but rather a map from $\{\alpha\}$ to $\kappa$.} %such that the seed of $E_{\alpha}$ is the map $s_{\alpha}\colon \{\alpha\}\mapsto \{\sigma_\alpha\}$.

%If we now define the extender based Prikry forcing with respect to $E$ we will have that, for each $\kappa\leq \alpha<2^\kappa$, the $\alpha^{\mathrm{th}}$-Prikry sequence introduced by this poset  will be eventually contained in all the members of $\mathcal{F}_\alpha$; that is, in all dense open subsets of $\mathbb{Q}_\alpha$. However, this is not enough to generate a $\mathbb{Q}_\alpha$-generic filter  over $V$, for it may perfectly happen that some of the members of the $\alpha^{\mathrm{th}}$-Prikry sequence are not $\leq_{\mathbb{Q}_\alpha}$-compatible.  To remedy this we need to tweak a bit the definition of the extender based Prikry forcing to ensure that this will never happen. 

%Following Gitik \cite{Gitikcompact}, we shall define an extender based Prikry forcing that is universal for all $\kappa$-distributive forcings of size $\kappa$.

\begin{definition}\label{Cohenpart}
Let $\mathbb{P}^*$ be the poset consisting of maps $f\colon \dom(f)\rightarrow [\kappa]^{<\omega}$, where $\dom(f)\in [\mathcal{I}]^{\leq \kappa}$, $\gamma_\kappa\in\dom(f)$ and for each $\gamma_\alpha\in \dom(f)$, $$\text{$f(\gamma_\alpha)$ is both $<$-increasing and $\leq_{\mathbb{Q}_\alpha}$-decreasing.}$$
	Given $f, g\in \mathbb{P}^*$ we write $f\leq g$ if $f\supseteq g.$
	
	For each $f\in\mathbb P^*$ and $\gamma\in\dom(f)$ denote $\mu^f_\gamma:=\max(f(\gamma))$. When $f$ is clear  we shall suppress the superscript $f$ and just write $\mu_\gamma.$
	\end{definition}
 \begin{remark}
     It is clear that $\mathbb P^*$ is a $\kappa^+$-directed-closed forcing.	
 Also, for each $\beta<\kappa$ there exists $\gamma\in \mathbb{Q}_\alpha$  such that $\gamma$ extends $\beta$ both in $\leq$ and $\leq_{\mathbb{Q}_\alpha}.$
 \end{remark}

%The above will constitute the \textit{Cohen part} of the Prikry conditions. Let us now define the \textit{$f$-objects}
	
	\begin{definition}[Objects]
	Fix $f\in\mathbb{P}^*$. A map $\nu\colon \dom(\nu)\rightarrow \kappa$ is called an \emph{$f$-object} if it has the following properties:
	\begin{enumerate}
 \item  $\dom(\nu)\in[\dom(f)]^{<\kappa}$ and $\gamma_\kappa\in\dom(\nu)$;
	    	\item $|\nu|\leq \nu(\gamma_\kappa)$;
		\item for all $\gamma_\alpha\in\dom(\nu)$, $\mu_{\gamma_\alpha}< \nu(\gamma_\alpha)$ and $\nu(\gamma_\alpha)\leq_{\mathbb Q_\alpha} \mu_{\gamma_\alpha}.$
	\end{enumerate}
	The set of $f$-objects is denoted by $\ob(f)$. 
	
	Given $\nu,\eta\in \ob(f)$,  write $\nu\prec\eta$ iff $\dom(\nu)\s \dom(\eta)$ and for every $\gamma_\alpha\in\dom(\nu)$, $\nu(\gamma_\alpha)<\eta(\gamma_\alpha)$ and $\sigma(\gamma_\alpha)\leq_{\mathbb Q_\alpha} \eta(\gamma_\alpha)$.
	%e $$\text{$\nu\prec\eta$ iff $\dom(\nu)\s \dom(\eta)$ and for all $\alpha\in\dom(\nu)$ $\nu(\alpha)<\eta(\alpha)$.} $$
	\end{definition}
	
	\begin{definition}[Measures]
		Given $f\in\mathbb{P}^*$  define $E(f)$ as follows: 
	$$\text{$X\in E(f)\;\;\Longleftrightarrow\;\; X\s \ob(f)\;\wedge\; \sigma_f\in j(X),$}$$
	where $\sigma_f:=\{\langle j(\gamma_\alpha), \sigma_{\eta} \rangle\mid \gamma_\alpha\in \dom(f)\}$ and $\sigma_{\eta}$ denotes the seed of the measure $U_\eta$, where $\eta$ is the unique ordinal such that $\mathbb{Q}_\alpha/\mu_{\gamma_\alpha}=\mathbb{Q}_{\eta}$, where $\mathbb{Q}_\alpha/\mu_{\gamma\alpha}$ is the forcing consisting of all conditions stronger than $\mu_{\gamma_\alpha}$ in $\mathbb{Q}_\alpha$.\footnote{This make sense because $\mathbb{Q}_\alpha/\mu_{\gamma_\alpha}$ is $\kappa$-distributive of size $\kappa$ and our enumeration is injective.}
	\end{definition}

		It is automatic that $E(f)$ is a $\kappa$-complete measure on $\ob(f)$.

	\begin{definition}
	Let $f\in \mathbb P^*$ and $\nu\in \ob(f)$.  Then, $f_{\langle \nu\rangle}$ denotes the function with domain $\dom(f)$ such that for all $\gamma\in\dom(f)$, 
	$$f_{\langle \nu\rangle}(\gamma):=\begin{cases}
	f(\gamma)^\smallfrown \langle \nu(\gamma)\rangle, & \text{if $\gamma\in\dom(\nu)$;}\\
	f(\gamma), & \text{otherwise.}
\end{cases}
$$
In general, if $\vec\nu$ is a finite $\prec$-increasing sequence in $\ob(f)$, the function  $f_{\vec\nu}$ is  recursively defined as $f_{\vec\nu}:=(f_{\vec\nu\restriction |\vec\nu|})_{\langle \vec{\nu}_{|\vec\nu|-1}\rangle}.$

	\end{definition}

	\begin{definition}[$f$-Trees]
	A set $T$ of finite $\prec$-increasing sequences in $\ob(f)$ is called an \emph{$f$-tree} if it is a tree with respect to end-extensions\footnote{I.e., if $\vec{\eta}\in T$ then $\vec{\eta}\restriction n\in T$ for all $n\leq |\vec{\eta}|.$} and %for each $\vec\nu\in T\cup\{\langle \rangle\}$, $\mathrm{Succ}_T(\vec\nu)\in E(f_{\vec\nu})$  where 
	$$\mathrm{Succ_T}(\vec\nu):=\{\eta\in\ob(f)\mid \max_\prec \vec\nu \prec \eta\;\wedge\; \vec{\nu}\,^\smallfrown\langle\eta\rangle\in T\}\in E(f_{\vec\nu}),$$
	for all $\vec\nu\in T\cup\{\langle \rangle\}$.
	
	For an $f$-tree $T$ and $\vec\nu \in T$, denote $T_{\vec\nu}:=\{\vec\eta\in T\mid \vec\nu{}^\smallfrown \vec\eta\in T\}$.
% \yair{Something is strange in this definition of $Succ_T(\vec \nu)$. You should require somewhere that the concatenation of $\vec\nu$ with $\eta$ is in $T$, right? Also, shouldn't $T$ be closed under initial segments? In particular, isn't the empty sequence always a member of $T$?}
% \ale{Thanks, I've fixed the first thing. About the second comment, isn't it implicit in the definition of a tree? I have added a footnote in that respect.}
	\end{definition}
	\begin{remark}
			For every $f$-tree $T$ and every $\vec\nu\in T$, $T_{\vec\nu}$ is again an $f_{\vec\nu}\,$-tree.
	\end{remark}

	Given $f,g\in \mathbb{P}^*$ with $f\leq g$, denote by $\pi_{\dom(f^p), \dom(f^q)}$ the restriction map from  $\ob(f^p)$ to $\ob(f^q)$; namely, the map $\nu\mapsto \nu\restriction \dom(f^q)$.

	\begin{definition}
		Denote by $\mathbb{P}$ the set of all pairs $p=\langle f^p, T^p\rangle$ such that $f^p\in \mathbb P^*$ and $T^p$ is an $f^p$-tree. Given $p,q\in \mathbb P$,  we write $p\leq^* q$ iff $f^p\supseteq f^q$ and $\pi_{\dom(f^p),\dom(f^q)}``T^p\s T^q$.%\marginpar{Say that this simply denotes the restriction map between objects}
	\end{definition}
	
	\begin{lemma}\label{Pullbacktrees}
		Suppose $f,g\in \mathbb P^*$ are such $f\leq g$ and $T$ is a $g$-tree.  Then, $$T^*:=\{\langle \nu_0,\dots, \nu_{n-1}\rangle  \in \ob(f)^{<\omega}\mid \langle \nu_0\restriction \dom(g),\dots, \nu_{n-1}\restriction \dom(g)\rangle\in T\}$$ is an $f$-tree such that $\pi_{\dom(f),\dom (g)}``T^*\s T$.
	\end{lemma}
	\begin{proof}
		Clearly, $T^*$ is a tree under end-extension. Fix, $\vec\nu\in T\cup\{\langle \rangle \}$.
		
		$\br$ If $\vec\nu=\langle \rangle$  then $``\mathrm{Succ}_{T^*}(\langle \rangle)\in E(f)$'' is equivalent to $\sigma_f\in j(T^*)$. Observe that $\sigma_f\restriction j(\dom(g))=\sigma_f\restriction j`` \dom(g)=\sigma_g\in j(T)$, hence  $\sigma_f\in j(T^*)$.
		
		$\br$ Let $\vec\nu=\langle \nu_0,\dots, \nu_{n-1}\rangle$. %For simplicity, let us assume that $\vec\nu= \langle \nu\rangle$. Then
		We need to check that $\max_{\prec}j(\vec{\nu}) \prec \sigma_f$ and $\sigma_f\in j(T^*)$. The latter has been proved before. About  the former, $j(\vec\nu)=\langle j(\nu_0),\dots, j(\nu_{n-1})\rangle$, where  $\dom(j(\nu_{n-1}))=j``\dom(\nu_{n-1})$ and $j(\nu_{n-1})(j(\alpha))=\nu_{n-1}(\alpha)<\kappa\leq \sigma_f(j(\kappa))$. Thus, $j(\nu_{n-1})\prec \sigma_f$.
	\end{proof}

	\begin{definition}[One point extensions]
	Let $p\in P$ and $\langle \nu\rangle \in T^p$. We denote by $p\cat \langle \nu\rangle$ the pair $\langle f^p_{\langle \nu\rangle}, T^p_{\langle \nu\rangle}\rangle$. For every $\vec\nu \in T^p$, $p\cat \vec{\nu}$ is defined recursively in the natural way. Also,  by convetion, $p\cat \langle \rangle := p$. %, where $T_{\langle \nu\rangle}:=\{\vec\eta\in T\mid \nu\prec \min_{\prec} \vec\eta\}$ and $f^p_{\langle \nu\rangle}$ is the function with $\dom(f^p_{\langle \nu\rangle})=\dom(f^p)$ such that for all $\alpha\in \dom(f^p)$,
	
	%$$f^p_{\langle \nu\rangle}(\alpha):=\begin{cases}
	%f^p(\alpha)^\smallfrown \langle \nu(\alpha)\rangle, & \text{if $\nu(\alpha)>\mu^p_\alpha$;}\\
	%f^p(\alpha), & \text{otherwise.}
%\end{cases}
%$$
	\end{definition}
	
	\begin{lemma}
		For each $p\in \mathbb{P}$ and $\vec\nu\in T^p\cup \{\langle\rangle\}$, $p\cat \vec\nu \in \mathbb{P}$.
	\end{lemma}
	\begin{proof}
		Let us prove the lemma by induction on $n$,  the length of $\vec\nu$. 
		
		If $n=0$ the result follows from the very definition of $p\cat \langle \rangle$. So, suppose the above property holds for all $\vec\nu\in T^p$ with $|\vec\nu|=n$.  Choose $\langle \nu_0,\dots, \nu_{n}\rangle \in{}^{n+1}T$ and set $\vec\eta:=\langle \nu_0,\dots, \nu_{n-1}\rangle$. By the induction hypothesis, $q:=\langle f^p_{\vec\eta}, T^p_{\vec\eta}\rangle \in \mathbb{P}$. Also, by definition, $\langle \nu_n\rangle \in T^q$. We thus need to show that $q\cat \langle \nu_n\rangle \in \mathbb{P}$.  
		
		On one hand, it is clear that $T^q_{\langle \nu_n\rangle}$ is an $f^{q}_{\langle \nu_n\rangle}$-tree. On the other hand,  $f^q_{\langle \nu_n\rangle}\in \mathbb{P}^*$:  Indeed,  $T^q$ is a $f^p_{\vec\eta}\,$-tree, hence $\nu_n\in \ob(f^p_{\vec\eta})$, and thus   $$\nu_n(\gamma_\alpha)>\max(f^p_{\vec\eta}(\gamma_\alpha))=\mu^q_{\gamma_\alpha}\;\;\text{and}\;\; \nu_n(\gamma_{\alpha})\leq_{\mathbb{Q}_\alpha}\mu^q_{\gamma_\alpha},$$
		for  $\gamma_\alpha\in \dom(\nu)$. Thereby, $f^q_{\langle \nu_n\rangle}$ witnesses the clauses of  Definition~\ref{Cohenpart}.
	\end{proof}
	
	\begin{definition}[The main poset]  \emph{Gitik's universal Extender based forcing} is the set $\mathbb{P}$ endowed with the following order: for two conditions $p,q\in \mathbb{P}$,  we write $p\leq q$ iff there is $\vec\nu\in T^q\cup\{\langle \rangle\}$ such that $p\leq^* q\cat \vec\nu$. 
		\end{definition}
		
		\begin{notation}
		Let $m<\omega$ and $p,q\in\mathbb{P}$. We write $p\leq^m q$ as a shorthand for $p\leq q$ and $|\vec{\nu}|=m$, where $\vec{\nu}$ is some sequence witnessing $p\leq^* q\cat\vec{\nu}.$
		
		Note that, under this convention, $\leq^0=\leq^*.$
		\end{notation}

		%\begin{remark}
		%	Note that given any $p\in \mathbb P$ it is possible to shrink $T^p$ to ensure that for all $\langle \nu\rangle\in T^p$ and all $\alpha\in \dom(\nu)$ then $\nu(\alpha)>\mu^p_\alpha$. Thus, since the poset formed by impossing this extra requirement is ($\leq^0$-)dense in $\mathbb{P}$ we may safely assume that this is always the case. 
		%\end{remark}

		\begin{lemma}
		$\langle \mathbb{P},\leq^*\rangle$ is $\kappa$-directed-closed.
		\end{lemma}
		\begin{proof}
			Let $D=\{p_\alpha\mid \alpha<\theta\}$ be a $\leq^*$-directed set of conditions in $P$ with $\theta<\kappa$. Set %$d:=\bigcup_{\alpha<\theta} \dom(f^{p_\alpha}})$ and 
			$f:=\bigcup_{\alpha<\theta} f^{p_\alpha}$. Since $D$ is $\leq^*$-directed it is immediate that $f\in P^*$. Put,  $T:=\bigcup_{\alpha<\theta} \pi_{\dom(f), \dom(f^{p_\alpha})}^{-1}T^{p_\alpha}$. Arguing as in Lemma~\ref{Pullbacktrees} one can check that $T$ is a $\dom(f)$-tree. Altogether, $p:=\langle f, T\rangle \in P$ yields a $\leq^*$-lower bound for the set $D$. 
		\end{proof}
		
		\begin{lemma}
		$\mathbb{P}$ has the $\kappa^{++}$-cc.	
		\end{lemma}
		\begin{proof}
			This is as in \cite[Proposition~3.17]{MerPrikryOnExt}.
		\end{proof}

The proof of the \emph{Strong Prikry Property} is essentially the same as that from \cite{MerExtender}. We reproduce it just for completeness and for the reader's benefit. Let us  begin with an useful lemma about \emph{integration} of conditions:
\begin{lemma}
Let $p\in \mathbb{P}$ and $D\colon T^p\rightarrow V$ be a function such that, for each $\vec{\nu}\in T$, $D(\vec\nu)$ is a $\leq^*$-dense open set below $p\cat \vec{\nu}$. Then, there is $p^*\leq^* p$ such that for each $\vec\mu\in T^{p^*}$,  $p^*\cat\vec\mu\in D(\vec\mu\restriction\dom(f^p)).$
\end{lemma}
\begin{proof}
Fix $\chi$  a big enough regular cardinal and let $N\prec H_\chi$ be such that $|N|=\kappa$, $N^{<\kappa}\s N$ and $\{p, \mathbb{P},\mathbb{P}^*, \{D(\vec\nu)\mid \vec\nu\in T^p\}\}\s N$.  Using a properness argument one can show that there is $f^*\leq f^{p}$ that is totally $\langle N,\mathbb{P}^*\rangle$-generic; namely, for every $E\s \mathbb{P}^*$ dense open set,  $E\in N$, there is $h\in N\cap E$ with $f^*\leq h$. For each $\vec\mu\in\ob(\dom(f^*))$  define $D^*(\vec\mu)$ as
$$\{f\leq f^p\mid \exists q\in \mathbb{P}\, (f^q=f_{\vec\mu}\,\wedge\, q\leq^* p\cat \langle \vec\mu\restriction\dom(f^p)\rangle \, \wedge\, q\in D(\vec\mu\restriction\dom(f^p)))\}.$$
Clearly, $D^*(\vec\mu)\in N$  (because $\vec\mu\in N$) and it is an open subset of $\mathbb{P}^*$. In addition, it is not hard to check that $D^*(\vec\mu)$ is dense below $f^p$. Thereby we have that $D^*(\vec\mu)\cup \{f\in\mathbb{P}^*\mid f\perp f^{p}\}$ is a dense subset of $\mathbb{P}^*$ in $N$ and so, by properness, there is $h\in N$ in this set with $f^*\leq h$, Thus, $h\in D^*(\vec\mu)$. 

The above shows that, for each $\vec\mu\in \ob(\dom(f^*))$, there is a map $h_{\vec\mu}$ in $D^*(\vec\mu)\cap N$, which amounts to saying that there is a condition $\langle f_{\vec\mu}, T^{\vec\mu}\rangle$ in the dense open set $ D(\vec\mu\restriction\dom(f^p))$ such that $\langle f_{\vec\mu}, T^{\vec\mu}\rangle\leq^* p\cat \langle \vec\mu\restriction\dom(f^p)\rangle$. 

We would like to integrate all these trees $T^{\vec\mu}$ into a single one $T^*$ for which $\langle f^*, T^*\rangle\leq^* \langle f^p, T^p\rangle$ and, for each $\vec\mu\in T^*$, $\langle f^*, T^*\rangle\cat\vec\mu\leq^* \langle f_{\vec\mu}, T^{\vec\mu}\rangle$. %Once this is achieved we will be done. 

Let $T^*:=\pi_{\dom(f^*),\dom(f^p)}^{-1}T^p\cap \diagonal_{\vec\mu\in\ob(\dom(f^*))} T^{\vec\mu}$, where this latter tree (call it $\diagonal T$) is defined by induction so as to ensure that, for each $\vec\mu\in \diagonal T$, $(\diagonal T)_{\vec\mu}=T^{\vec\mu}$. It should be clear that $T^*$ is the sought tree.
\end{proof}

The next  is the key towards the Strong Prikry property:

\begin{lemma}[Diagonalization]\label{lemma: diagonalization}
There is $p^*\leq^* p$ such that for each $n\geq 1$,
\begin{enumerate}
\item[$(1)_n$] either $``p^*\cat\vec\mu\in D$'' for all $\vec\mu\in T^{p^*}$ with $|\vec\mu|=n$,
\item[$(2)_n$] or $``\forall q\leq^* p^*\cat\vec\mu\, (q\notin D)$'' for all $\vec\mu\in T^{p^*}$ with $|\vec\mu|=n$.
\end{enumerate} 
%$$\text{either $\forall\vec\mu\in T^{p^*}\; (p^*\cat\vec\mu\in D)$ or 
%$\forall\vec\mu\in T^{p^*}\;\forall q\leq^* p\cat\vec\mu\, (q\notin D)$.}$$
\end{lemma}
\begin{proof}
For each $n\geq 1$ define $D_n$ to be the collection of all $p^*\leq^* p$ such that 
\begin{enumerate}
    \item[$(1)_n$] either $\forall \vec\nu\in [T^{p^*}]^n\, p^*\cat\vec\nu\in D$
    \item[$(2)_n$] or $\forall \vec\nu\in [T^{p^*}]^n\,\forall q\leq^* p^*\cat\vec\nu\, q\notin D$.
\end{enumerate}
Our goal is to show these $D_n$'s are $\leq^*$-dense open. For if this is the case we may be able to let $p^*\in \bigcap_{n<\omega} D_n$ and clearly $p^*$ has the required properties. 

\begin{claim}
    For each $n\geq 1$, $D_n$ is open and $\leq^*$-dense  below $p.$
\end{claim}
\begin{proof}[Proof of claim]
Clearly $D_n$ is open so it suffices to check $\leq^*$-density.

Let $q\leq^* p$. We shall find $p^*\leq^* q$ in $D_n$. To streamline the exposition we assume that $q=p$. This is harmless as the forthcoming argument adapts (modulo a switch of notations) to handle conditions $\leq^*$-stronger than $p$.

For each $\vec\nu\in T^p$ put 
$$D^0(\vec\nu):=\{q\leq^* p\cat\vec\nu\mid q\in D\},$$
$$
D^1(\vec\nu):=\{r\leq^* p\cat\vec\nu\mid \forall q\leq^* r\, (q\notin D)\},$$
 $$D(\vec\nu):=D^0(\vec\nu)\cup D^1(\vec\nu).$$
Note that $D(\vec\nu)$ is $\leq^*$-dense open below $p\cat\vec\nu$. 

Applying the previous lemma we find $p^*\leq^* p$ such that, for each $\vec\mu\in T^{p^*}$, $$p^*\cat \vec\mu\in D(\vec\mu\restriction\dom(f^p)).$$ Let us define, by induction on $n<\omega$, a tree $T^{*}\s T^{p^*}$ as follows.

%$\langle f^{p^*}, T^{*}\rangle$ is a condition $\leq^*$-below $p$ witnessing the claim.

Firstly, either the set $$\{\mu\in \ob(\dom(f^{p^*}))\mid p^*\cat\langle\mu\rangle\in D^0(\mu\restriction\dom(f^p))\}$$ is $E(\dom(f^{p^*}))$-large or so is its complement, $$\{\mu\in \ob(\dom(f^*))\mid p^*\cat\langle\mu\rangle\in D^1(\mu\restriction\dom(f^p))\}.$$  Shrink the first level of $T^{p^*}$ to such $E(\dom(f^{p^*}))$-large set and for each $\mu$ in there let $(T^1)_\mu:=T^{p^*}_\mu$. This defines another $E(\dom(f^{p^*}))$-tree $T^1\s T^{p^*}$ which is our first approximation to  $T^*$.

\smallskip

The second approximation to $T^*$  is defined as follows: for each object $\mu\in \mathrm{Lev}_1(T^1)$ look at $$X^0_\mu:=\{ \nu\in \ob(\dom(f^*))\mid \mu\prec\nu\,\wedge\, p^*\cat\langle\mu,\nu\rangle\in D^0(\langle\mu,\nu\rangle\restriction\dom(f^p))\}$$ and at its complement
$$X^1_\mu:=\{ \nu\in \ob(\dom(f^*))\mid \mu\prec \nu\, \wedge\, p^*\cat\langle\mu,\nu\rangle\in D^1(\langle\mu,\nu\rangle\restriction\dom(f^p))\}$$

 Denote by $X_\mu$ the unique of the above sets that is a $E(\dom(f^*))$-large.  Set $\{\mu\in \mathrm{Lev}_1(T^1)\mid X_\mu=X^0_\mu\}$ and  $\{\mu\in \mathrm{Lev}_1(T^1)\mid X_\mu=X^1_\mu\}$. At least one of them is large, call it $X$, and stipulate  the first level of  $T^2$  be  $X$. For each $\mu\in X$,  $\mathrm{Succ}_{T^2}(\langle\mu\rangle):=X^i_\mu$  and for each $\langle \mu,\nu\rangle\in X\times X^i_\mu$, $(T^2)_{\langle \mu,\nu\rangle}:= (T^1)_{{\langle \mu,\nu\rangle}}$. This way we define our second approximation, $T^2$. 

 \smallskip

Proceeding in this vein we generate an inclusion-decreasing sequence of $E(\dom(f^{p^*}))$-trees, $\langle T^n\mid n\geq 1\rangle$. Set $T^*:=\bigcap_{1\leq n<\omega} T^n$ and $p^*:=\langle f^*, T^*\rangle.$

Clearly, $p^*\leq^* p$ and by our construction it belongs to $D_n$.%\ale{Work out the details of this.}
\end{proof}
We are done with the lemma. 
\end{proof}

\begin{lemma}[Strong Prikry property]
For each $p\in \mathbb{P}$ and $D\s \mathbb{P}$ dense open there is a condition $p^*\leq^* p$ and an integer $n<\omega$ such that, for each $m\geq n$, if $q\leq^m p^*$ then $q\in D$.
\end{lemma}
\begin{proof}
Let $p\in\mathbb{P}$ and $D\s \mathbb{P}$ be as above. Invoke Lemma~\ref{lemma: diagonalization} to find a condition $p^*\leq^* p$ to which the above dichotomy applies. By density of $D$ there must be $n<\omega$ for which $(1)_n$ holds; namely, $p^*\cat\vec\mu\in D$ for all $\vec\mu\in [T^{p^*}]^n$. Now, let $q\leq^n p^*$. Since $q\leq^* p^*\cat \vec\mu$ for certain $\vec\mu\in [T^{p^*}]^n$ it follows that $q\in D$ (by openess of $D$).  Similarly, openess of $D$ yields $q\in D$ for all $q\leq^m p^*$ for $m\geq n.$ This concludes the proof of the lemma.
\end{proof}

\begin{cor}[Prikry Property]
	For each $p\in \mathbb{P}$ and every sentence $\varphi$ in the language of forcing there is $p^*\leq^* p$ such that $p^*$ decides $\varphi$.
\end{cor}

Let $G$ a $\mathbb P$-generic filter over $V$. For each $\alpha\in [\kappa, 2^\kappa)$ define 
		$$G_\alpha:=\bigcup\{f^p(\gamma_\alpha)\mid p\in G\,\wedge\, \gamma_\alpha\in \dom(f^p)\}.$$
		%By density arguments it is immediate that $\sup(G_\alpha)=\kappa$ and $\otp(G_\alpha)=\omega$. Also, for all  $\alpha, \beta\in [\kappa,2^\kappa)$ with $\alpha\neq \beta$ then $G_\alpha\neq G^\beta$.
		For each $\alpha\in [\kappa, 2^\kappa)$, %  the following hold:	
			 $\sup(G_\alpha)=\kappa$ and $\otp(G_\alpha)=\omega$. Also,  
			  $G_\alpha\neq G_\beta$, for all $\alpha, \beta\in [\kappa,2^\kappa)$ with $\beta\neq \alpha$. 
		The following lemma shows that,  indeed, $\mathbb{P}$ produces the configuration described in Theorem~\ref{GitiksTheorem}.
		\begin{lemma}[Universality]
		 For each $\alpha\in [\kappa, 2^\kappa)$, $G_\alpha$ is eventually contained in every dense open subset of $\mathbb Q_\alpha$ lying in $V$.
				In particular,  $$\{q\in \mathbb{Q}_\alpha\mid \exists p\in G_\alpha\, (p\leq_{\mathbb{Q}_\alpha} q) \}$$
			defines a  $\mathbb{Q}_\alpha$-generic filter over $V$.% lying in $V$.\marginpar{What sort of Prikry sequences do we 
				\end{lemma}
		\begin{proof}
	%	Items (1) and (2) follow from easy density arguments. 
		Fix a dense open set  $D\s\mathbb Q_\alpha$ and set $$E:=\{p\in\mathbb P\mid \alpha\in \dom(f^p)\,\wedge\, \forall\langle \nu\rangle \in \mathrm{Succ}_{T^p}(\langle \rangle)\; (\nu(\gamma_\alpha)\in D)\}.$$
		We claim that $E$ is dense: Let $p\in \mathbb P$ and assume that $\gamma_\alpha\in \dom(f^p)$. Since $T^p$ in an $f^p$-tree,     $\mathrm{Succ}_{T^p}(\langle \rangle)\in E(f^p)$. Let $\eta\in [\kappa, 2^\kappa)$ be the unique ordinal such that $\mathbb Q_\eta=\mathbb{Q}_\alpha/\mu_{\gamma_\alpha}$. Note that $U_\eta\LE_{\mathrm{RK}} E(f^p)$, as witnessed by the map %$\delta_{p, \alpha}\colon\ob(f^p)^1\rightarrow \kappa$ defined as 
		$e_{\gamma_\alpha}\colon \langle \nu\rangle \mapsto \nu({\gamma_\alpha})$. In particular,  $e_{\gamma_\alpha}``\mathrm{Succ}_{T^p}(\langle \rangle)\in U_\eta$ and thus the set $A:=e_{\gamma_\alpha}``\mathrm{Succ}_{T^p}(\langle \rangle)\cap D$ is $U_\eta$-large. Pulling $A$ back with $e_{\gamma_\alpha}$ one obtains a set $A' \s \mathrm{Succ}_{T^p}(\langle \rangle)$ with $A'\in E(f^p)$ such that for all $\langle \nu\rangle\in A'$, $\nu({\gamma_\alpha})\in D$.  Let $T'$ be the $f^p$-tree with $\mathrm{Succ}_{T'}(\langle \rangle)=A'$ and $T'_{\langle \nu\rangle}=T^p_{\langle \nu\rangle}$ for all $\langle \nu\rangle\in A'$. %Repeating this process by induction on the height of the tree $T^p$ one gets an $f^p$-tree $T'$ such that $T'\s T^p$ and $\mathrm{Succ}_{T'}(\langle \rangle)=A'$. 

  Clearly, $p':=\langle f^p, T '\rangle$ is a condition in $\mathbb{P}$ with $p'\LE^* p$ and $p'\in E$. 	%observe that  the definition of  conditions in $\mathbb{P}$ ensures that the above displayed set is a filter in $\mathbb{Q}_\alpha$. To infer that it is generic, let 
		
		Let $p\in G\cap E$. Since $\{p\cat \langle \nu\rangle \mid \langle \nu\rangle \in T^p\}$ forms a maximal antichain below $p$ there is some $f^p$-object $\nu$ such that $p\cat \langle \nu\rangle \in G$, hence $$\nu({\gamma_\alpha})=f^{p\cat \langle \nu\rangle}({\gamma_\alpha})\in G_\alpha\cap D.$$	
		Say $\nu({\gamma_\alpha})$ is the $n^{\mathrm{th}}$-member of $G_\alpha$. By the very definition of the forcing it turns out that all the members of $G_\alpha$ past stage $n$ are $\leq_{\mathbb Q_\alpha}$-stronger than $\nu({\gamma_\alpha})$ and thus, by openess of $D$, $G_\alpha\s^* D$.		
		%We claim that $E$ is dense in $\mathbb{P}$. 
		%Let $p\in\mathbb{P}$ and,  \marginpar{Finish!}
		\end{proof}
		
		The above concludes the proof of Gitik's theorem (Theorem~\ref{GitiksTheorem}).

\section{(Sub)compactness and the gluing property}\label{section:gluing-from-large-crdinals}

\subsection{The filter extension property and the gluing property}
In this section we analyze the implications between the $\lambda$-filter extension property and the gluing property. In particular, we show that if $\kappa$ is $\kappa$-compact %(i.e., $\kappa$ has the $\kappa$-filter-extension property) 
then $\kappa$ has  the $2^\kappa$-gluing property. As we will demonstrate this implication is optimal, for there are $\kappa$-compact cardinals without the $(2^\kappa)^+$-gluing property. We refer the reader to \S\ref{Preliminaries} were all these notions were properly introduced.%Recall that a cardinal $\kappa$ is said to have the \emph{$\lambda$-filter extension property} (for $\lambda\geq \kappa$) if every $\kappa$-complete filter on $\lambda$ extends to a $\kappa$-complete ultrafilter.

\smallskip

We commence with a strengthening of \cite[Corollary~6]{Hayutpartial}:
\begin{lemma}\label{BetterCharacterization}
Let $\kappa\leq \lambda$ be uncountable cardinals with $\lambda^{<\kappa}=\lambda$. Then the following two are equivalent:
\begin{enumerate}
    \item $\kappa$ has the $\lambda$-filter-extension property;
    \item for every transitive model $M$ with $2^\lambda\s M$, $|M|=2^\lambda$ and ${}^{\kappa}M\s M$ there is  $N$ and an elementary embedding $j\colon M\rightarrow N$ such that:
    \begin{enumerate}
        \item $\crit(j)=\kappa$;
        \item ${}^{\kappa}N\s N$;
        \item $s\in N$,
        $j``2^\lambda\s s$ and $N\models |s|<j(\kappa)$.
    \end{enumerate}
    %with ${}^{\kappa}N\s N$, $s\in N$, together with an elementary embedding $j\colon M\rightarrow N$, such that $\crit(j)=\kappa$, $j``2^\lambda\s s$ and $N\models |s|<j(\kappa)$.
\end{enumerate}
\end{lemma}
\begin{proof}
The  implication $(2)\Rightarrow(1)$ follows from \cite[Corollary~6]{Hayutpartial}. Assume that $\kappa$ has the $\lambda$-filter-extension property and let $M$ be as above. In \cite[Corollary~6]{Hayutpartial} it is shown that there is a transitive model $\mathcal{N}$, $\bar{s}\in \mathcal{N}$ and an elementary embedding $i\colon M\rightarrow \mathcal{N}$ witnessing (a) and (c) above. We shall now slightly modify both $\mathcal{N}$ and $i$ in order to obtain the desired outcome.

Define $\mathcal{X}:=\{i(f)(\kappa,\bar{s})\mid f\in M\}$. By \cite[Lemma~1.4]{HamCanonicalSeeds}, $\mathcal{X}\prec \mathcal{N}$. In particular, $\mathcal{X}$ is well-founded. Let $\pi\colon \mathcal{X}\rightarrow N$ be the Mostowski collapse, and put $j:=\pi\circ i$ and  $s:=\pi(\bar{s})$. % and $N:=\pi``\mathcal{X}$. 
Note that $\pi(\kappa)=\kappa$ and %(because $\kappa\s \mathcal{N}$) and
$$N=\pi``\{i(f)(\kappa,\bar{s})\mid f\in M\}=\{j(f)(\kappa,s)\mid f\in M\}.$$
Clearly, $s\in N$, $j``2^\lambda\s s$ and $N\models |s|<j(\kappa)$. Next, we  check Clause~(b).

Let $\vec{x}=\langle x_\alpha\mid \alpha<\kappa\rangle\s N$. For each $\alpha<\kappa$, $x_\alpha=j(f_\alpha)(\kappa,s)$ for some $f_\alpha\in M$. Observe that establishing $\vec{x}\in N$ amounts to verify that $\langle j(f_\alpha)\mid \alpha<\kappa\rangle\in N$ because in that case $\vec{x}$ will be defined by plugging $(\kappa,s)$ into this sequence, and this is certainly a definable process in $N$.

Let $h\colon \kappa\times M\rightarrow M$ be defined as $(\beta,x)\mapsto \langle f_\alpha(\beta,x)\mid \alpha<\beta\rangle$. Since ${}^{\kappa}M\s M$, the sequence $\langle f_\alpha \mid \alpha < \kappa\rangle \in M$ and therefore $h\in M$. So, \[j(h)(\kappa,s)=\langle j(f_\alpha)(\kappa,s)\mid \alpha<\kappa\rangle \in N.\qedhere\]
\end{proof}
The above is a sort of amalgam between the classical characterization of strong compactness \cite[Theorem 22.17]{Kan} and the corresponding one for  weak compactness due to Hauser \cite[Theorem~1.3]{Hauser}. In the next theorem we show how to use it to establish the $2^\lambda$-gluing property:
\begin{theorem}\label{kappacompactnessandgluing}
Let $\kappa\leq \lambda$ be uncountable cardinals with $\lambda^{<\kappa}=\lambda$. 

If $\kappa$ has the $\lambda$-filter-extension property then it has the $2^\lambda$-gluing property. %In particular, if $\kappa$  has the $\kappa$-filter-extension property has then it has the $2^\kappa$-gluing property
\end{theorem}%\ale{Please, check the new proof and the new definition of extender.}
\begin{proof}
Let $\langle U_\alpha\mid \alpha<2^\lambda\rangle$ be  $\kappa$-complete measures over $\kappa$ and  consider 
%$$ \mathcal{X}:=\langle U_\alpha\mid \alpha<2^\lambda\rangle\cup \{A\mid A\in U_\alpha,\,\alpha<2^\lambda\}\cup \bigcup\{\mathcal{P}(\Sigma^*_x)\mid {x\in[\kappa^+]^{\leq\kappa}}\},$$
$$ \mathcal{X}:=\langle U_\alpha\mid \alpha<2^\lambda\rangle.$$
%where $\Sigma^*_x:=\{\sigma\colon x\rightarrow\kappa\mid |\sigma|<\kappa\}.$

\smallskip

By coding $\mathcal{X}$ appropriately we have that $\mathcal{X}\s 2^\lambda$. Let $M$ be a transitive model as in the statement of Lemma~\ref{BetterCharacterization}, such that $\{\mathcal{X}\} \cup \mathcal{P}(\lambda) \cup 2^\lambda \subseteq M$. Using the $\lambda$-filter-extension property of $\kappa$ and Lemma~\ref{BetterCharacterization} we find a transitive model $N$ with ${}^{\kappa}N\s N$ and $s\in N$, together with an elementary embedding $j\colon M\rightarrow N$ with $\crit(j)=\kappa$, $j``2^\lambda\s s \in N$ and $N\models |s|<j(\kappa)$. Note that
\[\{j(U_\alpha) \mid \alpha < 2^\lambda\} \subseteq j(\mathcal{X})\image s.\]
%and
%$j\image U_\alpha\s s\cap j(U_\alpha)$ for all $\alpha<2^\lambda$. 
By fixing an enumeration of $\mathcal P(\kappa)$ in $M$, we have that $j\image U_\alpha$ is covered by a subset of $j(U_\alpha)$ of cardinality %\ale{I do not quite follow this paragraph...}
$<j(\kappa)$ in $N$: Fix $\varphi\colon \mathcal{P}(\kappa)\rightarrow 2^\kappa$ one of such enumerations. Then, $j(\varphi)(j(X))=j(\varphi(X))\in s\cap j(\varphi)``(j(U_\alpha))$ for all $X\in U_\alpha$ and $\alpha<2^\lambda.$ It follows that, $j``U_\alpha\s j(\varphi)^{-1} ``s \cap j(U_\alpha)\in N$. 

In a slight abuse of notation, let us denote this covering set by $s\cap j(U_\alpha)$. Since $j(U_\alpha)$ is a $j(\kappa)$-complete measure in $N$, $\bigcap (j(U_\alpha)\cap s)\neq \emptyset$. For each index $\alpha<2^\lambda$ pick $\eta_\alpha$ be a witness for this intersection to be non-empty. Note that $\eta_\alpha\in \bigcap j``U_\alpha$ precisely because $j``U_\alpha\s s\cap j(U_\alpha).$

Let us show that we can pick $\eta_\alpha$ to be increasing. Indeed, in $N$, we can apply this procedure by recursion to each measure of the form $j(\mathcal X)(\zeta)$, $\zeta \in s$. Namely, for each $\zeta \in s \cap j(2^{\lambda})$, pick $\tilde{\eta}_{\alpha} \in \bigcap (j(\mathcal{X})(\zeta) \cap s)$ larger than $\sup \{\tilde{\eta}_\beta \mid \beta \in s \cap \zeta\}$. Since $N \models |s| < j(\kappa)$, this is possible. Finally, we let $$\eta_\alpha := \tilde{\eta}_{j(\alpha)}.$$

Next, let us form the corresponding gluing extender. Let $\eta:=\sup_{\alpha<2^\lambda}\eta_\alpha$ and observe that this is below $j(\kappa)$. In effect, recall that the longer sequence $\langle\tilde{\eta}_\beta\mid \beta\in s\rangle$ was defined internally in $N$ and that in this latter model $j(\kappa)$ is a regular cardinal. Thus, $\eta=\sup_{\alpha<2^\lambda}\eta_\alpha\leq \sup_{\beta\in s}\tilde{\eta}_\beta<j(\kappa).$

\smallskip

Define a sequence $E=\langle E_a\mid \kappa\in a\in [\eta\setminus \kappa]^{\leq\kappa}\rangle$ as follows: %\yair{Why is $a$ now a sequence instead of a set?}
$$X\in E_a\;\Longleftrightarrow\; X\s[\kappa]^{<\kappa}\; \wedge\, a\in j(X).$$
%where, as in \S\ref{subsection:extenders}, $\Sigma_a:=\{\sigma\colon \dom(a)\rightarrow \kappa\mid |\sigma|<\kappa\}$ and
%$$\Upsilon_a:=\{\langle j(\alpha),a(\alpha)\rangle\mid \alpha\in\dom(a)\}.$$
\begin{claim}
    $\langle E_a\mid  \kappa\in a\in [\eta\setminus\kappa]^{\leq \kappa}\rangle$ is a $(\kappa,\eta)$-extender.
\end{claim}
\begin{proof}[Proof of claim]
   The above definition is clearly well-posed.
   %First,  $\mathcal{P}(\Sigma_a)\s M$ by our choice of $\mathcal{X}$. % as
  % $$\mathcal{P}^{\leq \otp(a)}(\kappa)\s \mathcal{X}\s 2^\lambda\s M.$$
  % \ale{Need to complete! We also have the problem that $\eta$ needs to be contained in $M$}.\ale{Somehow we would like to change the seed and say that it is $\{\langle j(\alpha), \eta_\alpha\rangle \mid \alpha\in a\}$ for $a\in [2^\lambda]^{\leq\kappa}$. However this does not satisfy the normality requirement of \S2.} 
  %Second, $\Upsilon_a\in N$: This is because $j``\dom(a)$ and $a$ belong to $N$ (by closure under $\kappa$-sequences) and from these  we can define $\Upsilon_a.$ %because $a\s \eta\in N$ and $N$ is closed under $\kappa$-sequences. 
  Let us now check that this sequence fulfills the clauses in Definition~\ref{Definitionextender}. First, each $E_a$ is a $\kappa$-complete, non $\kappa^+$-complete, ultrafilter over $[\kappa]^{<\kappa}$: % Second $E_{\{\langle 0,\kappa\rangle\}}$ is not $\kappa^+$-complete because 
  this is because each $A_\alpha:=\{\sigma\in [\kappa]^{<\kappa}\mid \min(\sigma)>\alpha\}\in E_{a}$ but their intersection is empty. Second, $E_{\{\kappa\}}$ is normal: Let $f\colon X\rightarrow \kappa$ be a regressive function with $f(\sigma)<\sup(\sigma)$  for all  $\sigma\in X\in E_{\{\kappa\}}$. Then, $f \in M$ and $j(f)(\{\kappa\})<\kappa$. Thus, there is some $\alpha<\kappa$ such that  $\{\sigma\in [\kappa]^{<\kappa}\mid f(\sigma)=\alpha\}\in E_{\{\kappa\}}$. Finally, the sequence satisfies the coherency requirement. For this it suffices to show that $j(\rho_{b,a})(b)=a$. Indeed, $j(\hat{a})\cap j(\kappa)_b=j(\hat{a})\cap \kappa=\hat{a}$. Also, $ j(f)\restriction\kappa=f$. So, $j(\rho_{b,a})(b)= \range((j(\pi)^{-1}_b\circ f)\restriction \hat{a})=\range(\id``a)=a.$%\ale{The argument uses the revised $\rho_{b,a}$. I left a comment when you defined it.}% for some $\alpha<\kappa.$
  %Third, for  $b\s a$ define $\pi_{b,a}\colon \sigma \mapsto \sigma\restriction\dom(a)$. An easy computation shows that $j(\pi_{b,a})(\Upsilon_b)=\Upsilon_a$. From this it is routine to show that this choice of $\pi_{b,a}$ complies with the coherency requirement in Definition~\ref{Definitionextender}(2). Finally let us check that $E_a$ is normal %for all indices $a\in {}^{\leq \kappa} \eta$ in the sense of Definition~\ref{Definitionextender}(3).  Suppose that $\{\sigma\in \Sigma_a\mid f(\sigma)\in \sup(\sigma)\}\in E_a.$ Then, by definition, $$j(f)(\Upsilon_a)<\sup_{\alpha\in\dom(a)}a(\alpha).$$ If $\delta:=j(f)(\Upsilon_a)\in\range(a)$ we are done. Otherwise, let $b:=a\cup\{\langle \beta,\delta\rangle\}$ where $\beta\notin\dom(a).$
%Let $\alpha^*$ be the first ordinal in $\dom(a)$ such that \ale{I tried to fix it, but this is still problematic. $b$ is not necessarily continious.}$\delta<a(\alpha^*)$. Since $$\sup_{\beta\in \dom(a)\cap \alpha^*} a(\beta)<a(\alpha^*)$$ and $a$ is continuous it follows that $\gamma:=\sup(\dom(a)\cap \alpha^*)<\alpha^*$. Define $b:=a\cup\{\langle \gamma, \delta\rangle\}.$ It is easy to show that $j(\pi_{b,a})(\Upsilon_b)=\delta$, which implies $\{\sigma\in \Sigma_b\mid f(\pi_{b,a}(\sigma))\in \range(\sigma)\}\in E_b$, as needed.
  %\ale{Complete. I think the projection should be definable using canonical functions when $\otp(b)\geq \kappa.$}
  % ${}^{\kappa}N\s N$ and $\eta<j(\kappa)$.
\end{proof}

%Note that this definition is well-posed: First, $\mathcal{P}(\Sigma_a)\s M$ because \ale{Need to complete! We also have the problem that $\eta$ needs to be contained in $M$}.\ale{Somehow we would like to change the seed and say that it is $\{\langle j(\alpha), \eta_\alpha\rangle \mid \alpha\in a\}$ for $a\in [2^\lambda]^{\leq\kappa}$. However this does not satisfy the normality requirement of \S2.} Second, $(j\restriction a)^{-1}\in N\cap j(\Sigma_a)$ because ${}^{\kappa}N\s N$ and $\eta<j(\kappa)$. 
%For each $a\in [2^\lambda]^{\leq \kappa}$, let $\nu_a:=\{\langle j(\alpha), s_\alpha\rangle \mid\alpha\in a\}$. Since ${}^{\kappa}N\s N$,  $\nu_a\in N$. Define
%$$X\in E_a\;\Longleftrightarrow\; X\s X_a\; \wedge\, \nu_a\in j(X),$$
%where $X_a:=\{f\colon a\rightarrow \kappa\mid \text{$f$ increasing}\,\wedge\, |f|<\kappa\}$. 
%One can check that $E$ is a $(\kappa,\eta)$-extender in the sense of Definition~\ref{Definitionextender}. Also, 
For each $\alpha<2^\lambda$ and $A\s\kappa$ note that
$$A\in U_\alpha\;\Leftrightarrow\;\{\{\langle \alpha,\beta\rangle\}\mid \alpha<\kappa,\,\beta\in A\}\in E_{\{ \kappa, \eta_\alpha\}}.$$
%$$E_{\{\langle \alpha, \eta_\alpha\rangle\}}=\{\{\{\langle \alpha,\beta\rangle\}\mid \beta\in A\}\mid A\in U_\alpha\}.$$
Appealing to the duality gluing extender/gluing property (see page~\pageref{DualityGluingExtender}) we conclude that the cardinal $\kappa$ has the $2^\lambda$-gluing property.
%This yields a $\kappa$-complete ultrafilter on $X_a$ such that $E_a\in M$.\footnote{Note that here we need to use that $2^\lambda\s M$ and ${}^{\kappa}M\s M$.} Put $E:=\langle E_a\mid a\in [2^\lambda]^{\leq \kappa}\rangle$ and observe that 
%for each $\alpha<2^\lambda$ and $A\in U_\alpha$, $\hat{A}:=\{\{\langle \alpha,\gamma\rangle \}\mid \gamma\in A\}$ is a member of $E_{\{\alpha\}}$ as needed. %\ale{This is not exactly how you formulate it, but I think it's morally the same. We can think on what's the better formulation for gluing}
\end{proof}

\begin{cor}
If $\kappa$ is a strongly compact cardinal then it has the $\lambda$-gluing property for every cardinal $\lambda$.
\end{cor}
%Theorem~\ref{kappacompactnessandgluing} can be also used to show that $\lambda^+$-$\Pi^1_1$-subcompactness (see Definition~\ref{Subcompactness}) entails the $2^\lambda$-gluing property. 
\begin{cor}
Let $\kappa\leq \lambda$ be uncountable cardinals with $\lambda^{<\kappa}=\lambda$. 

If $\kappa$ is $\lambda^+$-$\Pi^1_1$-subcompact then it has the $2^\lambda$-gluing property.
\end{cor}
\begin{proof}
    Combining Theorem~\ref{kappacompactnessandgluing} with the first author's result that $\lambda^+$-$\Pi^1_1$-subcompact cardinals have the $\lambda$-filter extension property (see \cite[Theorem~11]{Hayutpartial}) the corollary follows rightaway.
\end{proof}
\begin{remark}
    After the latest developments made in \cite{HayPovII} we discovered that a Kunen \cite{Ketonen} and (independently) Comfort--Negrepontis \cite{CNBook} showed that every $\kappa$-compact cardinal  has the \emph{$(2^\kappa)$-gluing property via ultrafilters} \cite{HayPovII}; namely, the gluing embedding $j\colon V\rightarrow M$ is the ultrapower by a $\kappa$-complete ultrafilter on $\kappa$.  In particular, cardinals $\kappa$ that are $\kappa^+$-$\Pi^1_1$-subcompact do have the $(2^\kappa)$-gluing property. The result of Kunen and Comfort--Negreponthis thus improve our Theorem \ref{kappacompactnessandgluing}.
\end{remark}

%In the next subsection we show that the above result is (consistently) optimal. Indeed, we shall  exhibit a generic extension where $\kappa$ is $\kappa^+$-$\Pi^1_1$-subcompact but it does not have the $(2^\kappa)^+$-gluing property. 

\begin{theorem}\label{thm:2-extendible}%\yair{Changed from a remark on the text to a theorem}
  Assume that there is $j \colon H(\kappa^{++}) \to H(\lambda^{++})$ with $\crit j = \kappa$ and $\kappa < \lambda$. Then, $\kappa$ has the $\lambda^{+}$-gluing property.
\end{theorem}
\begin{proof}
    The proof is very similar to the one displayed in Theorem~\ref{kappacompactnessandgluing}. Let $\langle U_\alpha \mid \alpha < \lambda^{+}\rangle$ be a sequence of $\kappa$-complete measures over $\kappa$. For each $\alpha<\lambda^+$, $U_\alpha \in H(\kappa^{++})$ and thus $j(U_\alpha) \in H(\lambda^{++})$ is a $\lambda$-complete ultrafilter. Since $H(\lambda^{++})$ is closed under sequences of length $\lambda^{+}$, the sequence $\langle j(U_\alpha) \mid \alpha<\lambda^{+}\rangle \in H(\lambda^{++})$. Moreover, $j\image \mathcal{P}(\kappa) \in H(\lambda^{++})$ hence, in $H(\lambda^{++})$, one can compute the set $j\image U_\alpha = j(U_\alpha) \cap j\image \mathcal{P}(\kappa)$ for all $\alpha<\lambda^+.$   So, one can continue in the same fashion as in the proof of Theorem \ref{kappacompactnessandgluing}, and construct the sequence of ordinals $\eta_\alpha$ and the corresponding extender. 
\end{proof}
Theorem \ref{thm:2-extendible} above shows \emph{$2$-extendible cardinals with target $\lambda$} have the $\lambda^{+}$-gluing property. In particular, the global gluing property follows from the assumption of $2$-extendiblity with unbounded target. This is strictly weaker (in terms of consistency strength) than $\kappa$ being $\kappa^{++}$-supercompact. %\ale{Reference?} 
So, in order to isolate a property which is more characteristic to supercompactness, one will need to glue $\kappa$-complete measures over arbitrary sets. We will not pursue this direction in this paper.

\subsection{\texorpdfstring{$\Pi^1_1$}{Pi11}-subcompactness and the gluing property}\label{section:consistent-bounds-on-gluing}
In this section we correct a mistake that appeared in earlier versions of this work (including the published version \cite{HP}) where it was claimed that $\kappa$-compact cardinals do not necessarily posses the $(2^\kappa)^+$-gluing property. The source of the error was \cite[Lemma 4.7]{HP} where it was claimed that any measurable cardinal $\kappa$ with $o(\kappa)=\kappa^{++}$ satisfying the $\kappa^{++}$-gluing property carries an elementary embedding $j\colon V\rightarrow M$ with $\crit(j)=\kappa$, ${}^\kappa M\s M$ and $\kappa^{++}_M=\kappa^{++}$. The argument given in \cite{HP} is flawed because the factor embeddings $k_\alpha$ do have critical point $\kappa$ and as a result $k(\kappa^{++}_{N_\alpha})\neq \kappa^{++}_M$. This bug was recently pointed out to us by G. Goldberg to whom we are very grateful.

\smallskip

The next confirms that  \cite[Theorem 4.12]{HP} was not correct:

\begin{lemma}\label{lemma: pi11havelotsofgluings}
    If $\kappa$ is $\kappa^+$-$\Pi^1_1$-subcompact then $\kappa$ has the $(2^\kappa)^+$-gluing property.
\end{lemma}
\begin{proof}
    By Kunen \cite[Theorem~2.3]{Ketonen} or Comfort–Negrepontis \cite[Theorem~4.3]{CNBook}, the Rudin-Keisler order on $\kappa$-complete ultrafilters on $\kappa$ is $(2^\kappa)^+$-directed. Let $\langle U_\alpha\mid \alpha<(2^\kappa)^+\rangle$ be  $\kappa$-complete ultrafilters on $\kappa$. %We use the  $(2^\kappa)^+$-directedness of the Rudin-Keisler order to produce a $\leq_{\mathrm{RK}}$-increasing sequence of ultrafilters $\langle W_\alpha\mid \alpha<(2^\kappa)^+\rangle$ as follows. 
  
  Let $\{I_\alpha\mid \alpha<(2^\kappa)^+\}$ be a partition of $(2^\kappa)^+$ into sets of size ${\leq}2^\kappa$. First, let $W_0$ be a $\kappa$-complete ultrafilter that is $\leq_{\mathrm{RK}}$-above $\{U_\alpha\mid \alpha\in I_0\}$. Suppose that $\langle W_\alpha\mid \alpha<\beta\rangle$ has been defined. If $\beta$ is successor then we find $W_\beta$ that is Rudin--Keisler above the ultrafilters in $\{W_{\beta-1}\}\cup\{U_\alpha\in \alpha\in I_\beta\}$. If $\beta$ is limit then we fix an increasing sequence $\langle \beta_i\mid i<\cf(\beta)\rangle$ that is cofinal in $\beta$ and we find an ultrafilter $W_\beta$  that is Rudin--Keisler-above all the ultrafilters in $\{W_{\beta_i}\mid i<\cf(\beta)\}\cup\{U_\alpha\mid \alpha\in I_\beta\}$.

  Let $j\colon V\rightarrow M$ be the direct limit embedding of the directed system $$\langle j_{\alpha,\beta}\colon  M_{W_\alpha}\rightarrow M_{W_\beta}\mid \alpha\leq\beta <(2^\kappa)^+\rangle$$
  where $j_{\alpha,\beta}\colon [f]_{W_\alpha}\mapsto [f\circ \pi_{\beta,\alpha}]_{W_\beta}$ and $\pi_{\beta,\alpha}\colon \kappa\rightarrow \kappa$ witnesses $W_\alpha\leq_{\mathrm{RK}}W_\beta$. It is not hard to show that  $\crit(j)=\kappa$,  ${}^\kappa M\s M$ and that  each of the $M_{W_\alpha}$'s are closed under $\kappa$-sequences of its elements.

  For each $\alpha<(2^\kappa)^+$ let $i(\alpha)<(2^\kappa)^+$ be the unique index where $\alpha\in I_{i(\alpha)}$. By construction, $U_\alpha\leq_{\mathrm{RK}}W_{i_\alpha}$. Chasing the diagram it is  easy to check that  $U_\alpha=\{X\s \kappa\mid \eta_\alpha\in j(X)\}$ for $\eta_\alpha=j(\sigma_{\alpha, i_\alpha})([\id]_{W_\alpha})$ and $\sigma_{\alpha,i_\alpha}\colon \kappa\to \kappa$ the map witnessing $U_\alpha\leq_{\mathrm{RK}} W_{i_\alpha}$. We have shown that $\langle U_\alpha\mid \alpha<(2^\kappa)^+\rangle$ can be 'weakly-glued' (in the sense of \cite[Definition 3.3]{HayPovII}), which is equivalent to saying that they can be glued (see the proof of \cite[Theorem 3.5]{HayPovII}).
\end{proof}
We will show that the same argument given \cite[Theorem 4.12]{HP} shows that $\kappa^+$-$\Pi^1_1$-subcompact cardinals do not have the $\kappa^{+\omega+1}$-gluing property. Although this is still far from being as sharp as it was claimed in the published version of this paper \cite{HP}, it still demonstrates that Gitik's argument from \cite[\S2]{GitikOnMeasurables} requires more than a $\kappa$-compact cardinal to work.

\smallskip

%We will use the previous lemma to show that (consistently) there are $\kappa^+$-$\Pi^1_1$-subcompact cardinals without the $(2^\kappa)^+$-gluing property. 

  Given any set $X$ of ordinals we shall denote by $\acc(X)$ the set of accumulation points of $X$: namely, $$\acc(X):=\{\alpha\in X\mid \sup(X\cap \alpha)=\alpha>0 \}.$$
\begin{definition}
    Let $\lambda$ be an uncountable cardinal. A \emph{$\square_\lambda$-sequence} is a sequence of the form $\langle C_\alpha\mid \alpha\in\acc(\lambda^+)\rangle$ such that:%\footnote{F we are $\acc(X)$ denotes the } 
    \begin{enumerate}
        \item $C_\alpha\s \alpha$ is a club in $\alpha$,
        \item $\otp(C_\alpha)\leq \lambda$,
        \item $\forall \beta\in \acc(C_\alpha)\,(C_\beta=C_\alpha\cap \beta).$
    \end{enumerate}
\end{definition}
The principle $\square_\lambda$ (and its variants) is a well-studied set-theoretic property and a source for anti-compact objects of size $\lambda^+$. We refer the interested reader to the excellent article \cite{CumSquares} for an extensive study of this principle. In this paper we are interested in one feature of $\square_\lambda$ - its tension with stationary reflection.
\begin{definition}
    Let $\lambda$ be an uncountable regular cardinal and $S\s \lambda$ a stationary set. We say that $\refl(S)$ holds if for every stationary set $T\s S$ there is $\alpha<\lambda$ with $\cf(\alpha)>\omega$ such that $T\cap \alpha$ is stationary in $\alpha$. 
\end{definition}
It is  well-known that $\square_\lambda$ implies the failure of $\refl(S)$  for every stationary set $S\s \lambda$ (see e.g. \cite[Theorem~2.1]{CumSquares}). %In the proof of the forthcoming Theorem~\ref{KillingGluing} we shall benefit from this tension to show that (consistently) there is a generic extension with a $\kappa^+$-$\Pi^1_1$-subcompact cardinal $\kappa$ but with no elementary embedding $j\colon V\rightarrow M$ with $\crit(j)=\kappa$ and $\kappa^{++}_M=\kappa^{++}.$  

\smallskip

The following is the natural poset to force a $\square_\lambda$-sequence:% based on bounded approximations:
\begin{definition}
    Let $\lambda$ be an uncountable cardinal. The poset $\mathbb{S}(\lambda)$ to add a $\square_\lambda$-sequence consists of sequences $p=\langle p(\alpha)\mid \alpha\in\dom(p)\rangle$ such that:
    \begin{enumerate}
        \item $\dom(p)=(\beta+1)\cap\acc(\lambda^+)$ for some $\beta\in \acc(\lambda^+)$;
        \item for each $\alpha\in \dom(p)$, $p(\alpha)\s \alpha$ is  a club and $\otp(p(\alpha))\leq\lambda$;
        \item for all $\alpha\in \dom(p)$ and $\beta\in\acc(p(\alpha))$, $p(\alpha)\cap \beta=p(\beta).$
    \end{enumerate}
   For each $p,q\in\mathbb{S}(\lambda)$, $p\leq q$ iff $q=p\restriction\dom(q).$
\end{definition}
\begin{remark}
    It can be shown that forcing with $\mathbb{S}(\lambda)$ adds a generic $\square_{\lambda}$-sequence. Also, $\mathbb{S}$ is $\sigma$-closed and ${<}\lambda^+$-strategic closed, hence it does not add new $\lambda$-sequences of ordinals. (see \cite[Definition~5.15 and p.797]{CumHandBook})
%   preserves cardinals ${\leq}\lambda^+$.
%\ale{Do we wnat to explain strategic closure in a footnote?}
\end{remark}

Let  $\mathbb{S}$  be the Easton-supported iteration $\langle \mathbb{S}_\alpha, \dot{\mathbb{Q}}_\beta\mid \beta\leq \alpha< \kappa\rangle$\footnote{Since any poset can be regarded as one of Prikry-type in the sense of Gitik (Definition~\ref{GitikPrikrydef}) this is nothing but a forcing iteration witnessing Clause~(a) in Definition~\ref{Gitikiteration}.} defined as follows: If $\alpha< \kappa$ is a singular cardinal with countable cofinality then $\mathds
{1}_{\mathbb{S}_\alpha}\forces_{\mathbb{S}_\alpha}\text{$``\dot{\mathbb{Q}}_\alpha=\mathbb{S}(\alpha)$''};$ otherwise, $\one_{\mathbb{S}_\alpha}\forces_{\mathbb{S}_\alpha}\text{$``\dot{\mathbb{Q}}_\alpha$ is trivial''}.$ 

Let $G$ a $V$-generic filter for $\mathbb{S}$. 

%Some few comments are in order at this point. First, 

\begin{theorem}\label{KillingGluing}
Assume the $\mathrm{GCH}$ holds. Let $\kappa$ be a $\kappa^+$-$\Pi^1_1$-subcompact cardinal for which $\refl(E^{\kappa^{+\omega+1}}_\omega)$ holds. Then, the following hold in $V[G]$:
\begin{enumerate}
    \item $\kappa$ is $\kappa^+$-$\Pi^1_1$-subcompact;
    \item There is no embedding $j\colon V\rightarrow M$ with $\crit(j)=\kappa$ and ${}^\kappa M\s M$ and $j(\kappa)>\kappa^{+\omega}$.
\end{enumerate}
%After forcing with the above iteration $\kappa$ remains $\kappa^+$-$\Pi^1_1$-subcompact but there is no elementary embedding $j\colon V\rightarrow M$ such that $\crit(j)=\kappa$, ${}^\kappa M\s M$ and $\kappa^{++}_M=\kappa^{++}$.
In particular, in $V[G]$, $\kappa$ does not have the $\kappa^{+\omega}$-gluing property.
\end{theorem}
\begin{proof}
We begin proving the following auxiliary claim:
\begin{claim}\label{ReflectionClaim}
$V[G]\models \refl(E^{\kappa^{+\omega+1}}_\omega)$.% holds.
\end{claim}
\begin{proof}[Proof of claim]
Let $\dot{S}$ be a $\mathbb{S}$-name and $p\in G$ a condition such that $$p\forces_{\mathbb{S}}\text{$``\dot{S}$ is a stationary subset of $\dot{E}^{\kappa^{+\omega+1}}_\omega$''.}$$
For each $q\leq p$, define $S_q:=\{\alpha<\kappa^{+\omega+1}\mid q\forces_{\mathbb{S}}\check{\alpha}\in\dot{S}\}$. Note that $S_r\s S_q$ for all $q\leq r\leq p$. Thus, if $S_q$ is stationary (in $V$) then so is $S_r.$

Clearly $p$ forces $``\bigcup_{q\in \dot{G}/p} \check{S}_q=\dot{S}$''. Since $\mathbb{S}\s V_\kappa$ (and $\kappa$ is inaccessible) this is a union of ${\leq}\kappa$-many sets whose union is $p$-forced to be stationary. Thus,  there must be some $q\leq p$ for which $S_q$ is stationary (in $V$). By genericity and our previous comments we may assume that $q\in G$. 

Next working in $V$, use that $\refl(E^{\kappa^{+\omega+1}}_\omega)$ holds to find an ordinal $\delta$ with $\theta:=\cf^V(\delta)>\omega$ such that $S_q\cap \delta$ is stationary. Put $S^*:= S_q\cap \delta.$ Clearly, $\check{S}^*\s\dot{S}_G\cap\delta$, so it suffices to check that $q\forces_{\mathbb{S}}\text{$``\check{S}^*$ is stationary''}$.

Indeed, $\mathbb{S}$ is $\sigma$-closed and in particular, proper. $S^*$ is a stationary set consist of ordinals of countable cofinality, and thus its stationarity is preserved by a proper forcing.
\end{proof}
The next claim yields Clause~(1) of the theorem:
\begin{claim}\label{PreservationofSubcompactness} 
The cardinal $\kappa$ is $\kappa^+$-$\Pi^1_1$-subcompact in $V[G]$.
\end{claim}
\begin{proof}[Proof of claim]
Working in $V[G]$, assume that $\Phi$ is the $\Pi^1_1$-sentence of the form $\forall X \varphi(X, A)$, where $A\s H(\kappa^+)$.\footnote{Formally speaking, $\Phi$ is a $\Pi^1_1$-sentence in the language of set theory extended by the predicate $A$.} Let us assume that $$V[G]\models \text{$``\langle H(\kappa^+), \in, A\rangle \models \forall X\, \varphi(X, A).$''}$$ 

%\ale{Please, check this paragraph.}\yair{I don't think that the $\kappa$-c.c.\ actually plays a role in this part of the argument}  
Let $y$ be a non-empty element of $H(\kappa^{+})^{V[G]}$. By standard coding arguments, there is $x \subseteq \kappa \times \kappa$ that codes $y$, in the sense that $x$ is a well founded relation of $\kappa$ and its transitive collapse is $y$. Let $\dot{x}$ be a good name for $x$, so $\dot{x}_G = x$. As $\mathbb S \subseteq V_\kappa$, such a name exists and it is a member of $H(\kappa^{+})^V$. 

Next, let $X \in V[G]$ be a subset of $H(\kappa^{+})^{V[G]}$. Then, as before, we can talk about the collection of all pairs $\langle p, \dot{x}\rangle$, where $p$ is a condition in $\mathbb{S}$ and $\dot{x}$ is a nice name for a subset of $\kappa \times \kappa$ such that $p$ forces that the set $y$ coded by $\dot{x}$ is in $\dot{X}$. This collection is again a subset of $H(\kappa^{+})^V$. 

Since the model $H(\kappa^{+})^{V[G]}$ can decode the codes in a definable way, we may assume that the statement $\varphi$ talks about codes for elements of $H(\kappa^{+})$, instead of arbitrary elements.    Thus, we are legitimated to pick a $\mathbb{S}$-name $\dot{A}\s H(\kappa^+)^ V$  such that $\dot{A}_G=A$.

%Recall that $\mathbb{S}$ is a $\kappa$-c.c.\ forcing notion and $\mathbb{S}\s V_\kappa \subseteq H(\kappa^{+})$. This implies that  every subset $X$ of $H(\kappa^{+})^{V[G]}$ in the generic extension $V[G]$ has a $\mathbb{S}$-name $\dot{X} \subseteq H(\kappa^{+})^V$ such that $\dot{X}_G=X$. Indeed, every such $X$ can be coded (in $V[G]$) via a subset of $2^\kappa$. Also, every $\mathbb{S}$-nice name for a subset of $2^\kappa$ is a subset of $H(\kappa^+)^V$ precisely because $\mathbb{S}$ is $\kappa$-cc and $\mathbb{S}\s H(\kappa^+)^V$. 
  
  %For the very same reasons, every $y\in H(\kappa^{+})^{V[G]}$ has a name $\dot{y} \in H(\kappa^{+})^V$ such that $\dot{y}_G = y$.  

  \smallskip

%Let us assume that in $V[G]$, 
%\[\langle H(\kappa^{+}), \in, A\rangle \models \forall X \subseteq H(\kappa^{+}), \varphi(X, A),\]

The \emph{Forcing Theorem} yields that for some condition $p\in G$, 
\[p \Vdash_{\mathbb{S}} \langle \dot{H}(\kappa^{+}), \in, \dot{A} \rangle \models \forall X \, \varphi(X, \dot{A})\]
which is equivalent, by the above discussion, to
\begin{equation*}\label{pi11}
 \tag{$\Upsilon$}   \langle H(\kappa^+)^V, \in, \dot{A}, \forces_{\mathbb{S}}, p\rangle \models \forall \dot{X} (\dot{X}\text{ is an }\mathbb{S}\text{-name}\implies\text{$``p\forces_{\mathbb{S}} \varphi(\dot{X}, \dot{A})$''}).
\end{equation*} 
Besides, by virtue of the \emph{Forcing Definability Theorem}, $\text{$``p\forces_{\mathbb{S}} \varphi(\dot{X},\dot{A})$''}$ is a first order sentence\footnote{Here we are using the fact that $\mathbb{S}$ is a \emph{member} of $H(\kappa^{+})$, and every dense subset of $\mathbb{S}$ belongs to $H(\kappa^{+})$.} in the structure $\langle H(\kappa^+)^V, \in, \dot{A}, \forces_{\mathbb{S}}, p, \dot{X}\rangle,$ hence the expression \eqref{pi11} is a $\Pi^1_1$-sentence in $\langle H(\kappa^+)^V, \in, \dot{A}, \forces_{\mathbb{S}}, p\rangle$.

\smallskip

Using that $\kappa$ is $\kappa^+$-$\Pi^1_1$-subcompact in $V$ we  find  $\varrho<\kappa$, $\dot{B}\s H(\varrho^+)^{V}$ and $q\in \mathbb{S}_\varrho$, along with  an elementary embedding $$j\colon \langle H(\varrho^+),\in, \dot{B},\forces_{\mathbb{S}_\varrho}, q\rangle\rightarrow \langle H(\kappa^+),\in, \dot{A},\forces_{\mathbb{S}}, p\rangle$$ with $\crit(j)=\varrho$, $j(\varrho)=\kappa$ and  such that %$\langle H(\varrho^+),\in, \dot{B},\forces_{\mathbb{S}_\varrho}, q\rangle\models \text{$``q\forces_{\mathbb{S}} \Phi$''}$.
\begin{equation*}\label{pi11two}
 \tag{$\ast$} \langle H(\varrho^+),\in, \dot{B},\forces_{\mathbb{S}_\varrho}, q\rangle\models\forall \dot{X} (\dot{X}\text{ is an }\mathbb{S}_\varrho\text{-name}\implies\text{$``q\forces_{\mathbb{S}_\varrho} \varphi(\dot{X}, \dot{B})$''}). 
\end{equation*}

Since $\crit(j)=\varrho$ and $\mathbb{S}_\varrho$ is an Easton-supported iteration of forcings in $H(\varrho^+)$ it follows that $j\restriction \mathbb{S}_\varrho=\mathrm{id}$.  In particular, putting $B:=\dot{B}_{G_\varrho}$, Silver's lifting criterion \cite[Proposition~9.1]{CumHandBook} implies that $j$ lifts to
$$j^*\colon \langle H(\varrho^+)^{V[G_\varrho]},\in, B\rangle\rightarrow \langle H(\kappa^+)^{V[G]},\in, A\rangle,$$%\ale{Should we say here that $H(\rho^+)^{V[G_\rho]}=H(\rho^+)[G_\varrho]$?}
with $\crit(j^*)=\varrho$ and $j^*(\varrho)=\kappa$.\footnote{Note that here we are implicitly using that $H(\varrho^+)^{V[G_\varrho]}=H(\varrho^+)[G_\varrho]$ and, similarly, that  $H(\kappa^+)^{V[G]}=H(\kappa^+)[G]$. This is granted by the coding argument given at the beginning of the claim. }

\smallskip

Note that since $j(q)=p\in G$ and $j\restriction \mathbb{S}_\varrho=\mathrm{id}$, $q\in G\cap \mathbb{S}_\varrho=G_\varrho$.  Also,  $\mathbb{S}$ is isomorphic to $\mathbb{S}_\varrho \ast \dot{\mathbb{S}}/\mathbb{S}_\varrho$ and $\dot{\mathbb{S}}/\mathbb{S}_\varrho$ is forced to be ${<}{\rho^{+}}$-strategically-closed (hence it does not add subsets to $H(\varrho^+)^{V[G_\varrho]}$). Combining this with expression \eqref{pi11two} above
%\[\langle H(\varrho^+)^V, \in, \dot{B}, \forces_{\mathbb{S} \restriction \varrho}, q\rangle \models \forall \dot{X}, (\dot{X}\text{ is an }\mathbb{S}\restriction \varrho\text{-name}\implies\text{$``q\forces_{\mathbb{S}\restriction \varrho} \varphi(\dot{X}, \dot{B})$''}),\]
we conclude that %$q\in G$ actually $\mathbb{S}$-forces  %over $\mathbb{S}$ 
%the same statement. 
$$q\forces_{\mathbb{S}_\varrho}\langle \dot{H}(\varrho^{+}), \in, \dot{B} \rangle \models \forall X\, \varphi(X, \dot{B}).$$
Since $q\in G$, $V[G] \models ``\langle H(\varrho^+)^{V[G_{\varrho}]}, \in , B\rangle \models \Phi$'', as desired.%\ale{Why $q$ forces such a thing? Maybe }
\end{proof}
Let us now prove Clause~(2). Assume towards a contradiction that in $V[G]$ there is an elementary embedding $j\colon V[G]\to M$ such that ${}^\kappa M\s M$ and $j(\kappa)>\kappa^{+\omega}$. Since $M$ is closed under $\omega$-sequences in $V[G]$, it sees $\kappa^{+\omega}$ as a singular cardinal of countable cofinality. Therefore, $\square_{\kappa^{+\omega}}$ holds in $M$. 

In addition, since the SCH holds both in $V$ and $M$ we have
$$(\kappa^{+\omega})^{+M}=((\kappa^{+\omega})^\omega)^M=(\kappa^{+\omega})^\omega =\kappa^{+\omega+1}.$$
Thus, $\square_{\kappa^{+\omega}}$ holds in $V[G]$, as well.  This forms a contradiction with the former  Claim~\ref{ReflectionClaim}  as $\square_{\kappa^{+\omega}}$ implies the failure of $\refl(E^{\kappa^{+\omega+1}}_\omega)$.  

In particular, this implies that $\kappa$ cannot have the $\kappa^{+\omega}$-gluing property in this model.
\end{proof}
By \cite[Corollary~7]{Hayutpartial},  $\refl(E^{\kappa^{++}}_\omega)$ holds provided $\kappa$ is a $\kappa^{++}$-$\Pi^1_1$-subcom\-pact cardinal. This readily implies the following corollary:
\begin{cor}
Assume the $\mathrm{GCH}$ and let $\kappa$ be $\kappa^{++}$-$\Pi^1_1$-subcompact.  Then, 
there is a generic extension  where $\kappa$ is $\kappa$-compact but does not have the $(2^\kappa)^+$-gluing property.
\end{cor}
%The previous discussion suggest the following general question:
%\begin{question}
%What is the exact consistency-strength of $$``\text{$\kappa$ is $\kappa^+$-$\Pi^1_1$-subcompact and $\refl(E^{\kappa^{++}}_\omega)$ holds''?}$$
%\end{question}
%\begin{question}
%Is it consistent that  $\kappa$ is  $\kappa^+$-$\Pi^1_1$-subcompact and $\kappa^{++}$-tall, yet there is no  $j\colon V\rightarrow M$ with $\crit(j)=\kappa$ and $\kappa^{++}_M=\kappa^{++}$.
%\end{question}
\section{The \texorpdfstring{$\omega$}{omega}-gluing property from a strong cardinal}\label{section:consistent-omega-gluing}
%In the next sections we deal with the consistency of the $\omega$-gluing property, starting with a large cardinal hypothesis weaker than subcompactness. 
%Let us observe first that the $\omega$-gluing property is in fact a property of measures and not of extenders.
%\begin{lemma}\label{lemma:omega-gluing-by-a-measure}
%   A cardinal $\kappa$ has the $\omega$-gluing property if and only if for every $\omega$-sequence of $\kappa$-complete ultrafilters $\langle U_n \mid n < \omega\rangle$ there is a $\kappa$-complete ultrafilter $W$ on  $\kappa^{\omega}$, concentrating on increasing sequences, such that the projection map $e_n$ sending $\eta \in \kappa^{\omega}$ to $\eta(n)$ is a Rudin-Keisler projection from $W$ to $U_n$.
%\end{lemma}
%\begin{proof}
%    The backwards direction is clear, as $W$ itself can be represented as an extender. For the forwards direction, let $E$ be a $\kappa$-complete extender such that there is a sequence $\langle \eta_n \mid n < \omega\rangle$, $U_n = \{X \subseteq \kappa^{\{n\}} \mid \{\langle n, \eta_n\rangle\} \in j_E(X)\}$. 
%   Let $W := \{X \subseteq \kappa^\omega \mid \langle \eta_n \mid n < \omega\rangle \in j_E(X)\}$. Then, $W$ witnesses the validity of the lemma.
%\end{proof}

Assume the GCH holds and let $\kappa$ be a strong cardinal. Let $\ell\colon \kappa\rightarrow V_\kappa$ be a \emph{Laver function}%\ale{Need to define it}  
in the sense of  \cite[\S2]{GitShe}. In this section we define a Gitik iteration $\mathbb{P}_\kappa$ yielding a generic extension where $\kappa$ has the $\omega$-gluing property. At this point the reader might want to revisit Definition~\ref{Gitikiteration}.% and \S\ref{sectionPrikrytype}.

\smallskip

Let $A\s \kappa$ be the collection of all %inaccessible 
measurable cardinals $\alpha<\kappa$ such that $\ell``\alpha\s V_\alpha.$ Notice that this is set is unbounded in $\kappa$ (in fact, stationary) by strongness of $\kappa$: Indeed, let $j\colon V\rightarrow M$ be an elementary embedding witnessing that $\kappa$ is $(\kappa+2)$-strong. Working inside $M$, $\kappa$ is measurable\footnote{Because $V_{\kappa+2}\s M$ and any measure over $\kappa$ belongs to the former.} and $j(\ell)``\kappa=\ell``\kappa\s V_\kappa$, hence $\kappa\in j(A)$. That means that $A\in \mathcal{V}$, where $\mathcal{V}:=\{X\s\kappa\mid \kappa\in j(X)\}.$ Since $\mathcal{V}$ is normal $A$ is stationary.

\smallskip

Let $B$ be the topological closure of the set $A\cup \{\beta+1\mid \beta\in A\}$. We define $\mathbb{P}_\kappa$ by induction on $\alpha\in B$ as follows. Suppose that $\mathbb{P}_\beta$ has been defined for all $\beta\in B\cap \alpha$. If $\alpha$ is a limit point of $B$ then $\mathbb{P}_\alpha$ is simply the Gitik iteration of the previous stages. Otherwise, suppose that $\alpha$ is a successor point of $\beta$ in $B$. If $\alpha\neq \beta+1$ then we stipulate $\mathbb{P}_\alpha:=\mathbb{P}_\beta\ast \dot{\mathbb{Q}}_\beta$, where $\one_{\mathbb{P}_\beta}\forces_{\mathbb{P}_\beta}\text{$``\dot{\mathbb{Q}}_\beta$ is trivial''}$. Otherwise we shall distinguish two cases:
%there are two options: either $\alpha=\beta+1$ or $\alpha=\beta+2$, for some $\beta\in A^{\ell}$. In the latter case  we stipulate $\mathbb{P}_\alpha:=\mathbb{P}_\beta\ast \dot{\mathbb{Q}}_\beta$, where $\one_{\mathbb{P}_\beta}\forces_{\mathbb{P}_\beta}\text{$``\dot{\mathbb{Q}}_\beta$ is trivial''}$. In the former, we need to distinguish two cases:
\begin{itemize}
    \item[$(a)$] If $p\forces_{\mathbb{P}_\beta}$``$\ell(\beta)=\dot{\mathcal{U}}$ is an $\omega$-sequence of  $\beta$-complete measures on $\beta$'', then $p\forces_{\mathbb{P}_\beta}\text{$``\dot{\mathbb{Q}}_\beta=\dot{\mathbb{T}}(\dot{\mathcal{U}})$''}$, where $\dot{\mathbb{T}}(\dot{\mathcal{U}})$ is a name for the $\mathcal{U}$-tree Prikry forcing of Definition~\ref{UtreePrikry}.

    \item[$(b)$] Otherwise, $\one_{\mathbb{P}_\beta}\forces_{\mathbb{P}_\beta}\text{$``\dot{\mathbb{Q}}_\beta$ is trivial''}.$ 
\end{itemize}

%The above recursive definition yields a $\kappa$-cc forcing iteration $\mathbb{P}_\kappa$.

\begin{theorem}\label{omegagluingnonoptimal}
Assume the $\gch$ and let $\kappa$ be a  strong cardinal.% $\kappa$. 

Then, forcing with $\mathbb{P}_\kappa$ yields a model for $\text{$``{\kappa}$ has the $\omega$-gluing property''}.$
\end{theorem}
\begin{proof}
Fix $G\s \mathbb{P}_\kappa$ a $V$-generic filter. Working in $V[G]$, let $\mathcal{U}=\langle U_n\mid n<\omega\rangle$ be an $\omega$-sequence of $\kappa$-complete %uniform 
measures on $\kappa$. Let $p\in G$, and $\langle \dot{U}_n\mid n<\omega\rangle$  be an $\omega$-sequence of $\mathbb{P}_\kappa$-nice-names  such that  $p$ forces these to have the above property.  Note that since $\mathbb{P}_\kappa$ has the $\kappa$-cc,\footnote{Recall that this holds for Gitik iterations, as we mentioned in page~\pageref{GitikIterationsccc}. A proof of this well-known fact can be found in \cite[Proposition~7.13]{CumHandBook}.} and $\mathbb{P}_\kappa\s V_\kappa$, 
%$\tau_n=\cup_{x\in \mathcal{P}(\kappa)}\{\check{x}\}\times A_x$. Since $|A_x|<\kappa$ and $\mathcal{P}(\kappa)\s V_\kappa$ then $A_x\in V_\kappa.$
we have $\dot{U}_n\in V_{\kappa+2}$ for all $n<\omega.$ By \cite[Lemma~2.1]{GitShe} we can find a $(\kappa,\kappa+2)$-strong embedding $j\colon V\rightarrow M$ such that $j(\ell)(\kappa)=\langle \dot{U}_n\mid n<\omega\rangle.$ We will use this to form a measure $\mathcal{W}$ on $\kappa^\omega$ with $U_n\leq_{\mathrm{RK}}\mathcal{W}$ for all $n<\omega.$ %The proof is inspired by work of Gitik. 
This, combined with Lemma~\ref{lemma:omega-gluing-by-a-measure}, will yield the  $\omega$-gluing property.

\smallskip

By our $\GCH$ assumption, we can enumerate all the $\mathbb{P}_\kappa$-nice-names for subsets of $\kappa^\omega$ as $\langle \dot{X}_\alpha\mid \alpha<\kappa^{+}\rangle$. %\footnote{Note that here we are making use of the GCH.} 
Look now at $j(\mathbb{P}_\kappa)$. This is isomorphic to an iteration of the form $\mathbb{P}_\kappa\ast \dot{\mathbb{Q}}_\kappa\ast \dot{\mathbb{R}}$, where 
\begin{equation*}\label{equationclosure}
   \tag{$\otimes$} {\one}_{\mathbb{P}_\kappa\ast \dot{\mathbb{Q}}_\kappa}\forces_{\mathbb{P}_\kappa\ast \dot{\mathbb{Q}}_\kappa}\text{$``\langle \dot{\mathbb{R}},\dot{\leq},\dot{\leq}^*\rangle$ is of Prikry-type and $\dot{\leq}^*$ is $\kappa^{+}$-closed''}.
\end{equation*}
The claim regarding $\leq^*$ is a consequence of the $\kappa^+$-closure of all the stages in the iteration (see Lemma~\ref{TreePrikryproperty}(2)) while the other follows from Gitik's theorem that this type of iterations have the Prikry property (see \S\ref{sectionPrikrytype}).

\smallskip

 In addition, since $p\forces_{\mathbb{P}_\kappa}^M\text{$``j(\ell)(\kappa)$ is an $\omega$-sequence of measures''}$, this latter condition forces $``\dot{\mathbb{Q}}_\kappa=\mathbb{T}(\dot{\mathcal{U}})$'' (see Clause~(a) above). % namely,  be the tree Prikry forcing with respect to the sequence $\langle \dot{U}_n\mid n<\omega\rangle.$ 
Note that the definition of $\mathbb{T}({\mathcal{U}})$ is absolute between $M[G]$ and $V[G]$ as $V[G]_{\kappa+2}\s M[G]$.\footnote{For a proof of this latter fact, see \cite[Proposition~8.4]{CumHandBook}.}

\smallskip

 By induction on $\alpha<\kappa^{+}$, and using \eqref{equationclosure}, define a sequence of $\mathbb{P}_\kappa\ast \mathbb{T}(\dot{\mathcal{U}})$-names $\langle \dot{r}_\alpha\mid \alpha<\kappa^{+}\rangle$   for conditions in $\mathbb{R}$ such that:
\begin{itemize}
    \item $\one \forces_{\mathbb{P}_\kappa\ast \mathbb{T}(\dot{\mathcal{U}})}``\langle \dot{r}_\alpha\mid \alpha<\kappa^{+}\rangle$ is $\leq^*$-decreasing'';
    \item for all $\alpha<\kappa^{+}$, $\one \forces_{\mathbb{P}_\kappa\ast \mathbb{T}(\dot{\mathcal{U}})}``\dot{r}_\alpha\parallel_{\dot{\mathbb{R}}}\dot{b}_\kappa\in j(\dot{X}_\alpha)$'',  where $\dot{b}_\kappa$ denotes the standard name for the generic $\omega$-sequence added by $\mathbb{T}(\dot{\mathcal{U}}).$ %\ale{More details?}
\end{itemize}
%This is possible by using the fact that 
Define $\mathcal{W}$ as follows: $X\in \mathcal{W}$ iff there is $q\in G$, $q\leq p$ and $\alpha<\kappa^+$ such that $$q\cup \{\langle \varnothing, \dot{T}\rangle\}\cup r_\alpha\forces^M_{j(\mathbb{P}_\kappa)}\dot{b}_\kappa\in j(\dot{X}_\alpha),$$
where $\dot{T}$ is a $\mathbb{P}_\kappa$-name for a $\dot{\mathcal{U}}$-tree (see Definition~\ref{Utree}). Clearly, $\mathcal{W}\in V[G].$

\begin{claim}
    $\mathcal{W}$ is a $\kappa$-complete measure over $\kappa^\omega$. In addition, $$\mathcal{I}:=\{\vec{\eta}\in \kappa^\omega\mid \text{$\vec{\eta}$ is increasing}\}\in \mathcal{W}.$$
\end{claim}
\begin{proof}[Proof of claim]
    The argument is fairly standard and follows Gitik (see e.g., \cite[p. 1439]{Gitik-handbook} or \cite[Lemma~2.2]{GitikNonStaI}). To give the reader a flavor of the proof we unwrap the argument only for the ``In addition'' assertion.

     Work in $V[G]$. Let $\alpha<\kappa^+$ be such that $\mathcal{I}=(\dot{X}_\alpha)_G$. %To streamline the proof, let us assume that $\one\forces_{\mathbb{P}_\kappa}\dot{X}_\alpha\s \kappa^\omega$. By elementarity, the trivial condition of $j(\mathbb{P}_\kappa)$ forces that $``j(\dot{X}_\alpha)\s j(\kappa)^\omega$''. 
    We know that $$\one\forces_{\mathbb{T}(\dot{\mathcal{U}})}\dot{r}_\alpha\parallel_{\dot{\mathbb{R}}}\dot{b}_\kappa\in j(\dot{X}_\alpha).$$
    Using the Prikry property of $\mathbb{T}(\dot{\mathcal{U}})$ (Lemma~\ref{TreePrikryproperty}~(1)) find $\langle \varnothing, T\rangle\leq^* \one$ deciding the truth value of the statement $\sigma_\alpha\equiv \dot{r}_\alpha\parallel_{\dot{\mathbb{R}}}\dot{b}_\kappa\in j(\dot{X}_\alpha).$ If the decision was negative then that would mean that $\langle \varnothing, T\rangle\cup r_\alpha$ forces the $\dot{\mathbb{T}}(\mathcal{U})$-generic sequence $\dot{b}_\kappa$ to not be increasing, which is false. Thus, $\langle \varnothing, T\rangle$ forces $\sigma_\alpha.$

    Work back in $V$. Let $q\leq p$ in $G$ forcing $``\langle \varnothing, T\rangle\forces_{\dot{\mathbb{T}}(\mathcal{U})}\sigma_\alpha.$'' By definition of $\mathcal{W}$ this shows that $(\dot{X}_\alpha)_G=\mathcal{I}\in \mathcal{W}$.
\end{proof}
Let $\mathcal{W}\restriction\mathcal{I}:=\{A\cap \mathcal{I}\mid A\in \mathcal{W}\}$. Since $\mathcal{I}\in\mathcal{W}$, $\mathcal{W}\restriction\mathcal{I}$ is a $\kappa$-complete measure over $\mathcal{I}$. In a slight abuse of notation let us keep denoting   $\mathcal{W}\restriction\mathcal{I}$ by $\mathcal{W}$. This is just a cosmetic move to eventually apply Lemma~\ref{lemma:omega-gluing-by-a-measure}.% the collection.

%\smallskip

%Standard arguments  prove that this is a $\kappa$-complete measure on (the increasing sequences in) $\kappa^\omega$, in $V[G]$. For more details, see  \cite[p. 1439]{Gitik-handbook}. 
\begin{claim}\label{Wglues}
The following holds in $V[G]$: For all $n<\omega$, $U_n\leq_{\mathrm{RK}}\mathcal{W}$ as witnessed by the evaluation map $e_n\colon \kappa^\omega\rightarrow \kappa$,  $\vec{s}\mapsto \vec{s}(n).$
 %$V[G]\models``\forall n<\omega\,( U_n\leq_{\mathrm{RK}}\mathcal{U})$''. Moreover, for each $n<\omega$, this is witnessed by $e_n\colon \kappa^\omega\rightarrow \kappa$, the evaluation map $f\mapsto f(n).$ % for all $n<\omega.$
\end{claim}
\begin{proof}[Proof of claim]
 %For each $n<\omega$, let $e_n\colon \kappa^\omega\rightarrow \kappa$ be the evaluation map $f\mapsto f(n).$ 
 % We claim that $e_n``X\in U_n$ for all $X\in \mathcal{U}$. 
 Suppose otherwise. Let $q\in G$ forcing $``\dot{X}\in\mathcal{W}\,\wedge\, \dot{Y}\notin \dot{U}_n$'', where $\dot{Y}$ is a $\mathbb{P}_\kappa$-nice name for a subset of $\kappa$ such that $\dot{Y}_G=e_n``\dot{X}_G.$  Since $\dot{Y}$ is $\mathbb{P}_\kappa$-nice, $\dot{Y}\s V_\kappa$. %\ale{ This claim is problematic in the non-stationary-supported case. Actually, regardless $\dot{Y}\s V_\kappa$, isn't it always true that $\one \forces^M_{j(\mathbb{P}_\kappa)}j(\dot{Y})\cap \kappa=\dot{Y}.$? $\dot{Y}$ is forced to be a subset of $\kappa$ and the critical point of the lifting of $j$ is $\kappa$.}
% \yair{Not precisely: we don't lift the embedding $j$ (this is impossible, since it would require us to change the cofinality of $\kappa$). So, we must be careful, and use the fact that for every $\alpha < \kappa$, the maximal antichain that decides its membership to $\dot{Y}$ is not move. In the case of the non-stationary support, we must work below a condition $p\in G$ such that for every $\alpha < \kappa$ there is $\beta < \kappa$ such that $\alpha \in \dot{Y}$ is decides by meeting a dense open set in $\mathbb{P}_\beta$.}
 %\footnote{Once again, here we  use that $\mathbb{P}_\kappa$ is $\kappa$-cc and $\mathbb{P}_\kappa\s V_\kappa$.}
By definition of $\mathcal{W}$,  there is $\alpha<\kappa^{+}$ and $\dot{T}$ such that  $$q\cup \{\langle \varnothing, \dot{T}\rangle\}\cup r_\alpha\forces^M_{j(\mathbb{P}_\kappa)} \dot{b}_\kappa\in j(\dot{X}).\footnote{Formally speaking, there is $q'\leq q$ such that $q\cup \{\langle \varnothing, \dot{T}\rangle\}\cup r_\alpha$ forces the above. To not overcomplicate the notations we have opted to keep denoting this condition $q'$ by $q$.}$$ 
In particular, 
$$q\cup \{\langle \varnothing, \dot{T}\rangle\}\cup r_\alpha\forces^M_{j(\mathbb{P}_\kappa)} \dot{b}_\kappa(n)\in j(\dot{Y})\cap \kappa.$$ 
By our choice of $\dot{Y}$ it follows that $\one \forces^M_{j(\mathbb{P}_\kappa) / \mathbb{P}}j(\dot{Y})\cap \kappa=\dot{Y}_G.$ Indeed, for every $\alpha < \kappa$, as $\mathbb{P}_\kappa$ is $\kappa$-cc (see p.\pageref{GitikIterationsccc}), the antichain that decides the membership of $\alpha$ to $\dot{Y}$ does not move  under $j$, and thus that antichain decides in the same way the membership of $\alpha$ to $j(\dot{Y})$. So, all in all, the above condition forces $``\dot{b}_\kappa(n)\in \dot{Y}$''. On another front, since $q\forces_{\mathbb{P}_\kappa}\dot{Y}\notin \dot{U}_n$, there is $q'\leq q$ in $G$ such that $q'\forces_{\mathbb{P}_\kappa}\langle \varnothing, \dot{T}'\rangle\leq^* \langle \varnothing,\dot{T}\rangle$, where $\dot{T}'$ is  a $\mathbb{P}_\kappa$-name for a tree whose $n^{\mathrm{th}}$-level $\mathrm{Lev}_n(T')$ is disjoint from $\dot{Y}$ (see Definition~\ref{Utree}). 

Then, $$q'\cup \{\langle \varnothing, \dot{T}'\rangle\}\cup r_\alpha\leq q\cup \{\langle \varnothing, \dot{T}\rangle\}\cup r_\alpha$$ and, clearly, the stronger condition forces $``\dot{b}_\kappa(n)\notin \dot{Y}$''. This conflicts with our previous comments, which yields the sought contradiction.
\end{proof}
Invoking Lemma \ref{lemma:omega-gluing-by-a-measure} we complete the proof of the theorem. \qedhere
%\begin{claim}
%In $V[G]$ there is $E=\langle E_a\mid a\in [\omega]^{<\omega}\rangle$ gluing $\langle U_n\mid n<\omega\rangle.$
%\end{claim}
%\begin{proof}[Proof of claim] \ale{I'm a bit hesitant about the extender here. In particular, I'm not sure the resulting ultrapower is even $\sigma$-closed.}
%Let $i\colon V[G]\rightarrow N$ be the ultrapower by $\mathcal{U}$ and put $\eta:=[\mathrm{id}]_{\mathcal{U}}.$ Clearly, $\eta\in i(\kappa)^\omega$ and it is increasing.  Observe that $\eta(n)\in \bigcap i``U_n$ for all $n<\omega$: In effect, if $A\in U_n$ then, by the previous claim,  $e_n^{-1}A\in \mathcal{U}$, hence $\eta\in i(e_n^{-1}A)$, and thus $\eta(n)\in i(A).$ 

%For $a\in [\omega]^{<\omega}$ denote by  $X_a$ the set of order-preserving $f\colon a\rightarrow \kappa$.  Also, put $\nu_a:=\{\langle n, \eta(n)\rangle\mid n\in a\}.$ Using these we stipulate that
%$$Y\in E_a\;:\Longleftrightarrow\; Y\s %X_a\,\wedge\, \nu_a\in i(X).$$
%Clearly, $E=\langle E_a\mid a\in [\omega]^{<\omega}\rangle$ forms a $(\kappa,\omega)$-extender. Finally, it can be checked that $E_{\{n\}}\supseteq \{\hat{A}^n\mid A\in U_n\}$, where $\hat{A}^n:=\{\langle n,\alpha\rangle\mid \alpha\in A\}$. Since $U_n$ is a ultrafilter in $V[G]$ we actually have that $E_{\{n\}}=\{\hat{A}^n\mid A\in U_n\}$.
%\end{proof}
%The proof has been accomplished.
\end{proof}
\begin{remark}\label{StrongsDoNotHaveGluing}
    In spite of Theorem~\ref{omegagluingnonoptimal} it is not true that strong cardinals do enjoy the $\omega$-gluing property. Indeed, suppose that there is no inner model with a Woodin cardinal and let $\mathcal{K}$ denote the \emph{core model up to a Woodin cardinal} \cite{MitchellUpto}. Let $\kappa$ be the first strong cardinal. We claim that $\kappa$ does not have the $\omega$-gluing property in $\mathcal{K}$. Suppose otherwise, and let $j\colon \mathcal{K}\rightarrow M$ and $\langle \eta_n\mid n<\omega\rangle$   gluing $\langle U\mid n<\omega\rangle$ where $U$ is a normal measure on $\kappa.$ By the main result of \cite{Schi}, $j$ is a normal iteration of extenders in $\mathcal{K}$. Such an iteration must be finite, as ${}^\kappa M\cap \mathcal{K}\s M$. Say the critical points of this iteration are $\langle \mu_i\mid i\leq n\rangle$.  By the forthcoming Claim~\ref{claim:covering-almost-all-ordinals}, the set $\bigcap j``\mathrm{Cub}^{\mathcal{K}}_\kappa$ is contained in  $\{\mu_i\mid i\leq n\}$. However, the fact that $\langle \eta_n\mid n<\omega\rangle$ glues $\langle U\mid n<\omega\rangle$ ensures that  $\{\eta_n\mid n<\omega\}\s \bigcap j``\mathrm{Cub}^{\mathcal{K}}_\kappa$, which is impossible. 
\end{remark}

\section{Improving the upper bound for the \texorpdfstring{$\omega$}{omega}-gluing property}\label{section; improving}
In this section we would like to improve the large cardinal hypothesis employed in Theorem~\ref{omegagluingnonoptimal}. %to the existence of a measurable cardinal $\kappa$ with $o(\kappa)=\omega_1$
In particular we prove the following  result:
\begin{theorem*}\label{ImprovingConsistency}
It is consistent, relative to the existence of a measurable  with Mitchell's order $\omega_1$, that there is a cardinal $\kappa$ with the $\omega$-gluing property.
\end{theorem*}
Our final verification of the $\omega$-gluing property will follow the steps of the argument in Theorem~\ref{omegagluingnonoptimal}. However,  we will need to implement some few non-trivial modifications %to the above argument in order 
to have a better control upon the $\kappa$-complete ultrafilters that will appear in the intermediate models of the final generic extension (see the proof of Theorem~\ref{MainTheorem}).
To be able to identify these $\kappa$-complete ultrafilters %that exist in (intermediate) generic extension 
we will rely on the following blanket assumption:\footnote{Since we are assuming $V=\mathcal{K}$ we are in a situation where GCH holds (see page~\pageref{GCHinK}).}
\begin{equation}\label{BlanketAssumption}
 \tag{$\mathscr{H}$}   \text{$``\rm{V}=\mathcal{K}$'' + ``There is no inner model of $\exists \alpha\, (o(\alpha)=\alpha^{++})$''.}
\end{equation}

On one hand, the first of these assumptions (yet natural) is probably unnecessary, as demonstrated by Gitik and Kaplan in \cite{GitikKaplan-nonstationary2022}. On the other hand, our second demand will enable us  to appeal to \emph{Mitchell's Core Model theorem} (Theorem~\ref{CoreModelTheorem}) and exploit all the power of $\mathcal{K}$. In particular, we shall use this assumption to ensure that every (non-trivial) elementary embedding $j\colon \mathcal{K}\rightarrow M$ (with $M$ a transitive class) is a normal iteration using normal measures in $\mathcal{K}$ and to say that the GCH holds in our model.   %More precisely, we shall be in conditions .% will be available and we shall be in con, by a result of Mitchell \cite{MitIter}, (which was extended by Schindler \cite{Schi}),  that  every iteration of ultrapowers  is normal and that it uses only normal measures.

\subsection{Coding \texorpdfstring{$\kappa$}{kappa}-complete measures after adding a fast function}
The first technical problem that we must address is the lack of a Laver function. Formerly, in \S\ref{section:consistent-omega-gluing}, the existence of such object was granted by our departing large-cardinal assumption; namely, strongness. This time we would like to force some sort of Laver function in a way that we can still keep track of the possible measures in the extension. %This is precisely the aim of the following forcing notion:
%To cope with this issue we would like to force some sort of fast function in a way that we keep the amount of possible measures. This is precisely the aim of the following forcing notion: %force with \emph{Silver's forcing}:
For this, we will consider the following \emph{non-stationarily-supported} variation of \emph{Woodin's fast function} forcing.
\begin{definition}[Fast Function Forcing]
    For an inaccessible  $\kappa$  we denote by $\mathbb{S}_\kappa$ the poset consisting on partial functions $s \colon \kappa \to H(\kappa)$ such that  %$\dom s\s \mathrm{Inacc}$ and the following hold:
\begin{enumerate}\label{fastfunctionforcing}%\ale{Check if this is $\kappa^+$-cc, but I'd say it is just $\kappa^{++}$-cc. I think to make sure the codes are in $H(\kappa^+)_V$ we need $\kappa^+$-ccness. }
%     \item $|s|<\kappa$,
    \item $\dom s\s \mathrm{Inacc}$,
    \item $(\dom s) \cap \beta\in \mathrm{NS}_\beta$ for all $\beta\in\mathrm{Inacc}\cap (\kappa+1)$,
    \item and $s(\alpha)\in H(\alpha^+)$ for all $\alpha\in\dom s$.\footnote{The restriction of $s(\alpha)$ to $H(\alpha^{+})$ is inessential at this part.}
\end{enumerate}
The order of $\mathbb{S}_\kappa$ is defined naturally as $s\leq t$ iff $s\supseteq t.$
\end{definition}
The symbol $\mathrm{Inacc}$ above stands for the class of inaccessible cardinals and $\mathrm{NS}_\beta$ is the usual notation for the non-stationary ideal on the cardinal $\beta$.

\smallskip

The next proposition gathers the properties of $\mathbb{S}_\kappa$ that we shall use:
\begin{prop}[Some properties of $\mathbb{S}_\kappa$]\label{PropertiesofS}%Suppose $\kappa$ is inaccessible. Then,
\hfill
\begin{enumerate}
    \item $\mathbb{S}_\kappa$ adds a generic fast function.\footnote{Specifically a function $\ell\colon\kappa\rightarrow\kappa$ for which the following holds: For all $x\in H(\kappa^+)^V$ there is an  embedding $j\colon \mathcal{K}[S]\rightarrow M$ with $\crit(j)=\kappa$, ${}^\kappa M\s M$ and $j(\ell)(\kappa)=x$.}
    \item $\mathbb{S}_\kappa$ is $\kappa^{++}$-cc and preserves $\kappa^+$.
     \item Forcing with $\mathbb{S}_\kappa$ preserves the $\mathrm{GCH}$.%\ale{Even though we did not use it later I think we should keep it. The covering claim is gonna being invoked later on.}
    \item For every  condition $s=\{\langle \alpha, s(\alpha)\rangle\}\in \mathbb{S}_\kappa$ %the poset $\mathbb{S}_\kappa/s$ factors as 
    $$\mathbb{S}_\kappa/s\simeq \mathbb{S}_\alpha\times \mathbb{S}_\kappa\setminus \alpha,$$ where $\mathbb{S}_\alpha$ is $\alpha^{++}$-cc and $\mathbb{S}\setminus \alpha$ is $\theta_\alpha$-closed ($\theta_\alpha:=\min (\mathrm{Inacc}\setminus \alpha^+)$).
    \item  If $\kappa$ is measurable then forcing with $\mathbb{S}_\kappa$ preserves this property.
\end{enumerate}
\end{prop}
\begin{proof}
    (1), (4) and (5) are easy and we refer the reader to \cite[\S1]{HamLott} for details. The first part of Clause~(2) is evident by virtue of our GCH assumption.  The rest can be proved using a \emph{fusion-like} argument. We just give details for (3) as the preservation of $\kappa^+$ can be established analogously.
    \begin{claim}\label{kappabounding}
       The trivial condition of $\mathbb{S}_\kappa$ forces the following statement: 
       
       ``For every $\mathbb{S}_\kappa$-name $\dot{X}$ for a subset of ordinals of cardinality ${\leq}\kappa$ there is a set $Y\in V$, $|Y|^V\leq \kappa$, such that $\dot{X}\s Y$''. % there is $V\ni g\colon\kappa\rightarrow\kappa$ such that $\dot{f}(\alpha)<g(\alpha)$ for all $\alpha<\kappa$.
    \end{claim}
    \begin{proof}[Proof of claim]
        Given $s, t\in\mathbb{S}_\kappa$ and $\alpha<\kappa$ write $$\text{$s\sq_\alpha t\,\Leftrightarrow\,s\leq t$ and $s\restriction (\alpha+1)=t\restriction (\alpha+1).$}$$
        Clearly, $\langle \mathbb{S}_\kappa,\sq_\alpha\rangle$ is  a $2^\alpha$-closed poset. A sequence of conditions $\langle s_\alpha\mid \alpha< \kappa\rangle$ in $\mathbb{S}_\kappa$ will be called a \emph{fusion sequence} if $s_{\alpha+1}\sq_\alpha s_\alpha$ for all $\alpha<\kappa.$ 

        Let $s\in\mathbb{S}_\kappa$  and $\dot{f}\in V^{\mathbb{S}_\kappa}$ be such that $s\forces_{\mathbb{S}_\kappa}\dot{f}\colon\kappa\rightarrow\ord.$ We shall define a fusion sequence deciding the values of $\dot{f}$ as follows. Put $s_0:=s$. Suppose that $\langle s_\alpha\mid \alpha<\beta\rangle$ is given and  let $s_\beta:=\bigcup_{\alpha<\beta}s_\alpha.$ A moment's reflection makes clear that $s_\beta$ is a well-defined condition in $\mathbb{S}_\kappa.$ 
        
        Let $\langle t^\beta_\gamma\mid \gamma<\beta^{++}\rangle$ be an enumeration of all  the $\restriction(\beta+1)$-truncations of extensions of $s_\beta$. Inductively, define a decreasing sequence $\langle s^\beta_\gamma\mid \gamma<\beta^{++}\rangle$ in $\mathbb{S}\restriction\beta$ as follows: First, $s^\beta_0:=s_\beta$. At limit stages $\gamma$ we take $s^\beta_\gamma$ to be a lower bound of the previous stages (this is possible by (4)). At successor stages we ask whether there is $r\leq_{\mathbb{S}\setminus\beta} s^\beta_\gamma$ such that $t^\beta_\gamma\cup r\parallel_{\mathbb{S}_\kappa}\dot{f}(\beta)$. If so, we let $s^\beta_{\gamma+1}$ one of such $r$; otherwise, $s^\beta_{\gamma+1}:=s^\beta_\gamma.$ In the end let $\Delta_\beta$ be the collection of all decisions for $\dot{f}(\beta)$ and
        $\text{$s_{\beta+1}:=(s_\beta\restriction(\beta+1))\cup \bigwedge_{\gamma<\beta^{++}}s^\beta_\gamma$,}$
        where this latter component denotes  a lower bound of the $s^\beta_\gamma$'s.

        \smallskip

        After this process we have produced a fusion sequence $\langle s_\alpha\mid \alpha<\kappa\rangle$. Let $s_\kappa:= \bigcup_{\alpha<\kappa}s_\alpha$ and $Y:=\bigcup_{\alpha<\kappa}\Delta_\alpha$. Note that $s_\kappa$ is a condition and that $|Y|\leq \kappa$  because, for each $\alpha<\kappa$, $|\Delta_\alpha|\leq 2^\alpha<\kappa$. We claim that $$s_\kappa\forces_{\mathbb{S}_\kappa}\ran(f)\s Y.$$
        Let $\alpha<\kappa$ and $u\leq s_\kappa$ deciding a value for $\dot{f}(\alpha)$. Since $u\restriction(\alpha+1)\leq s_\kappa\restriction (\alpha+1)=s_\alpha\restriction (\alpha+1)$ it follows that $u\restriction(\alpha+1)=t^\alpha_\gamma$. Also, $u\setminus\alpha$ is a condition stronger than $s_{\alpha+1}\setminus\alpha$, hence stronger than $s^\alpha_\gamma$, such that $t^\alpha_\gamma \cup u\setminus \alpha$ decides the value of $\dot{f}(\alpha).$ By our construction, this implies that $t^\alpha_\gamma \cup s^\alpha_{\gamma+1}$ also decides the value of  $\dot{f}(\alpha)$ and this is, by definition, an ordinal in $\Delta_\alpha$. Thereby, $u\forces_{\mathbb{S}_\kappa} \dot{f}(\alpha)\in Y.$ This establishes the above  equation.

        \smallskip

        By genericity, the above argument yields the claim.
        %For each $\alpha<\kappa$, let $\Delta_{\alpha}\s \kappa$ the collection of ordinals decided by  
    \end{proof}
    Let $S\s \mathbb{S}_\kappa$ be a $\mathcal{K}$-generic filter and $\delta$ be some $\mathcal{K}[S]$-cardinal.
    
   $\br$ Suppose $\delta$ is below the first inaccessible. Then $\mathbb{S}_\kappa$ is $\delta^+$-closed and hence it does not disturb the power-set-pattern.

     $\br$ Suppose $\delta\geq \kappa^{+}$. The usual counting-nice-names arguments involving the $\kappa^{++}$-ccness of $\mathbb{S}_\kappa$ and GCH in $V$ imply that $2^\delta=\delta^+$ holds in $V[S].$ 

    $\br$ Suppose $\delta$ is between two inaccessibles. Then the above argument applies also using Clause~(4) of the proposition.

    $\br$ Suppose $\delta$ is some inaccessible $\leq \kappa$. We argue the result  for $\kappa$ as for the other cardinals the argument is the same.  %Suppose towards a contradiction that $\mathcal{K}[S]\models ``2^\kappa\geq \kappa^{++}$''. 
%    Each $x\in\mathcal{P}(\kappa)^{\mathcal{K}[S]}$ 
   % is naturally identified with its characteristic function, hence 
   By the above claim %yields some $y\in\mathcal{P}(\kappa)^{\mathcal{K}}$ such that $x\s y.$ This implies that $
   $$\mathcal{P}(\kappa)^{\mathcal{K}[S]}\s\bigcup\mathcal{P}(\kappa)^{\mathcal{K}}$$ and this latter union is of size $(\kappa^+)^{\mathcal{K}}=(\kappa^{+})^{\mathcal{K}[S]}$ by GCH in $\mathcal{K}$.
\end{proof}

%\begin{definition}
%Let $\mathbb{S}$ be the forcing poset consisting on partial functions $s \colon \kappa \to V_\kappa$ such that $\dom s\s \mathrm{Inacc}$, $s(\alpha)\in H(\alpha^+)$ for $\alpha\in\dom(s)$, and $(\dom s) \cap \beta\in \mathrm{NS}_\beta$ for $\beta\in\mathrm{Inacc}\cap \kappa$. % and $\dom s \subseteq \mathrm{Inacc}$. We require also that $s(\alpha) \in H(\alpha^+)$ for all $\alpha \in \dom s$.
%Given $s,s'\in\mathbb{S}$  write $s\leq s'$ iff $s\supseteq s'.$
%\end{definition}
%\begin{remark}
%    It is easy to show that forcing with $\mathbb{S}_\kappa$ does preserve cardinals and that if $S\s \mathbb{S}_\kappa$ is generic then the generic function (i.e. $\bigcup S$)  gives rise to the  entire generic extension $V[S]$. Another useful property of $\mathbb{S}_\kappa$ is the following: for any Mahlo cardinal $\alpha< \kappa$  and $s=\{\langle \alpha, s(\alpha)\rangle\}\in \mathbb{S}_\kappa$, the poset $\mathbb{S}_\kappa/s$ factors as $\mathbb{S}_\alpha\times \mathbb{T}$, where $\mathbb{S}_\alpha$ is $\alpha^{++}$-cc and $\mathbb{T}$ is $\theta_\alpha$-closed, being $\theta_\alpha$ is the first inaccessible past $\alpha$. For more information, see \cite[\S1]{HamLott}.
%\end{remark}
Hereafter we shall assume that $\kappa$ is a measurable cardinal and will write $\mathbb{S}$ instead of $\mathbb{S}_\kappa$. The following is the main result of the section. Succinctly, it gives a complete cartography of the $\kappa$-complete measures (over $\kappa$) in $\mathcal{K}[S]$:
\begin{lemma}\label{lemma:extending-S} 
%Let us assume $V = K$ and there is no inner model with $\exists \alpha o(\alpha) = \alpha^{++}$. 
Let $S \subseteq \mathbb{S}$ be a $\mathcal{K}$-generic filter and  $\mathcal{U}$ a $\kappa$-complete ultrafilter over $\kappa$ in $\mathcal{K}[S]$. Then, there is a finite  iteration $\langle \iota_{m,n}\mid m\leq n\leq k+1\rangle$ of measures in $\mathcal{K}$ with critical points $\langle \mu_0, \dots, \mu_k\rangle$, $\kappa = \mu_0 < \mu_1 < \dots < \mu_k$, a function $f \colon \kappa^{k + 1} \to \kappa$ in $\mathcal{K}$ and $\langle a_0,\dots, a_{k}\rangle \in\prod_{i\leq k} H(\mu_i^{+})$ such that
$$\mathcal{U}=\{\dot{X}_S\s \kappa\mid\exists p\in S\,(\iota(p)\cup\{\langle \mu_i,a_i\rangle\mid i\leq k\}\forces_{\iota(\mathbb{S})} \check{\epsilon}\in \iota(\dot{X}))\},$$
where $\epsilon=\iota_k(f)(\mu_0,\dots, \mu_{k}).$
%\textcolor{blue}{determine $\mathcal U$.}\ale{What does this exactly mean?}
%In particular, there are at most $o(\kappa) + \kappa^{+}$ many different $\kappa$-complete ultrafilters on $\kappa$ in $K[S]$.%\ale{Why? Spell out the calculation.}
\end{lemma}
Before being in conditions to addressing that proof we need to digress and prove some technical lemmas about linear iterations.  At this point the reader might want to revisit Definition~\ref{Iteration} in page~\pageref{Iteration}.

\begin{lemma}\label{lem;representing-elements-in-direct-limit}
Let $\langle \iota_{\alpha, \beta} \mid \alpha \leq \beta \leq \delta\rangle$ be a normal iteration using normal measures,  with critical points $\langle \mu_\alpha\mid \alpha<\delta\rangle$. Then, for every set $x \in \mathcal{M}_\delta$ there is a finite normal iteration $\tilde\iota \colon \mathcal{M}_0 \to \tilde{\mathcal{M}}$ together with a (factor) embedding $k \colon \tilde{\mathcal{M}} \to \mathcal{M}_\delta$ such that $k \circ \tilde{\iota} = \iota_\delta$ and $x \in \range(k)$.
\end{lemma}
\begin{proof}
A straightforward induction on $\delta$ shows that, for every $x \in \mathcal{M}_\gamma$ ($\gamma \leq \delta$) there is  $f \in \mathcal{M}_0$ and ordinals $\alpha_0 <  \dots < \alpha_{n-1} < \gamma$ such that \[\iota_\gamma(f)(\mu_{\alpha_0},\dots, \mu_{\alpha_{n-1}}) = x.\]

For every $\alpha < \delta$, let us pick a representation of $U_\alpha$; namely
\[U_\alpha = \iota_{\alpha}(g_\alpha)(\mu_{\beta^\alpha_0}, \dots, \mu_{\beta^\alpha_{m_\alpha - 1}}),\]
where each $\beta^\alpha_i$ is below $\alpha$.

Let us say that an increasing sequence of ordinals $\bar a\in [\delta]^{<\omega}$ is \emph{saturated} if  $\{\beta^\alpha_0,\dots\beta^\alpha_{m_\alpha-1}\} \subseteq \bar a$ for all $\alpha \in \bar a$. Clearly, every set  $b\in[\delta]^{<\omega}$ is a subset of a finite saturated set $\bar{a}$; this is because the generators corresponding to each $\alpha\in b$ are strictly below $\alpha$. In particular, given $x \in \mathcal{M}_\delta$ there is  $\bar{a} = \{\alpha_0,\dots,\alpha_{n-1}\}$ a saturated set and a function $f \in \mathcal{M}_0$ such that
$$\iota_\delta(f)(\mu_{\alpha_0},\dots, \mu_{\alpha_{n-1}}) = x.$$ Let us construct a finite iteration $\langle \tilde{\iota}_{k,m} \mid k \leq m \leq n\rangle$ witnessing the validity of the lemma. For each relevant $i$,  denote by $\tilde{\mu}_i$  the critical point of $\tilde{\iota}_{i, i+1}$, the ultrapower embedding of  $\tilde{U}_i$ over the $i$-th model.  We will construct by induction elementary embeddings $\tilde{k}_i$ such that $\tilde{k}_i \circ \tilde{\iota}_i = \iota_{\alpha_i}$ and moreover $\tilde{k}_i(\tilde{\mu}_j) = \mu_{\alpha_j}$, for all $j < i$. 

Since $\bar{a}$ is saturated, for every $i < n$, the measure $U_{\alpha_i}$ is represented by generators whose indexes belong to $\bar{a} \cap \alpha_i$. Let $b_i$ be this set of indexes. Let $\tilde{\mu} = \langle \tilde{\mu}_0,\dots, \tilde{\mu}_{n-1}\rangle$ (note that each time only an initial segment of this sequence is used in order to define the next step). To simplify notations, let us write $\delta := \alpha_n$.
%\ale{Pending to polish.}
For $i = 0$, let $\tilde{k}_0 = \iota_{\alpha_0}$.  Let us assume that $\tilde{k}_i, \tilde{\iota}_i$ and $\tilde{\mu}\restriction i$ are defined. 
Let $\tilde{U}_i = \tilde{\iota}_i(g_{\alpha_i})(\tilde\mu \restriction b_i)$. By the inductive hypothesis, 
\[\tilde{k}_i(\tilde{U}_i) = \iota_{\alpha_i}(g_{\alpha_i})(\mu_{\beta^{\alpha_i}_{0}},\dots, \mu_{\beta^{\alpha_i}_{m_\alpha-1}}) = U_{\alpha_i}.\]
In particular, $\tilde{U}_i$ is a normal measure on $\tilde{\mu}_i$ such that $\tilde{k}_i(\tilde{\mu}_i) = \mu_{\alpha_i}$.

Let \[\tilde{k}_{i+1}([r]_{\tilde{U}_i}) := \iota_{\alpha_i + 1, \alpha_{i+1}}(\iota_{\alpha_i,\alpha_{i+1}}(\tilde{k}_i(r))(\mu_{\alpha_i})).\]
Since $\tilde{k}_i(\tilde{U}_i) = U_{\alpha_i}$, it is standard to verify that $\tilde{k}_{i+1}$ is well defined and elementary. Moreover, $\tilde{k}_{i+1}\circ \tilde{\iota}_{i+1} = \iota_{\alpha_{i+1}}$. We use the assumption that the iteration is normal, in order to know that $\iota_{\alpha_i + 1, \alpha_{i+1}}$ does not move $\mu_{\alpha_i}$.  

Finally, $\tilde{k}_n \circ \tilde{\iota}_n = \iota_{\delta}$, so we set $k = \tilde{k}_n$.  

Let us show that $x \in \range k$. Indeed, without loss of generality,
\[x = \iota_{\delta}(f)(\mu_{\alpha_0},\dots, \mu_{\alpha_{n-1}}) = k\circ \tilde{\iota}(f)(k(\tilde \mu)) = k(\tilde{\iota}(f)(\tilde \mu)).\qedhere\]
\end{proof}
\smallskip

\begin{lemma}\label{LemmaAboutStationary}
Suppose that $\langle \iota_{\alpha,\beta}\mid \alpha\leq\beta<\delta\rangle$ is a normal iteration with critical points $\langle \mu_\alpha\mid \alpha<\delta\rangle$ and let $\iota:=\iota_\delta\colon \mathcal{M}_0\rightarrow \mathcal{M}_\delta$. %  be $\iota_\delta$. %is the direct limit of a normal iteration $\langle \iota_{\alpha,\beta}\mid \alpha\leq\beta<\delta\rangle$ of measures  , and let $\mathbb{P}\in \mathcal{M}$ be a forcing notion with the following properties:
%\begin{enumerate}
 %   \item Every condition $p\in\mathbb{P}$ is a partial function $p\colon \mu_0\rightarrow \mathcal{M}$;
  %  \item If $p\in\mathbb{P}$ then for every $\mathcal{M}$-inaccessible $\beta\leq \mu_0$, $(\dom p)\cap \beta\in \mathrm{NS}^{\mathcal{M}}_\beta;$
   % \item For a condition $p\in\mathbb{P}$ and a nowhere stationary set $T\in \mathrm{NS}^{\mathcal{M}}_{\mu_0}$ there is $q\leq p$ such that $T\s\dom q.$
%\end{enumerate}
For every $\mathcal{M}_0$-generic filter $S\s\mathbb{S}_{\mu_0}$ we have that $\bigcup_{s\in S}\iota(\dom(s))=\iota(\mu_0)\setminus \langle \mu_\alpha\mid \alpha<\delta\rangle.$
\end{lemma}
\begin{proof}
%To enlighten the exposition 
Denote $S_\iota:=\bigcup_{s\in S}\iota(\dom(s))$. % the left-hand-side of the above-displayed  equation. 
Here we just deal with the case where $\iota$ is finite, say, with critical points $\langle \mu_0,\dots,\mu_k\rangle.$ The general case is handled using that any normal iteration is a directed limit of finite iterations, as demonstrated in Lemma~~\ref{lem;representing-elements-in-direct-limit}. %Through the proof we denote $\kappa:=\mu_0$. 

The crux of the matter is the following claim:
\begin{claim}\label{claim:covering-almost-all-ordinals}
For each %inaccessible 
$\alpha \in \iota(\mu_0) \setminus \{\mu_0,\dots,\mu_k\}$ there is a nowhere statio\-nary\footnote{Namely, $T\cap \beta\in \mathrm{NS}_\beta$ for all $\beta<\mu_0$ regular.} set $T\s\mu_0$ with $\alpha \in \iota(T)$. 
\end{claim}
Using this latter fact it is easy to derive the desired conclusion:
\begin{claim}
$S_\iota = \iota(\mu_0) \setminus \{\mu_0,\dots,\mu_k\}$.
\end{claim}
\begin{proof}[Proof of claim]
Fix $\alpha\in \iota(\mu_0)\setminus \{\mu_0,\dots,\mu_{k}\}$ and put $$D_\alpha:=\{p\in\mathbb{S}_{\mu_0}\mid \alpha\in \iota(\dom(p))\}.$$
 $D_\alpha$ is dense in $\mathbb{S}_{\mu_0}$: Indeed, let $p\notin D_\alpha$ and use Claim~~\ref{claim:covering-almost-all-ordinals} in order to find a nowhere stationary set $T\s\mu_0$  such that $\alpha\in \iota(T).$  Next, find $q\leq p$ such that $T\s \dom(q)$. Thus $\alpha\in \iota(\dom(q))$ and $D_\alpha$ is dense. %Using Clause~(3) above we find $q\leq p$ with $T\s \dom q$. %Define $s':=s\cup\{\langle \beta, \beta\rangle\mid \beta\in T\setminus \dom(s)\}.$ Clearly, $s'\in D_\alpha$ extends $s$. By genericity of $S$ there is $s\in D_\alpha\cap S$ and so $\alpha\in \dom S_j.$ 

The above argument shows that $\iota(\mu_0)\setminus \{\mu_0,\dots,\mu_{k}\}\s S_\iota.$

\smallskip

Conversely, we claim that $S_\iota\cap \{\mu_0,\dots, \mu_{k}\}=\emptyset.$ For this  fix $i\leq k$ and suppose towards a contradiction that $\mu_i\in  S_\iota$. Let $p\in S$ such that $\mu_i\in \iota(T)$ where $T:=\dom p.$ Notice that $\mu_i\notin \iota_{i+1}(T)$: By elementarity,  $\iota_i(T)$ is nowhere stationary, hence %(relative to $\mathrm{Inacc}$), 
$\iota_i(T)\cap \mu_i\in \mathrm{NS}_{\mu_i}^{\mathcal{M}_i}$. Also, if  $\mu_i\in \iota_{i+1}(T)$ then normality of the iteration  would imply that $\iota_{i}(T)\cap \mu_i$ is stationary in $\mu_i.$ 

Therefore %$\mu_i\notin\iota_{i+1}(T)$ and thus 
$\mu_i\notin \iota(T)$, yielding the sought contradiction.
\end{proof}
So, let us prove Claim~\ref{claim:covering-almost-all-ordinals} above.
%\begin{proof}
%We begin proving the in particular clause, provided the claim holds.
%\begin{subclaim}
%$\dom(S_j)=(j(\kappa)\cap\mathrm{Inacc})\setminus \{\mu_0,\dots,\mu_k\}$.
%\end{subclaim}
%\begin{proof}[Proof of claim]
%Fix $\alpha\in (j(\kappa)\cap \mathrm{Inacc})\setminus \{\mu_0,\dots,\mu_{k}\}$ and put $$D_\alpha:=\{s\in\mathbb{S}\mid \alpha\in j(\dom(s))\}.$$
%We show that $D_\alpha$ is dense in $\mathbb{S}$. Indeed, let $s\notin D_\alpha$ and use the claim to find a nowehere stationary set $T\s\kappa$  such that $\alpha\in j(T).$ Define $s':=s\cup\{\langle \beta, \beta\rangle\mid \beta\in T\setminus \dom(s)\}.$ Clearly, $s'\in D_\alpha$ extends $s$. By genericity of $S$ there is $s\in D_\alpha\cap S$ and so $\alpha\in \dom S_j.$ Thus, $(j(\kappa)\cap \mathrm{Inacc})\setminus \{\mu_0,\dots,\mu_{k}\}\s \dom S_j.$
%Conversely, we claim that $\dom S_j\cap \{\mu_0,\dots, \mu_{k}\}=\emptyset.$ For this  fix $i\leq k+1$ and suppose towards a contradiction that $\mu_i\in \dom S_j$. Then, there is some $s\in S$ such that $\mu_i\in j(T)$ where $T:=\dom s.$ Notice that $\mu_i\notin \iota_{i+1}(T)$: In effect, by elementarity,  $\iota_i(T)$ is non stationary (relative to $\mathrm{Inacc}$), hence $\iota_i(T)\cap \mu_i\in \mathrm{NS}_{\mu_i}^{\mathcal{K}_i}$. However, if $\mu_i\in \iota_{i+1}(T)$,  normality of the iteration  implies that $\iota_{i}(T)\cap \mu_i$ is stationary in $\mu_i.$ So, $\mu_i\notin\iota_{i+1}(T)$ and thus $\mu_i\notin j(T)$.
%\end{proof}
%\textcolor{blue}{First, note that $\dom S_j\s j(\kappa)\setminus \{\mu_0,\dots, \mu_k\}$ because...}
Fix $\alpha<\iota(\mu_0)$ not in the critical sequence $\langle \mu_0,\dots, \mu_k\rangle$  and let us find a nowhere stationary set $T\s \mu_0$ such that $\alpha\in \iota(T)$. To avoid trivialities let us further assume that $\alpha \geq \mu_0$. First,  $\alpha = \iota(g)(\mu_0,\dots,\mu_{n-1})$ for some $g \colon \mu_0^n \to \mu_0$ and $n \leq k $. Assume  that $n$ is the minimal such index (hence $\mu_{n-1}<\alpha$).

%Consider the set $I:=\{i<n\mid \iota_{i,n}(\mu_i)>\alpha\}$. Clearly $I\neq \emptyset$, for $\mu_0\in I$.\smallskipFor each $i<n$ such that $\mu_i$ is not an image of $\kappa$  let $f_i\colon \kappa^i\rightarrow \kappa$ be such that   $\mu_i = \iota_i(f_i)(\mu_0,\dots,\mu_{i-1})$. Otherwise,  $\mu_i=\iota_i(\kappa)$ and we pick $f_i\colon \kappa^{i}\rightarrow\{\kappa\}$.\footnote{By convention, $\dom(f_0):=\{\emptyset\}$.}  %Without loss of generality $\ran (f_i)$ consists only of inaccessibles. 

\smallskip

Define $B:=\{\bar{x}\in\mu_0^n\mid \max(\bar{x})<g(\bar{x})\}$ and $T := g``B$.  Clearly, $(\mu_0,\dots,\mu_{n-1})\in \iota(B)$.
%\[\begin{matrix}
%B := 
%\{\bar{x} \in \kappa^{n}\mid & \bar{x}\;\text{is increasing}, \\
 %& i \in I \rightarrow x_i \in D^g_{\bar{x}}\;\&\; \text{$D^g_{\bar{x}}$ is a club in $f_i(\bar{x})$}, \\ 
%& i \in I \rightarrow g(\bar{x})<f_i(\bar{x}\restriction i),\\
%&  i\notin I\rightarrow g(\bar{x}) \geq  f_i(\bar{x}\restriction i),\\ & \forall i\; (x_i < g(\bar{x})\; \text{and}\; x_i < f_i(\bar{x}\restriction i))\} 
%\end{matrix} 
%\] %\ale{I think with the change in $I$ clause (4) is wrong. }
%\ale{Get rid of the $f_i$}
%\begin{subclaim}
%	$(\mu_0,\dots, \mu_{n-1})\in\iota(B)$.
%\end{subclaim}
%\begin{proof}
%	Let us %verify that $(\mu_0,\dots, \mu_{n-1})\in \iota(B)$: %by 
%	verify this by going over the clauses defining the set $\iota(B)$: 
%	The first clause holds by normality of the iteration; %the second clause is a consequence of Subclaim~\ref{SubclaimDgalpha};
%	the second and third clauses follow from the definition of $I$; finally, the fifth clause follows from minimality of $n$ together with  $\mu_i<\iota_{i,n}(\mu_i)$. %We are left with the second clause.This concludes the proof of the claim.
	%the second is true by our choice upon $\mu_i$;  the third clause is Subclaim~\ref{SubclaimAccumulation}; the fourth clause follows from minimality of $n$ and the choice of $f_i$; finally, the fifth clause also follows from minimality of $n$. 
%\end{proof}
%\{g(\bar{x}) \mid \bar{x} \in B\}$. 
 Thus,  $\alpha=\iota(g)(\mu_0,\dots, \mu_{n-1})\in \iota(T).$\label{Nowherestationary} 
 
 \smallskip

 We claim that 
$T$ is nowhere stationary. Suppose towards a contradiction that $\rho<\mu_0$ is a regular cardinal such that $T\cap \rho$ is stationary. Then every $\beta\in T\cap \rho$ is of the form $g(\bar{x})$ for some $\bar{x}\in B\cap \rho^n$. %\footnote{This is a consequence of Clauses~(1) and (4) in the definition of $B$.} 
	Let $B_\rho\s B\cap\rho^n$ such that $T\cap \rho=g``B_\rho$. Next, let $h_\rho\colon T\cap \rho\rightarrow \rho$ be given by 
	\[\beta\mapsto \min\{\max(\bar{x}) \mid g(\bar{x}) = \beta\}.\] 
	This map is well-defined and regressive, by the definition of $B$. %Also, $h_\rho$ is regressive in $T\cap \rho$ % $h_\rho(g(\bar{x}))=\max(\bar{x})<g(\bar{x})$ for all $g(\bar{x})\in T\cap \rho$, 
	Thus, there is $T_\rho\s T\cap\rho$ stationary and $\eta<\rho$ with $h_\rho``T_\rho=\{\eta\}$. This latter fact outright implies $T_\rho\s g``\eta^n$, which is not impossible in that $|g``\eta^n|=|\eta|<\rho$.

This completes the proof of %Claim~\ref{claim:covering-almost-all-ordinals} and, as a result, the proof of 
Lemma~\ref{LemmaAboutStationary}.
\end{proof}

\begin{remark}\label{remark; almost-injective-regressive-function}
The proof of Claim~\ref{claim:covering-almost-all-ordinals} shows something slightly stronger: for every ordinal $\alpha$ there is a set $T \subseteq \kappa$ such that $\alpha \in \iota(T)$ and there is a regressive function $f \colon T \to \kappa$ such that for every $\rho$, $|f^{-1}(\{\rho\})| \leq |\rho| + \aleph_0$.
\end{remark}

The next strengthening of Claim~\ref{claim:covering-almost-all-ordinals} will not be used in this section but comes handy in the proof of the future Lemmas~\ref{lem;maximal-stem} and \ref{lem;conditions-in-k-inv-H}.
\begin{lemma}\label{CoveringNowhereLemma}
	Let $\iota\colon \mathcal{M}\rightarrow \mathcal{N}$ be a finite normal iteration with critical points $\vec\mu=\langle \kappa\rangle^\smallfrown\langle \mu_i\mid 1\leq i<n\rangle$. For every nowhere stationary $T\in \mathcal{P}(\iota(\kappa))^N$ with $T\cap \vec\mu=\emptyset$  there is another nowhere stationary   $\bar{T}\in \mathcal{P}(\kappa)^\mathcal{M}$  such that $T\s \iota(\bar{T})$ 
\end{lemma}
\begin{proof}
	%Let  $\langle \mu_i\mid i<n\rangle$ denote the critical points of the iteration $\iota$, and set 
	Put $\mu_n:=\iota(\kappa)$.  Since $T\in \mathcal{P}(\iota(\kappa))^N$, $T$ decomposes as $T=\bigcup_{i<n} T_i$ where $T_i:=T\cap [\mu_i,\mu_{i+1})$. Clearly, each $T_i$ is nowhere stationary in $\mu_{i+1}$. To establish the lemma it will suffice to find, for each $i<n$, a nowhere stationary set $\bar{T}_i\in \mathcal{P}(\kappa)^\mathcal{M}$ with $T_i\s \iota_{i+1}(\bar{T}_i)$. So, fix $i<n$. %In effect, in that case  $\bar{T}:=\bigcup_{i<n} \bar{T}_i$ yields a nowhere stationary set with the required property.

\smallskip
	
		%here is $\bar{T}_i\in\mathcal{P}(\kappa)^\mathcal{M}$ nowhere stationary with $T_i\s \iota_{i+1}(\bar{T}_i)$.

	 Let $g\colon \kappa^{i+1}\rightarrow \mathcal{P}(\kappa)^\mathcal{M}$ be a map in $\mathcal{M}$ representing $T_i$; %and $f\colon \kappa^{i+1}\rightarrow \kappa$ be maps in $M$ representing $T_i$ and $\mu_{i+1}$, respectively; 
	 namely,  $$T_i=\iota_{i+1}(g)(\mu_0,\dots, \mu_i).$$ Without loss of generality, $g(\bar{x})$ is nowhere stationary for all $\bar{x}\in \kappa^{i+1}$. % and $\mu_{i+1}=\iota(f)(\mu_0,\dots, \mu_i)$. %Also, let $h\colon \kappa^{i+1}\rightarrow \mathcal{P}(\kappa)^M$ representing a club subset of $\mu_{i+1}$ disjoint from $T_i$. By elementarity we may safely assume that 

 Define $B:= \{\bar{x} \in \kappa^{i+1}\mid \max(\bar{x})<\min g(\bar{x})\}$ and $$\bar{T}_i:=\bigcup\{g(\bar{x})\setminus \max(\bar{x})+1\mid \bar{x}\in B\}.$$
%\[\begin{matrix}
%B := 
%\{\bar{x} \in \kappa^{i+1}\mid & \bar{x}\;\text{is increasing}, \\
 % & j \in I \rightarrow x_j \in C^h_{\bar{x}\restriction j}, \\ 
%& j \in I \rightarrow \min(g(\bar{x}))<f_j(\bar{x}\restriction j),\\
%&  j\notin I\rightarrow \min(g(\bar{x})) \geq  f_j(\bar{x}\restriction j),\\ & \forall j\leq i \; (\max(\bar{x}) < \min(g(\bar{x}))\; \text{and}\; x_j < f_j(\bar{x}\restriction j))\} 
%\end{matrix}\]
Clearly, $(\mu_0,\dots, \mu_{i})\in \iota_{i+1}(B)$ so that, as   $T\cap \vec\mu=\emptyset$,
%Note that	$(\mu_0,\dots, \mu_{i})\in \iota_{i+1}(B)$: The first clause holds by normality of the iteration; %the second clause is a \textcolor{green}{consequence of the argument in Subclaim~\ref{Chisclub};}
%	the second and third %clauses 
%	follow from the definition of $I$;  finally, the fourth clause is a consequence of $\mu_i\notin T_i$ and $\mu_j<\iota_{j,i+1}(\mu_j)$, respectively.
%Define $\bar{T}_i:=\bigcup\{g(\bar{x})\setminus \max(\bar{x})+1\mid \bar{x}\in B\}$.  By the above subclaim,
$$\iota_{i+1}(\bar{T}_i)\supseteq \iota_{i+1}(g)(\mu_0,\dots, \mu_{i})\setminus \mu_i+1=T_i\setminus (\mu_i+1)=T_i.$$%\footnote{This latter equality follows from our assumption that.}$$
%the latter equality being true because $\mu_i\notin T_i$. 
Also, $\bar{T}_i$ is nowhere stationary by the same argument as in Claim~\ref{claim:covering-almost-all-ordinals}.
%We are done with the proof of the lemma.
\end{proof}
%The following serves as a  warm-up for Lemma~\ref{lemma:representing-ultrafilters} and provides a complete cartography of the measures in the generic extension $\mathcal{K}[S]$:
%\begin{lemma}\label{lemma:extending-S} 
%Let us assume $V = K$ and there is no inner model with $\exists \alpha o(\alpha) = \alpha^{++}$. 
%Let $S \subseteq \mathbb{S}$ be a $\mathcal{K}$-generic filter and  $\mathcal{U}$ a $\kappa$-complete ultrafilter on $\kappa$ in $\mathcal{K}[S]$. Then, there is a finite  iteration $\langle i_{m,n}\mid m\leq n\leq k+1\rangle$ of measures in $\mathcal{K}$ with critical points $\langle \mu_0, \dots, \mu_k\rangle$, $\kappa = \mu_0 < \mu_1 < \dots < \mu_k$, a function $f \colon \kappa^{k + 1} \to \kappa$ in $\mathcal{K}$ and $\langle a_0,\dots, a_{k}\rangle \in\prod_{i\leq k} H(\mu_i^{+})$ such that
%$$\mathcal{U}=\{\dot{X}^S\s \kappa\mid\exists p\in S\,(\iota(p)\cup\{\langle \mu_i,a_i\rangle\mid i\leq k\}\forces_{\iota(\mathbb{S})} \check{\epsilon}\in \iota(\dot{X})^{S'})\},$$
%where $\epsilon=\iota_k(f)(\mu_0,\dots, \mu_{k}).$
%\textcolor{blue}{determine $\mathcal U$.}\ale{What does this exactly mean?}
%In particular, there are at most $o(\kappa) + \kappa^{+}$ many different $\kappa$-complete ultrafilters on $\kappa$ in $K[S]$.%\ale{Why? Spell out the calculation.}
%\end{lemma}
We are now ready to prove the main lemma in the section.  Recall that $\kappa$ was a measurable cardinal and $\mathbb{S}$ denoted the fast function forcing $\mathbb{S}_\kappa$.
\begin{proof}[Proof of Lemma~\ref{lemma:extending-S}]
Let $\mathcal{U} \in \mathcal{K}[S]$ be a $\kappa$-complete ultrafilter over $\kappa$ and $j \colon \mathcal{K}[S] \to M$ be the corresponding ultrapower embedding. Use  our anti-large-cardinal  hypothesis \eqref{BlanketAssumption} in page~\pageref{BlanketAssumption} and combine it with Theorem~\ref{CoreModelTheorem} (2) to infer that $j \restriction \mathcal{K}$ is an iteration of normal measures in $\mathcal{K}$. Let $\epsilon := [\id]_\mathcal{U}$.% denote the seed of the measure $\mathcal{U}.$
\begin{claim}\label{claim:warmup-S}
$j\restriction \mathcal{K}$ is a finite iteration.
\end{claim}
\begin{proof}[Proof of claim]
Observe  that $M=\mathcal{K}^{M}[S']$ for $S'\s j(\mathbb{S})$  a $\mathcal{K}^{M}$-generic filter: First, by elementarity, $M=N[S']$ for some $N$-generic filter $S'\s j(\mathbb{S})$.  Second, by   absoluteness of the core model,\footnote{See Theorem~\ref{CoreModelTheorem}(3).} $\mathcal{K}^{\mathcal{K}[S]}=\mathcal{K}$, and so $$j(\mathcal{K})=j(\mathcal{K}^{\mathcal{K}[S]})=\mathcal{K}^{M}.$$ Finally, by the Laver/Woodin theorem on definability of the ground model (see \cite{LavGroundModel,WoodinSuitableII}) $``\mathcal{K}$ is the ground model of $\mathcal{K}[S]$'' is first-order expressible in $\mathcal{K}[S]$, hence $N=\mathcal{K}^{M}$ is the ground model of $M$.%\footnote{Alternatively, one can use Hamkins' \emph{Gap Theorem} \cite{HamGap} to infer that $j$ is definable in $V$ (which is $\mathcal{K}$).}%\yair{I think that it also follows from Hamkin's gap theorem: $\mathbb{S}$ admits a gap, so $j$ is definable in $V$ (which is $\mathcal{K}$).}

We are now in conditions to argue that $j\restriction \mathcal{K}$ is a finite iteration. Clearly, $M$ is closed under $\kappa$-sequences in $\mathcal{K}[S]$, hence also closed under $\omega$-sequences. Since $j(\mathbb{S})$ is $\omega_1$-closed it must be the case that ${}^\omega{}M\cap \mathcal{K}[S]\s \mathcal{K}^M$. However, if $j\restriction \mathcal{K}$ is an infinite iteration then the sequence of its first $\omega$-many critical points $\vec{x}=\langle \mu_n\mid n<\omega\rangle$ belongs to $M^\omega\cap \mathcal{K}[S]$ but not to $\mathcal{K}^M$.% which is impossible as $\vec{x}\in M^\omega\cap \mathcal{K}[S].$
\end{proof}
%Let us argue that this iteration must be finite. Indeed, $M$ is closed under $\omega$-sequences, but if the iteration is infinite then the $\omega$-sequence of the first $\omega$ critical points is missing from $\mathcal{K}^M$. But (using elementarity), $M = \mathcal{K}^M[S']$, namely a generic extension of $\mathcal{K}^M$ using a $\sigma$-closed forcing notion and thus this sequence cannot be introduced to $M$. 
So, let $\iota=\langle \iota_{m,n}\colon\mathcal{\mathcal{K}}_m\rightarrow \mathcal{K}_n \mid m \leq n \leq k+1\rangle$  be this iteration. For each $i\leq k$ put $\mu_i:=\crit(\iota_{i,i+1})$. By Theorem~\ref{CoreModelTheorem}(2) ${\iota}$ is a  normal iteration.  %\ale{Why without loss of generality? Isn't it by Mitchell?}%: namely, $\mu_{i}<\mu_{i+1}$ for $i\leq k$. 
%Let $f \colon \kappa^{k + 1} \to \kappa$ be a function representing $\epsilon$ in the iteration; to wit, $\iota_{k+1}(f)(\mu_0,\dots,\mu_{k})=\epsilon.$ 

We would like to analyse the $\mathcal{K}^M$-generic filter $S'$, or equivalently, the function $\bigcup S'$. Let us consider $$S_{\iota} := \bigcup \{\iota(\dom(s)) \mid s \in S\}.$$ 
%Clearly, $\mathbb{S}$ satisfies the assumptions of 
By Lemma~\ref{LemmaAboutStationary} we have that $S_{\iota}=\iota(\kappa)\setminus \{\mu_0,\dots,\mu_{k}\}.$
Since $\mathcal{U}$ is the measure induced by $j$ we have that $$\mathcal{U}=\{\dot{X}_S\s \kappa\mid\exists q\in S'\,(q\forces_{j(\mathbb{S})} \check{\epsilon}\in j(\dot{X})_{S'})\}.$$ 
Let $f\colon\kappa^{k+1}\rightarrow \kappa$ be such that $\iota(f)(\mu_0,\dots,\mu_k)=\epsilon$. 
Using Lemma~\ref{LemmaAboutStationary} we find $\langle a_0,\dots, a_k\rangle\in \left(\prod_{i\leq k}H(\mu^+_i)^M\right)$ such that for every condition $q \in S'$ there is $p\in S$, $j(p)\cup \{\langle \mu_{i_j}, a_{i_j}\rangle\mid j\leq k\}\leq q$. Indeed, note that these $a_i$ are nothing but the values of the generic $S'$ at $\mu_i$ (i.e., $a_i=s(a_i)$ for $s\in S'$).

\smallskip

Therefore,% $S'=\{q\in j(\mathbb{S})\mid \exists p\in S\,\exists j\leq k\, $
$$\mathcal{U}=\{\dot{X}_S\s \kappa\mid\exists p\in S\,(j(p)\cup\{\langle \mu_i,a_i\rangle\mid i\leq k\}\forces_{j(\mathbb{S})} \check{\epsilon}\in j(\dot{X})_{S'})\},$$ 
%\textcolor{blue}{Finally, for the in particular clause ...} \ale{Complete}
%is the collection of all sets $X \subseteq \kappa$ of the form $\dot{X}^S$ such that $ \iota_k(f)(\mu_0,\dots,\mu_k) \in \iota_k(\dot{X})^{S'}$
which completes the proof.
\end{proof}
The following useful observation arises from the previous proof:
\begin{lemma}\label{lem;lifting-S}
Suppose $j\colon \mathcal{K}[S] \to M$ is an arbitrary elementary embedding such that $\iota = j \restriction \mathcal{K}$ is a normal iteration of measures in $\mathcal{K}$, say $\langle \iota_{\alpha,\beta} \mid \alpha \leq \beta \leq \delta\rangle$.\footnote{Note that $M$ might not be closed under countable sequences.}   Then, there is an iteration $\langle j_{\alpha,\beta} \mid \alpha\leq\beta\leq\gamma\rangle$ such that each embedding $j_{\alpha,\beta}$ lifts the corresponding embedding $\iota_{\alpha,\beta}$. \end{lemma}
\begin{proof}
Let us prove by induction that each embedding  $\iota_{\alpha,\alpha+1} \colon \mathcal{K}_\alpha \to \mathcal{K}_{\alpha + 1}$ extends to an embedding $j_{\alpha, \alpha+1} \colon \mathcal{K}_{\alpha}[j(S) \restriction \mu_{\alpha}] \to \mathcal{K}_{\alpha + 1}[j(S) \restriction \mu_{\alpha + 1}]$. 

By the arguments of Lemma \ref{lemma:extending-S}, in order to lift this embedding, we must pick a value for the point $\crit(\iota_{\alpha, \alpha + 1})$ in the generic filter for $\iota_{\alpha+1}(\mathbb{S})$. This value fully determines a generic for $\iota_{\alpha+1}(\mathbb{S})$ which contains  $\iota_{\alpha,\alpha+1}``j_{\alpha}(S)$. 
\end{proof}
Lemma~\ref{lemma:extending-S} was essentially a warm-up exercise for the forthcoming lemma  characterizing $\kappa$-complete measures in $\mathcal{K}[S]$-generic extensions  by non-statio\-nary support iterations of $\mathcal{U}$-Tree Prikry forcings. This lemma will replace the nice-name-argument used in Theorem \ref{omegagluingnonoptimal} when claiming that the Laver function captures the $\omega$-sequence of $\mathbb{P}_\kappa$-names $\langle \dot{U}_n\mid n<\omega\rangle.$

Roughly speaking, Lemma~\ref{lemma:representing-ultrafilters} shows that every measure $\mathcal{U}\in \mathcal{K}[S\ast G]$ admits a \emph{code} in $H(\kappa^+)$, hence in the range of the  generic Laver function given by $S$. This is necessary, as in the present context $\mathbb{P}_\kappa$ is not $\kappa$-cc nor there are  $(\kappa+2)$-strong embeddings capturing the names for the ultrafilters.

%Most of this section is dedicated to its proof.
\begin{lemma}[Coding Lemma]\label{lemma:representing-ultrafilters} 
Let $\mathbb{P}_\kappa=\langle  \mathbb{P}_{\alpha}, \mathbb{Q}_{\beta}\mid \alpha<\beta < \kappa\rangle$ be a non-stationary-supported iteration of $\mathcal{U}$-Tree Prikry forcings in $\mathcal{K}[S]$. Assume that,  for each $\alpha<\kappa$, the iteration has the following properties:
\begin{enumerate}
    \item $|\mathbb{P}_\alpha| \leq 2^\alpha$ %and $\mathbb{Q}_\alpha$ is trivial unless $\alpha$ is strongly inaccessible;
    and $\one \Vdash_{\mathbb{P}_{\alpha}} ``\langle \mathbb{Q}_{\alpha},\leq^*\rangle \text{ is }\alpha\text{-closed}$'';
    \item $\one \Vdash_{\mathbb{P}_{\alpha}} ``\forall p,q\in\mathbb{Q}_\alpha$ compatible $p\wedge q$ exists''.\footnote{Here $p\wedge q$ stands for the $\leq$-greatest condition $r\leq p,q,$ if such a condition exists.}
\end{enumerate}
%the blanket assumptions of Lemma~\ref{lem;prikry property in iterated ultrapower} with respect to the Prikry order $\leq^*_{\mathbb{P}_\kappa}$.

For each generic filter $G\s \mathbb{P}_\kappa$  there are $\kappa^{+}$-many $\kappa$-complete ultrafilters in $\mathcal{K}[S][G]$. Moreover, for each such ultrafilter $\mathcal{U}\in \mathcal{K}[S][G]$ there are
\begin{enumerate}
    \item[$(\alpha)$] a finite sub-iteration $\iota\colon \mathcal{K}[S] \to \mathcal{K}^M[\iota(S)]$ of $j_\mathcal{U}\restriction\mathcal{K}[S]$,\footnote{Hence, a lifting of a sub-iteration of $j_{\mathcal{U}}\restriction\mathcal{K}$ (by Lemma~\ref{lem;lifting-S}).} %$\iota\colon \mathcal{K}[S] \to \mathcal{K}^M[\iota(S)]$ of $j_\mathcal{U}\restriction\mathcal{K}[S]$,\footnote{Hence, a lifting of a sub-iteration of $j\restriction\mathcal{K}$ (by Lemma~\ref{lem;lifting-S}).}
     \item[$(\beta)$] an ordinal $\bar{\epsilon} < \iota(\kappa)$ with $\bar{\epsilon}\in \range(k)$,\footnote{As usual, $k$ is the factor emebdding such that $j_{\mathcal{U}}=k\circ \iota$.}
%\ale{Note the change: It used to say that $r \in \iota(\mathbb{S}\ast\mathbb{P}_\kappa)$.} 
\item[$(\gamma)$] $r \in \iota(\mathbb{P}_\kappa)$ with finite support compatible with all conditions in $\iota\image G$. Moreover, for every $p\in G$, $\iota(p)\wedge r$ exists, and
\end{enumerate}
working in $\mathcal{K}[S]$, for each $p\in G$, $p\forces_{\mathbb{P}_\kappa}``\dot{X}\in \dot{\mathcal{U}}$'' if and only if there is a condition $q\in \iota(\mathbb{P}_\kappa)$  such that
\[(k(q)\in j_{\mathcal{U}}(G)\,\&\, q\leq^ *\iota(p) \wedge r\,\&\, \supp(q)=\supp(\iota(p)\wedge r)\,\&\, q\Vdash_{\iota(\mathbb{P}_\kappa)}  \bar{\epsilon} \in \iota (\dot{X})).\]%\ale{Note the change. I think it's important to say that $k(q)$ belongs to $H$.}
%\ale{Don't we need to take here the lifting of $\iota$?}
%where $\Vdash^{**}$ means here there is a direct extension with the same support forcing the statement.
\end{lemma}
%\begin{remark}
During the proof of the lemma we  will be able to describe $r$ explicitly, based on the finite iteration $\iota$ (see Definition~\ref{definitionriota}). Nevertheless, we do not know whether in general $r$ can be computed from the iteration $\iota$ itself. %(or, at least, its $\mathbb{P}_\kappa$-part).
%\end{remark}

We postpone the proof of this lemma to the end of \S\ref{Section;proof of coding} (see p.~\pageref{proofofcoding}).

\subsection{Interlude on non-stationary-supported iterations}\label{SectionInterlude}
In the road to prove Lemma~\ref{lemma:representing-ultrafilters} we will need some abstract results about non-stationary-supported iterations. Essentially, we want a lemma saying the following: Given a finite normal iteration $\iota\colon \mathcal{M}\rightarrow \mathcal{N}$, a non-stationary supported iteration $\mathbb{P}\in \mathcal{M}$, and a dense open set $D\in\mathcal{N}$  for $\iota(\mathbb{P})$,  it suffices to perform a finite modification of the conditions in $\iota`` G$  to enter $D$. 

To make this precise we introduce the following convenient notation:

\begin{definition}
Let $\mathbb{P}_\kappa=\langle \mathbb{P}_{\alpha}, \mathbb{Q}_{\beta} \mid \beta < \alpha < \kappa\rangle$ be an iteration of forcings. For $p, q \in \mathbb{P}_\kappa$ and $\Gamma \subseteq \kappa$, write $q \leq_\Gamma p$ if $q \leq p$ and 
\[\{\alpha<\kappa \mid q\restriction\alpha \not\Vdash_{\mathbb{P}_\alpha} \dot{p}(\alpha) = \dot{q}(\alpha)\} \subseteq \Gamma.\]
%For two compatible conditions $p,q\in \mathbb{P}_\kappa$  we denote by $p\wedge q$ the $\leq$-greatest lower bound for $p$ and $q$, provided this latter exists.
\end{definition}
\begin{remark}
Notice that for $p,q\in\mathbb{P}_\kappa$, if $q \leq_\Gamma p$ then $\supp q \subseteq \supp p \cup \Gamma$.
\end{remark}

We will eventually apply the forthcoming lemma to a non-stationary-supported iteration of $\mathcal{U}$-Tree Prikry forcings, endowed with its direct extension order. We refer to \S\ref{sectionPrikrytype} where we introduced all the pertinent concepts.

%For an exposition on non-stationary-supported iterations of Prikry-type forcings we refer the reader to Ben-Neria and Unger's paper \cite{BenUng}.
\begin{lemma}\label{lem;prikry property in iterated ultrapower} %\ale{Should it be $\iota(p)\cup r$ or $\iota(p)\wedge r$?}
Let $\mathbb{P}_\kappa=\langle \mathbb{P}_{\alpha}, \mathbb{Q}_{\beta} \mid \beta < \alpha < \kappa\rangle$ be a non-stationary-supported iteration satisfying, for each $\alpha<\kappa$, the following properties: %$\one \Vdash_{\mathbb{P}_{\alpha}} ``\mathbb{Q}_{\alpha} \text{ is }\alpha\text{-closed}$'', $\mathbb{Q}_\alpha$ is trivial unless $\alpha$ is strongly inaccessible and $|\mathbb{P}_\alpha| \leq 2^\alpha$, for all $\alpha<\kappa$. We shall also assume that  for every pair $p,q\in\mathbb{P}_\kappa$ of compatible conditions $p\wedge q$ exists.
\begin{enumerate}
 \item $|\mathbb{P}_\alpha| \leq 2^\alpha$ and $\mathbb{Q}_\alpha$ is trivial unless $\alpha$ is inaccessible;
    \item $\one \Vdash_{\mathbb{P}_{\alpha}} ``\langle \mathbb{Q}_{\alpha},\leq\rangle \text{ is }\alpha\text{-closed}$'';
    \item $\one \Vdash_{\mathbb{P}_{\alpha}} ``\forall p,q\in\mathbb{Q}_\alpha$ compatible $p\wedge q$ exists''.
%    \ale{Maybe we want this for two conditions where, say, $p$ has finite support? As the condition $r$ has.}
\end{enumerate}
Let $\iota \colon \mathcal{M} \to \mathcal{N}$ be a finite iteration of normal measures and consider  $\langle \iota_{k, \ell} \mid k \leq \ell \leq n\rangle$  the collection of all sub-iterations. Put $$\text{$\mu_i := \crit(\iota_{i, i +1})$ and  $\Gamma :
= \{\iota_{k, \ell}(\mu_k) \mid k \leq \ell \leq n\}$}.$$

Let $D \in \mathcal{N}$ be a dense open subset of $\iota(\mathbb{P}_\kappa)$ and $r\in \iota(\mathbb{P}_\kappa)$  a condition with $\supp r \subseteq \Gamma$. 
%\textcolor{red}{Let us assume, in addition, that $p \in \mathbb{P}_\kappa$ is  a condition such that, for each $q \leq p$, $r$ is compatible with $\iota(q)$.}\ale{\tiny Why don't  defining $E=\{p\in \mathbb{P}_\kappa\mid (\iota(p)\perp r)\,\vee\, \exists q\in D\, (q\leq^*_{\Gamma} \iota(p)\wedge r)\}$?} %We assume moreover that $\iota(q) \wedge r$ exists and satisfies $\iota(q) \wedge r \leq_\Gamma \iota(q)$.
Then, there is a dense open set $E \subseteq \mathbb{P}_{\kappa}$ such that for every $p\in E$, either $\iota(p)$ is incompatible with $r$ or the following two hold: 
\begin{itemize}
    \item[$(\aleph)$] $\iota(p) \wedge r$ exists (hence $\iota(p) \wedge r \leq_\Gamma \iota(q)$);
    \item[$(\beth)$] {there is $q \leq_{\Gamma} \iota(p) \wedge r$ such that $q\in D$.}
\end{itemize}
\end{lemma}
\begin{proof}
Before addressing the proof of the lemma we need some preliminary considerations. % that deal with operations on conditions in an iteration.
Given  $u,v\in \mathbb{P}_\kappa$ with disjoint supports we write
\[(u \cup v)(\alpha) := \begin{cases} u(\alpha), & \alpha \in \supp u; \\
v(\alpha), & \alpha \in \supp v. \end{cases}\] 
\label{remark:unions-and-partitions-of-supports}
%We claim that $u\cup v$ is a condition. To see this %In order to see that $u \cup v$ is a condition, 
%one needs to verify that its support is nowhere stationary and that $\one\Vdash_{\mathbb{P}_\alpha} (u \cup v)(\alpha) \in \dot{\mathbb{Q}}_\alpha$ for all $\alpha$. Indeed, the union of two nowhere stationary sets is nowhere stationary. The second condition is also true as $u(\alpha)$ and $v(\alpha)$ are always forced by the trivial condition to be a condition, and $(u \cup v)(\alpha)$ is exactly one of them.
Clearly, $u\cup v\in\mathbb{P}_\kappa.$  Similarly, if $u\in\mathbb{P}_\kappa$  and $X \cup Y \supseteq \supp u$ with $X, Y$  disjoint, we can define $u \restriction X, u \restriction Y$ in the natural way so that $$u = (u \restriction X) \cup (u \restriction Y).$$

Note that this definition makes sense as we are using Jech's definition of iteration in which a condition in an iteration satisfies that every coordinate of it is forced by the trivial condition to be a condition in the corresponding step of the iteration.

%We need yet a few more terminology. For two conditions $u,v\in\mathbb{P}_\kappa$ we say that \emph{$u$ agrees with $v$ on coordinates in $X$} if $\one\Vdash_{\mathbb{P}_\beta} u(\beta) = v(\beta)$ for all $\beta\in X$.

\begin{claim}\label{Claimdensity}
    There is a $\leq_{\iota(\kappa)\setminus \Gamma}$-dense open set $D'$  such that for every condition $q' \in D'$, there is yet another condition $q'' \leq_\Gamma q'$ in $D$.
\end{claim}
\begin{proof}
Enumerate $\Gamma$ as $\{\gamma_0, \dots, \gamma_{m-1}\}$ and put $\gamma_{-1}:=0$ and $\gamma_m:=\iota(\kappa).$

One can regard $\iota(\mathbb{P}_\kappa)$ as a finite iteration of the forcings $\iota(\mathbb{P}_\kappa)\restriction[\gamma_{n-1},\gamma_n)$'s. By standard facts about the term-space forcing there is a canonical projection from $\mathbb{A}({\one},\mathbb{R}_0)\times \left(\prod_{n=0}^{m-1}\mathbb{A}(\mathbb{R}_{n},\iota(\mathbb{P}_\kappa)_{\gamma_n})\right)\times \mathbb{A}(\mathbb{R}_m,\{\one\})$ to $\iota(\mathbb{P}_\kappa)$, where $\mathbb{R}_n$ is the iteration $\iota(\mathbb{P}_\kappa)\restriction[0,\gamma_{n})$.\footnote{See, e.g., \cite[Proposition~2.25]{CumHandBook}.} % for $i\leq n-1$. 
In particular, there is  $D_{-1}\times (\prod_{n=0}^{m-1}D_n)\times D_m$ which is dense and open in  this product and that it is contained in $D$.\footnote{Formally speaking, $D$ contains the image of $D_{-1}\times (\prod_{n=0}^{m-1}D_n)\times D_m$ under the canonical projection.}

Let $q'\leq_{\iota(\kappa)\setminus \Gamma}q$. By induction on $n\in [-1,m]$ define  $r\leq_{\iota(\kappa)\setminus \Gamma} q'$ such that $q'':=r\restriction \gamma_0^\smallfrown\langle s(\gamma_0)\rangle^\smallfrown r\restriction (\gamma_0,\gamma_1)^\smallfrown \langle s(\gamma_1)\rangle^\smallfrown \cdots^\smallfrown r\restriction (\gamma_{m-1},\gamma_m)$ is a condition where $r\restriction\gamma_0\in D_{-1}$, $\langle s(\gamma_n)\mid n<m\rangle\in\prod_{n=0}^{m-1}D_n$ are forced to extend $r(\gamma_n)$ and $r\restriction (\gamma_{m-1},\gamma_m)\in D_m$. This is possible by density and openess of the $D_n$'s. Clearly, $q''\leq_{\Gamma} r$ and $q''\in D$. Therefore there is a $\leq_{\iota(\kappa)\setminus \Gamma}$-dense open set $D'$ as in the claim.
\end{proof}

\begin{comment}

\begin{claim}
    Let $q$ be a condition.
    
    There is a $\leq_{\iota(\kappa)\setminus \Gamma}$-dense open set $D'$ below $q$ such that for every $q' \in D'$, there is a condition $q'' \leq_\Gamma q'$ in $D$.
\end{claim}
\begin{proof}
Since $\Gamma$ is finite, let us enumerate it $\{\gamma_0, \dots, \gamma_{m-1}\}$. 

By Easton Lemma, $D$ contains a dense open set of the form $D_0 \times D_1 \times \cdots \times D_m$, where $D_i$ is a dense open set of the termspace forcing $\iota(\mathbb{P}) \restriction [\gamma_{i-1}, \gamma_i)$ where $\gamma_{-1} = 0, \gamma_m = \iota(\kappa)$\footnote{formally, $D$ contains the image of this set under the canonical projection.}. 

Let us deal with each $D_i$ separately. Let us  argue for $D_0$. The other components are identical. 

Assume that there is $s \leq_{\gamma_0} q \restriction \gamma_0 + 1$, such that for every $s' \leq_{\gamma_0} s$ there is no $r \leq_{\{\gamma_0\}} s'$ in $D_0$. This is obviously impossible --- just pick any condition in $D_0$ below $s$ and decompose it into first extending the $\gamma_0$ part and then the last coordinate. 
\end{proof}
\end{comment}

By the claim, it is enough to show that for every $\leq_{\iota(\kappa)\setminus \Gamma}$-dense open set $D\in \mathcal{N}$ the set $E\s \mathbb{P}$ of  all conditions $q$ such that either $\iota(q)$ is incompatible with $r$ or $\iota(q)\wedge r \in D$, is dense: Indeed, fix $D\in \mathcal{N}$ and $r\in \iota(\mathbb{P}_\kappa)$ be as in the lemma. Invoking Claim~\ref{Claimdensity} with respect to $D$ we get in return a $\leq_{\iota(\kappa)\setminus\Gamma}$-dense open set $D'$ with the above-mentioned property.  Using our assumption we have that the set $E\s \mathbb{P}$ of all $q$'s such that either $\iota(q)\perp r$ or $\iota(q)\wedge r\in D'$ is dense.  By the properties of $D'$ there is a condition $q'\leq_\Gamma \iota(q)\wedge r$ in $D$, so we are done.

\smallskip

%In order to make the inductive hypothesis run smoothly, we are going to tweak it slightly.  

We will eventually prove the lemma by induction on the length of the iteration. Let us first deal with the case of a single ultrapower. In order to make the inductive hypothesis run smoothly we are going to tweak it slightly. Specifically, the set $\Gamma$ of the forthcoming proof will contain the critical points of the embedding,  their images and possibly other  cardinals.

So, let $\iota \colon \mathcal{M} \to \mathcal{N}$ be a normal ultrapower embedding with critical point $\mu$ and let $\Gamma'$ be a finite set of {inaccessible cardinals.} Put $\Gamma := \iota(\Gamma') \cup \{\mu, \iota(\mu)\}$. Let $\mathbb{P}_\kappa$ be our non-stationary-supported iteration,  $r \in \iota(\mathbb{P}_\kappa)$ with $\supp r=\Gamma$ and $D\s\iota(\mathbb{P}_\kappa)$, $D\in\mathcal{N}$, a dense open set with respect to $\leq_{\iota(\kappa)\setminus\Gamma}$.

Proving the claim suffices to infer the first step of the induction.
\begin{claim}\label{KeyClaim}
Let $E$ be the collection of all conditions $q$ such that 
\begin{enumerate}
    \item either there is $\alpha \in \Gamma$ such that $$\iota(q)\restriction\alpha\forces_{\iota(\mathbb{P}_\kappa)\restriction\alpha}\text{$``\iota(q)(\alpha)$ is incompatible with $r(\alpha)$''}$$% $\iota(q) \restriction \alpha \not\Vdash_{\iota(\mathbb{P}_\kappa)_\alpha} r(\alpha)\leq \iota(q)(\alpha)$\ale{$\iota(q)\restriction\alpha$ or $r\restriction\alpha$?}
    \item or $\iota(q) \wedge r \leq_\Gamma \iota(q)$ and $\iota(q) \wedge r\in D$.
\end{enumerate}
%either $\exists \alpha \in \Gamma,\, \iota(q) \restriction \alpha \not\Vdash \iota(q)(\alpha) \geq r(\alpha)$ or this does not hold (so $\iota(q) \wedge r \leq_\Gamma \iota(q)$) and there is $q' \leq_\Gamma \iota(q) \wedge r$ in $D$.  

Then $E$ is $\leq_{\kappa\setminus(\Gamma' \cup \{\mu\})}$-dense open.
\end{claim}
%Note that we cannot quite have $E$ to be open, because of the additional assumption that $\iota(q) \wedge r \leq_\Gamma \iota(q)$ (and even the assumption that $\iota(q) \wedge r$ exists), is not open. Nevertheless, we will see that by picking $q$ carefully, we can verify that the compatibility with $r$ is decided by a bounded part of the condition. 
 %Suppose that  $q\in \mathbb{P}_\kappa$  is such that $\iota(q)$ is compatible with $r$. In that case, $E\cap (\mathbb{P}_\kappa/p)$ is a dense open set $\leq$-below $p$ consisting of conditions $q$ for which clause (2) above holds. %Note that this is precisely the thesis of the lemma.

\begin{proof}[Proof of Claim~\ref{KeyClaim}]
 $E$ is open in $\leq_{(\kappa\setminus \Gamma' \cup \{\mu\})}$ so we shall show it is dense.

Fix $d$ and $\tilde{r}$ be functions with domain $\mu$ representing $D$ and $r$, respectively; i.e., $\iota(d)(\mu) = D$ and $\iota(\tilde{r})(\mu) = r$. Without loss of generality, we may assume that, %Fix a function $d$ with domain $\mu$, representing the dense open set $D$ in the ultrapower (i.e.\ $\iota(d)(\mu) = D$). Similarly, let $\tilde{r}$ representing the condition $r$, so $\iota(\tilde{r})(\mu) = r$. Let us assume that 
for each $\alpha < \mu$, $\supp \tilde{r}(\alpha) = \Gamma' \cup \{\alpha, \mu\}$ and that $d(\alpha)$ is $\leq_{\kappa\setminus \Gamma'\cup\{\alpha,\mu\}}$-dense open. In what follows $q$ will be a fixed condition in $\mathbb{P}_\kappa$. 

We  inductively define a $\leq_{\kappa\setminus \Gamma'\cup\{\mu\}}$-decreasing sequence $\langle q_\alpha\mid \alpha\leq \mu\rangle\s \mathbb{P}_\kappa$ below $q$ together with a continuously-decreasing sequence of clubs $\langle C_\alpha\mid \alpha<\mu\rangle$ on $\mu$ %\subseteq \mu$, for $\alpha < \mu$, 
such that $(\supp q_\alpha)\cap C_\alpha= \emptyset$.  For each $\alpha<\mu$ we  stipulate that $$\gamma_\alpha := \min (C_\alpha \setminus (\max \Gamma' \cup \sup_{\beta<\alpha} \gamma_\beta ) + 1).$$
Our mission is to find a dense subset of conditions $e(\gamma_\alpha)\s \mathbb{P}_{\gamma_\alpha + 1}$ with the following property:   for all $s \in e(\gamma_\alpha)$ such that $s \leq q_{\alpha} \restriction \gamma_\alpha + 1$, put $q'_s := s \cup q_{\alpha} \restriction [\gamma_{\alpha} + 1, \kappa)$. Then,
\begin{enumerate}
    \item[$(\alpha)$] either %$q'_s = s \cup q_{\alpha} \restriction [\gamma_{\alpha} + 1, \kappa)$ 
    $q'_s$ is incompatible with $\tilde{r}(\gamma_\alpha)$,
    \item[$(\beta)$] or  $q'_s \wedge \tilde{r}(\gamma_\alpha)$ exists, $q'_s \wedge \tilde{r}(\gamma_\alpha)\leq_{\Gamma' \cup \{\gamma_\alpha, \mu\}} q'_s$ and $q'_s \wedge \tilde{r}(\gamma_\alpha) \in d(\gamma_\alpha)$. 
    % there is $q''\in d(\gamma_\alpha)$ with $\text{$q''\leq_{\kappa\setminus (\Gamma'\cup \{\gamma_\alpha,\mu\})}q'_s\wedge\tilde{r}(\gamma_\alpha)$.}$\ale{${\leq_{\kappa\setminus (\Gamma'\cup \{\gamma_\alpha,\mu\})}}$ or $\leq_{\Gamma'\cup\{\gamma_\alpha,\mu\}}$?}% {In addition,  there is  $r'_\alpha \leq_\Gamma \tilde{r}(\gamma_\alpha)$ that is compatible with $q'$ and $$q''=q' \wedge r'_\alpha \in d(\gamma_\alpha).$$} 
    %\ale{$\leq_{\Gamma'\cup\{\gamma_\alpha,\mu\}}$ or $\leq_{\kappa\setminus \Gamma'\cup\{\gamma_\alpha,\mu\}}$?}
\end{enumerate}
We will inductively maintain  that $q_\alpha \restriction \gamma_\alpha + 1= q_\beta \restriction \gamma_\alpha + 1$ for all $\alpha < \beta$. In particular, $\gamma_\alpha$ will never be in the support of $q_\beta$, as $(\supp q_\alpha)\cap C_\alpha=\emptyset$. Moreover, we will assume that $q_\alpha$ agrees with $q$ on $\Gamma' \cup \{\mu\}$. %(see the paragraph before Claim~\ref{Claimdensity}).   
We proceed with the construction of $q_\alpha$ in a fusion-like fashion. 

\smallskip

$\br$ Set $q_0 := q$ and let $C_0\s \mu$ be any club disjoint from $\supp q_0$.

\smallskip

$\br$ Assume that $\langle q_\alpha \mid \alpha < \beta\rangle$ and $\langle C_\alpha \mid \alpha < \beta\rangle$ were constructed % the conditions $\langle q_\alpha \mid \alpha < \beta\rangle$ 
for some $\beta > 0$. If $\beta$ is limit then, by our inductive assumption, there is a condition $q^*_\beta$ that serves as a lower bound for $\langle q_\alpha\mid \alpha < \beta\rangle$ and whose support is disjoint from $\bigcap_{\alpha < \beta} C_\alpha$. Note that this latter is a club for all $\beta < \mu$. Likewise, for $\beta = \mu$, the assumption on freezing initial segments ensures that this lower bound exists and that its support is disjoint from $\diagonal_{\beta<\mu} C_\alpha.$

%If $\beta = \alpha + 1$, let $q^*_\beta = q_\alpha$..

\smallskip

 Let $\beta < \mu$ be arbitrary. Let us show how to define $q_\beta$ together with the dense set $e(\gamma_\beta)$.  Let us consider $\langle s_\zeta \mid \zeta < \zeta_*\rangle$, an enumeration of all the conditions in $\mathbb{P}_{\gamma_\beta + 1}$ which are below $q^*_\beta \restriction \gamma_{\beta} + 1$.  By our departing assumption upon $\mathbb{P}_\kappa$, $\zeta_* \leq 2^{\gamma_\beta}$, which  is below the degree of closure of $\mathbb{P}_\kappa/\mathbb{P}_{\gamma_\beta+1}.$ This enables us to define a decreasing sequence $\langle q^\beta_\zeta\mid \zeta < \zeta_*\rangle$ as follows:% and with $q^\beta_\zeta\restriction\gamma_\beta+1=q^*_\beta\restriction\gamma_\beta+1$, as follows: 

$\br$ If $s_0\cup q^*_\beta\restriction[\gamma_\beta+1,\kappa)$ is incompatible with $\tilde{r}(\gamma_\beta)$ then stipulate $q^\beta_0:=q^*_\beta$. Otherwise, $(s_0\cup q^*_\beta\restriction[\gamma_\beta+1,\kappa))\wedge \tilde{r}(\gamma_\beta)$ exists and we can $\leq_{\kappa\setminus \Gamma'\cup\{\gamma_\beta,\mu\}}$-extend it to a condition $q''\in d(\gamma_\beta).$ In this case define $$q^\beta_0:=(q^*_\beta\restriction\gamma_\beta+1)\cup (q''\restriction[\gamma_\beta+1,\kappa)).$$

$\br$ In general, let $r_\eta$ be a lower bound for $\langle q^\beta_\zeta\mid \zeta<\eta\rangle$ and proceed exactly as before with respect to the condition $s_\eta\cup r_\eta\restriction[\gamma_\beta+1,\kappa)$.

\smallskip

%in a way that all of them agree up to $\gamma_\beta + 1$. 

%Specifically,  for each $\zeta$, if there is $q^\beta_\zeta$ below all conditions in $\{q^*_\beta\} \cup\{q^\beta_\xi \mid \xi < \zeta\}$, that agrees with $q^*_\beta$ below $\gamma_\beta + 1$, and $q'=q_\zeta' = s_\zeta \cup q^\beta_{\zeta} \restriction [\gamma_\beta + 1, \kappa)$ satisfies $(\alpha)$ and $(\beta)$, then pick it.

%i.e.\ either $q'$ is incompatible with $\tilde{r}(\gamma_\beta)$ or they are  compatible as witnesses $q' \wedge \tilde{r}(\alpha) \leq_{\Gamma' \cup \{\gamma_{\beta}, \mu\}} q'$, and there is an extension $q'' \leq_{\Gamma' \cup \{\gamma_{\beta}, \mu\}} q'$, $q''\in D$. 

%\smallskip

%Clearly, every condition $u$ can be extended to a condition $u'$ which is either incompatible with $\tilde{r}(\gamma_\beta)$ or a condition such that $u'\wedge \tilde{r}(\gamma_\beta) \leq_{\Gamma' \cup \{\gamma_\beta, \mu\}} u'$. In the second case, by further extending $u'$ we can make it in $d(\gamma_\beta)$. Thus, the collection of all conditions $s_\zeta$ for which $q''$ exists is dense. Let $e(\gamma_\beta)$ be the collection of such $s_\zeta$ together with all the conditions in $\mathbb{P}_{\gamma_\beta + 1}$ which are incompatible with $q^*_\beta$.

Finally, let $q_\beta$ be the condition \[q^*_\beta \restriction \left(\gamma_{\beta}+1 \cup \Gamma'\cup\{\mu\}\right) \cup q^\beta_{\zeta_*} \restriction \left([\gamma_\beta + 1, \kappa) \setminus \Gamma'\right),\]
and $e(\gamma_\beta)$ be the  dense open set defined as
\[\begin{matrix}
\{s\mid s\perp (q^*_\beta\restriction\gamma_\beta+1)\}  \cup \\ \{s\mid\exists \zeta (s=s_\zeta\;\wedge\; s\,\cup\, q^\beta_\zeta\restriction [\gamma_\beta+1,\kappa) \text{ witnesses $(\alpha)$ or $(\beta)$})\}. \end{matrix}\]

%\textcolor{blue}{Let us now look at $q_\mu$. By our inductive construction, $q_\mu\leq_{\kappa\setminus (\Gamma'\cup\{\mu\})} q.$ We wish to argue that $q_\mu\in E$. If (1) of Claim~\ref{KeyClaim} holds we are done. Otherwise, $\iota(q_\mu)$ is compatible with $r$ and $\iota(q_\mu)\wedge r\leq_\Gamma \iota(q_\mu)$. In this case, set $s:=(\iota(q_\mu)\wedge r)\restriction\mu+1$. Since $\iota(q_\mu)\restriction\mu+1=\iota(q)_\mu\restriction\mu+1$ we have that $s\leq \iota(q)_\mu\restriction\mu+1$.  By our inductive construction and elementarity of $\iota$, either $s\cup \iota(q)_\mu\restriction[\mu+1,\iota(\kappa))$ is incompatible with $r$ or it satisfies $(\beta)$. The former is impossible, so the latter should hold; namely, there is $q'\leq_{\Gamma}s\cup \iota(q)_\mu\restriction[\mu+1,\iota(\kappa))$ in $D$.}

Let us apply $\iota$ on the sequence $\langle q_\beta \mid \beta < \mu\rangle$, and let $\iota(q)_\mu$ be the $\mu$-th element of the obtained sequence. By elementarity, for densely many $s \in \iota(e)(\mu)$, which are below $q_\mu \restriction \mu+1 = \iota(q)_\mu \restriction \mu+1$, the condition $s \cup \iota(q_\mu) \restriction [\mu+1, \iota(\kappa))$ satisfies the requirements of the claim (here we use the fact that $\iota(q_\mu)$ is stronger than $\iota(q)_\mu$).  But $s \leq q_\mu \restriction \mu + 1$ and therefore $s \cup q_\mu \restriction [\mu, \kappa) \leq q$ belongs to $E$. 
\end{proof}
 As mentioned earlier we will prove the lemma by induction on the length of the iteration $\iota\colon \mathcal{M}\rightarrow \mathcal{N}$. Recall that it is enough to show it for $\leq_{\iota(\kappa)\setminus \Gamma}$-dense open sets. %We prove the lemma by induction on the length of the iteration. 
So assume by induction that the previous claim holds for iterations of length $n - 1$. Let $\iota\colon \mathcal{M}\rightarrow \mathcal{N}_n$ be an iteration of length $n$, $D\in \mathcal{N}$ a $\leq_{\iota(\kappa)\setminus \Gamma}$-dense open set and $r\in\iota(\mathbb{P}_\kappa)$ with $\supp r\s \Gamma$.  Let $\mu$ be the last critical point of $\iota$  and denote by $\iota_{n-1} \colon \mathcal{M}\rightarrow \mathcal{N}_{n-1}$ and $\iota_{n-1,n} \colon \mathcal{N}_{n-1} \to \mathcal{N}_{n}$  the first $n-1$ ultrapowers  and the last ultrapower of $\iota$, respectively.%, respectively. % a measure $U$ on $\mu := \mu_n$.\footnote{Note that $\mu$ can be $\iota_{n-1}(\kappa)$.} 

Let $\Gamma_{n-1}$ be the collection of critical points of $\iota_{n-1}$ together with their images. Then $\Gamma=\iota_{n-1,n}(\Gamma_{n-1})\cup \{\mu,\iota_{n-1,n}(\mu)\}$. Applying Claim~\ref{KeyClaim} to $\iota_{n-1,n}$ we get a $\leq_{\iota_{n-1}(\kappa)\setminus \Gamma_{n-1}\cup\{\mu_{n-1}\}}$-dense open set $E'$ with the above-stated properties. Next we invoke our induction hypothesis to $\iota_{n-1}$, $E'$ and $r\restriction\mu_{n-1}$
and find a dense open $E$ satisfying the conclusion of Claim~\ref{KeyClaim}.

Let $q\in E$ % If $\iota_{n-1}(q)$ 
be a condition such that $\iota(q)\restriction \mu_{n-1}$ is compatible with $r\restriction\mu_{n-1}$. %\footnote{Recall that we were assuming the existence of a condition $p\in\mathbb{P}_\kappa$ such that $\iota(q)$ is compatible with $r$ for all $q\leq p.$ Since $E$ is dense, any condition in $E\cap (\mathbb{P}_\kappa/p)$ works.}
Then, $\iota_{n-1}(q)\wedge r\restriction\mu_{n-1}$ is $\leq_{\Gamma_{n-1}}$-below $\iota_{n-1}(q)$ and there is  $q'\in E'$ with $$q'\leq_{\Gamma_{n-1}}\iota_{n-1}(q)\wedge r\restriction\mu_{n-1}.$$ 

Note that 
$$\iota_{n-1,n}(q')\leq_{\iota_{n-1}(\Gamma_{n-1})}\iota_{n-1,n}(\iota_{n-1}(q)\wedge r\restriction\mu_{n-1})=\iota(q)\wedge r\restriction\mu_{n-1},$$
hence $\iota_{n-1,n}(q')$ is {compatible} with $r$. Thus, by the property of $E'$, there is $q''\leq_{\Gamma} \iota_{n-1,n}(q')\wedge r$ in $D$. Since 
$\iota_{n-1,n}(q')\wedge r\leq_\Gamma\iota_n(q)\wedge r$ this yields 
a condition $q''\leq_\Gamma \iota(q)\wedge r$, which accomplishes the proof. \qedhere%in $\iota(\mathbb{P}_\kappa)$ admitting a $\leq_\Gamma$-extension $q'\in D$.\qedhere
\end{proof}

\subsection{Proof of the Coding Lemma}\label{Section;proof of coding}
After a short digression with non-stationary-supported iterations we focus on proving Lemma~\ref{lemma:representing-ultrafilters}. %Before we proceed with the proof, we need a strengthening of Claim \ref{claim:covering-almost-all-ordinals}To guide the reader 

\begin{setup}
For the rest of this section our ground model $V$ is $\mathcal{K}[S]$ and $\mathbb{P}_\kappa$ is a non-stationary-supported iteration (in $\mathcal{K}[S]$)  as described  in Lemma~\ref{lemma:representing-ultrafilters}. 

We will fix both a $V$-generic filter $G\s\mathbb{P}_\kappa$  and $j\colon V[G]\to M[H]$, a $V[G]$-elementary embedding with $\crit(j)= \kappa$. Notice that, by virtue of our anti-large-cardinal assumption (\eqref{BlanketAssumption} in page~\pageref{BlanketAssumption}), $j\restriction\mathcal{K}$ is a normal iteration and thus so is $j\restriction V$ (by Lemma~~\ref{lem;lifting-S}) - but those iterations can be infinite.\footnote{In fact, they must be quite long, as the forcing $j(\mathbb{P}_\kappa)$ adds many $\omega$-sequences} Recall that, by Lemma~\ref{lem;representing-elements-in-direct-limit}, for each $x\in M$, there is a finite iteration $\iota\colon V\rightarrow N$ such that $k\circ \iota=j\restriction V$ and $x\in \mathrm{range}(k)$, where $k\colon N\rightarrow M$ is the factor map between $\iota$ and $j$ (we will say that \emph{$\iota$ factors $j\restriction V$}).
\end{setup}
Recall that our forcing $\mathbb{P}_\kappa$ is a (non-stationary-supported) iteration of $\mathcal{U}$-tree Prikry forcings. Thus, a typical condition $p\in \mathbb{P}_\kappa$ is a sequence of the form $\langle p(\alpha)\mid \alpha\in \supp(p)\rangle$, where $p(\alpha)$ is a $\mathbb{P}_\alpha$-name for a condition of the form $\langle \dot{t}_\alpha, \dot{T}_\alpha\rangle.$  Since $G\restriction\alpha:=G\cap\mathbb{P}_\alpha$ is $V$-generic we can interpret both $\dot{t}_\alpha$ and $\dot{T}_\alpha$, and this interpretation gives rise to a stem $t_\alpha$ and a tree $T_\alpha$. 
\begin{notation}
    Given $p\in\mathbb{P}_\kappa$ and $\alpha\in \supp(p)$, $\stem(p(\alpha))_{G\restriction\alpha}:=(\dot{t}_\alpha)_{G\restriction\alpha}.$
\end{notation}

The current section is conceptually divided in two parts.  In the first part (Lemmas~\ref{ClosureUnderkappaseq} -- \ref{AsubsetGamma}) we shall work towards identifying the condition $r$ mentioned in Clause~$(\gamma)$ of the \emph{Coding Lemma}. In the second part, we will prove Lemma~\ref{MainLemma2} which will readily yield the coding lemma.

\smallskip

Let us begin with a technical lemma:

\begin{lemma}\label{ClosureUnderkappaseq}
%Assume $2^\kappa = \kappa^+$ holds in $V$. 
Suppose that $\iota \colon V \to N$ is a finite normal iteration with critical point %$\vec{\mu}=\langle \mu_0,\dots,\mu_{n-1}\rangle$ and that $\min\vec\mu\geq\kappa$. 
at least $\kappa$. %  and let $G\s\mathbb{P}_\kappa$ be $V$-generic.  
 Then, $N[G]$ is closed under ${<}\kappa$-sequences in $V[G]$. %of length ${<}\kappa$. 
 
 Moreover, for every such sequence $\vec{\alpha}$  
 there is $\zeta<\kappa$ such that $\vec{\alpha}\in N[G\restriction\zeta].$
\end{lemma} 
\begin{proof}
As usual, let $\vec\mu$ be the sequence of critical points. Since $\min\vec\mu \geq \kappa$, $V_{\kappa+1} \subseteq N$ and thus $\mathbb{P}_{\kappa}$ is computed correctly in $N$. In particular, the model $N[G]$ is well-defined and it is the generic extension of $N$ by $G$.

Let $\langle \alpha_i \mid i < i_\star\rangle$ be a sequence of ordinals in $V[G]$ with $i_\star<\kappa$. For each $i<i_\star$ fix a
function $V\ni f_i \colon \kappa^n \to \mathrm{Ord}$ such that $\alpha_i = \iota(f_i)(\vec{\mu})$. Everything boils down to show that $\langle f_i \mid i < i_\star\rangle$ belongs to $V[G \restriction \gamma]$ for some $\gamma < \kappa$. 

Invoking the Axiom of Choice in $V$ we let $h\colon\dom(h)\rightarrow V$, $h\in V$, with $\dom(h)\in \ord$ and $f_i \in \range(h)$ for all $i<i_\star$. This choice is possible in that each $f_i$ belongs to the ground model. Note that a priori, we do not put any bound on $\dom h$. %(recall that each one of them is from $V$). 
Put $a_\star:=\{h^{-1}(f_i)\mid i<i_\star\}$ and note that this  is a set of ordinals in $V[G]$ whose order-type is $<\kappa$.
%So, we would like to argue that every sequence of ordinals of length $<\kappa$ in $V[G]$ belongs to $V[G \restriction \gamma]$ for some $\gamma < \kappa$.   
\begin{claim}
Let $a$ be a set of ordinals in $V[G]$ with $\otp a < \kappa$. Then, there is $\zeta < \kappa$ and a function $g \colon \zeta \to \Ord$ in $V$ such that $a \subseteq \range g$.
\end{claim}
\begin{proof}
%First, since $|\mathbb{P}| \leq 2^\kappa$, 
Suppose (for simplicity of forthcoming notations) that both $\otp(a)$ and $\sup(a)$ are decided by the trivial condition of $\mathbb{P}_\kappa$. 
First, we claim that there is a function $\bar{g}\colon \mathbb{P}_\kappa \times \otp a\rightarrow\sup a$ in $V$ such that $a \subseteq \range\bar{g}$. Indeed, let $\bar{g}$ for example the map sending each pair $\langle p, \xi\rangle \in \mathbb{P}_\kappa\times \otp a$ %and ordinal $\xi < \otp a$ 
to the $\xi$-th element of $a$ decided by $p$, if this exists; otherwise, set $\bar{g}(p,\xi):=0$. Since $|\mathbb{P}_\kappa| = 2^\kappa$, by well ordering $\mathbb{P}_\kappa$, we may assume that $$\dom(\bar{g})= (2^\kappa)^V = (\kappa^{+})^V.\footnote{Recall that $V=\mathcal{K}[S]$. Thus, by Proposition~\ref{PropertiesofS}, the GCH holds in this model.}$$

A fusion argument akin to the one of \cite[Corollary~2.6]{BenUng} shows that $\mathbb{P}_\kappa$ forces  $(\kappa^+)^V = (\kappa^{+})^{V[G]}$, hence $(\kappa^+)^V$ is regular in $V[G]$. Thus, $\sup a < \kappa^+$. In particular, there must be some  $\zeta' < \kappa^+$ such that $\range (\bar{g}\restriction \zeta')$ already covers $a$. Note that this  $\zeta'$ might be determined only in  $V[G]$. 

Next, let us pick a surjection $\varphi \colon \kappa \to \zeta'$ in $V$. Then, $\bar{g}\circ\varphi \colon \kappa \to \sup a$ covers $a$. But $\kappa$ is regular in $V[G]$ as well, so we apply the same argument and obtain some $\zeta < \kappa$ such that $g:=\bar{g} \circ \varphi \restriction \zeta$ already covers $a$. 
\end{proof}
Using the above claim we have that $a_\star\s \range g$ for  some $g\colon \zeta\rightarrow V$ in $V$ with $\zeta<\kappa.$ Thus, in order to decide the value of $a_*$, it is enough to decide the value of $g^{-1}(a_*)$ which is a subset of $\zeta$.

%Using the claim, we may assume without loss of generality, that $a$ is a subset of $\zeta$ for some $\zeta < \kappa$, since computing $a$ and computing $\langle f_i \mid i < i_\star\rangle$ is equivalent modulo information from $V$. 
Using the closure of the $\leq^*$-order, we may decide the values of $a_\star$ up to the forcing $\mathbb{P}_{\zeta + 1}$. This is achieved by repeatedly strengthening the components of the condition above $\zeta$. For details, see  Claim~\ref{kappabounding} or Claim~\ref{KeyClaim}.

%Using the closure of those components, this is possible.
%Working back in $V[G]$, define $a := \{h^{-1}(f_i) \mid i < i_*\}$. Note that $a\s \zeta$, hence %is a bounded subset of $\kappa$ (bounded by $\zeta$) and thus 
% $a\in V[G \restriction \zeta + 1]$. Let ${\tau}$ be a $\mathbb{P}_{\zeta + 1}$-name such that ${\tau}_{G \restriction \zeta + 1} = a$. So, $a \in N[G\restriction\zeta+1]$. Moreover, since $\crit(\iota)\geq \kappa$ the iteration extends to an elementary embedding  $\iota\colon V[G \restriction \zeta + 1] \to N[G\restriction \zeta + 1]$. Thus,  we get that $$\iota(h)\image a = \{\iota(f_i) \mid i < i_\star\}\in N[G\restriction\zeta+1],$$
% as this set is definable via the parameters $\iota(h), a\in N[G\restriction\zeta+1]$. Applying the functions appearing in $\iota(h)``a$ on $\vec\mu$ we recover $\langle \alpha_i\mid i<i_*\rangle.$
 %and therefore by applying those functions on $\vec\mu$, we obtained the required set of ordinals.
 Since $a_\star\in V[G\restriction\zeta+1]$ and $h\in V$ it follows that $h``a_\star=\{f_i\mid i<i_\star\}$ belongs to $V[G\restriction\zeta+1]$, hence it also does  $\langle f_i \mid i < i_\star\rangle$. Now, the embedding $\iota$ lifts uniquely (as the forcing $\mathbb{P}_{\zeta+1}$ is small) to an embedding $\iota_*$. Thus, $\iota_*(\langle f_i \mid i < i_\star\rangle) = \langle \iota(f_i) \mid i < i_\star\rangle \in N[G \restriction \zeta + 1] \subseteq N[G]$. By plugging the generators $\vec{\mu}$ to this sequence  we infer that $\langle \alpha_i\mid i<i_\star\rangle\in N[G].$
\end{proof}

The next will help us isolate the condition $r$ from  the coding lemma: %Lemma~\ref{lemma:representing-ultrafilters}$(\gamma)$:

\begin{lemma}\label{lem;maximal-stem}
 % such that $j^*:=j\restriction V$ is a normal iteration using normal measures
Let $\iota \colon V \to N$ be a finite iteration and $k\colon N\rightarrow M$ be such  that $k\circ \iota=j\restriction V.$  Then, the set $A_\iota$ defined as %\ale{I need to understand better the intention of this definition.} 
%\[A := \{\alpha\in [\kappa,\iota(\kappa)) \mid \exists r \in k^{-1}(H)\, \forall p\in G\; \varphi(\alpha,r, \iota(\mathbb{P}), p)\}\]
%\[\{\alpha\in [\kappa,\iota(\kappa)) \mid \exists r\in k^{-1}(H)\,\forall p\in G\, (k(r(\alpha))_{H\restriction k(\alpha)}<j(p)(k(\alpha))_{H\restriction k(\alpha)})\}\footnote{Here $p<q$ stands for $p\leq q$ and $|\mathrm{stem}(p)|>|\mathrm{stem}(q)|$.}\]
\begin{align*}
{\{\alpha\in [\kappa,\iota(\kappa)) \mid }
 &{\,\exists r\in k^{-1}(H)\;
   \forall p\in G\;
   \exists q \in 
   (\dot{\mathbb{Q}}_\alpha)_{H \restriction k(\alpha)}}
 \\[4pt]
 &{\big(
   q \leq^* k(r(\alpha))_{H\restriction k(\alpha)}
   \;\wedge\;
   q < j(p)(k(\alpha))_{H\restriction k(\alpha)}
  \big)
\}\footnote{In other words, we are looking at those coordinates $\alpha$ for which there is a condition in $k^{-1}(H)$ whose stem at coordinate $k(\alpha)$ is stricly longer than the stems of $j(p)(\alpha)_{H\restriction k(\alpha)}$ for all $p\in G$ with $k(\alpha)\in \supp(j(p))$}.}
\end{align*}
is finite. 
\smallskip
%\ale{Pending to polish}  
Moreover, for each $\alpha\in A_\iota$ 
%\[\bar{H}_\alpha = \{k(\dot{\eta})_{H\restriction k(\alpha)} \mid \exists r \in k^{-1}(H)\, \exists p\in G\, (\text{$\dot{\eta}$ witnesses $\varphi(\alpha,r,\iota(\mathbb{P}),p)$})\}\]
\[\bar{H}_\alpha = \{\mathrm{stem}(k(r(\alpha)))_{H\restriction k(\alpha)} \mid \text{$r$ witnesses $\alpha\in A_\iota$}\}\]
%\ale{What generic interprets $\dot{\eta}$? It should be a generic for $\iota(\mathbb{P})$...}
is finite and has $\bigcup \bar{H}_\alpha$ as a maximal element. %are finite for all $\alpha \in A$.
\end{lemma}

\begin{proof}
Let us show that $\bar{H}:=\bigcup_{\alpha\in A} \bar{H}_\alpha$ is a finite set. Assume towards a contradiction that $\bar{H}$ is infinite  and let $\langle k(\dot{\eta}_n)_{H\restriction k(\alpha_n)}\mid n<\omega\rangle\in V[G]$ be a injective enumeration of this set. Note that $\dot\eta_n$ is an $\iota(\mathbb{P}_\kappa)_{\alpha_n}$-name for a finite sequence of ordinals. 

\smallskip

Put $\vec{\alpha}:=\langle \alpha_n\mid n<\omega\rangle\in V[G]$. 
Note that we allow $\vec{\alpha}$ to exhibit repetitions, although  $\bar{H}$ is assumed to be infinite.

For each $n<\omega$, let $r_n\in k^{-1}(H)$ be a condition such that $$\mathrm{stem}(k(r_n(\alpha_n)))_{H\restriction k(\alpha_n)}=k(\dot{\eta}_n)_{H\restriction\alpha_n}\text{ and }$$ 
%$$\textcolor{blue}{\stem(j(p))(k(\alpha_n))\sqsubseteq k(\dot{\eta}_n)_{H\restriction \alpha_n}\;
$$ \forall p \in G,\,\;k(\dot{\eta}_n)_{H\restriction \alpha_n}\nsubseteq \stem(j(p))(k(\alpha_n)).$$
   
Since $k(r_n)\in H$ it follows that  $k(\dot{\eta}_n)_{H\restriction k(\alpha_n)}$ is an initial segment of the $k(\alpha_n)$-Prikry sequence introduced by $H$.

\smallskip

By the proof of Lemma~\ref{ClosureUnderkappaseq}, there is $\zeta<\kappa$ such that $$\langle \dot{\alpha}_n\mid n<\omega\rangle, \langle \dot{\eta}_n\mid n<\omega\rangle\in N[G\restriction\zeta].$$ Note that $k$ lifts to $k\colon N[G\restriction\zeta]\rightarrow M[G\restriction\zeta]$ as $\crit(k)\geq \kappa$. Ergo  $$\langle k(\alpha_n)\mid n<\omega\rangle, \langle k(\dot{\eta}_n)\mid n<\omega\rangle\in M[G\restriction\zeta]\s M[H].$$ 
%$$\langle j(p_n)\mid n<\omega\rangle\in M[H]$$ 
because $M[H]$ is closed under $\omega$-sequences in $V[G]$. %\ale{Revise}

\smallskip

Appealing to the forcing theorem we find $q\in H$ such that, for each $n<\omega$,
\begin{equation}\label{eq1}
    q\restriction k(\alpha_n)\forces_{j(\mathbb{P})\restriction k(\alpha_n)}\text{$``k(\dot{\eta}_n)$ is an initial segment of $\mathrm{stem}(q(k(\alpha_n)))$''}.
\end{equation}
and
\begin{equation}\label{eq2}
    q\restriction k(\alpha_n)\forces_{j(\mathbb{P})\restriction k(\alpha_n)}\text{$``\forall p\in \dot{G}\; \big(k(\dot{\eta}_n)\nsubseteq \mathrm{stem}(j(p)(k(\alpha_n)))\big)$''}.
\end{equation}
%$q\in H$ such that
%The latter implies that for all $n<\omega$
%\begin{equation}\label{eq3}
%   q\restriction k(\alpha_n)\forces_{j(\mathbb{P})\restriction k(\alpha_n)}\text{$``k(\dot{\eta}_{n})\sqsubseteq \mathrm{stem}(\dot{q}(k(\alpha_n)))$''}. 
%\end{equation}
We now produce a contradiction by distinguishing two cases: either $\vec{\alpha}$ is finite or it is not. In the first case there is   $\alpha_*\in\vec{\alpha}$ and $I\in[\omega]^{\omega}$ such that $\alpha_*=\alpha_n$ for all $n\in I.$ In particular $\langle k(\dot{\eta}_n)_{H\restriction k(\alpha_*)}\mid n\in I\rangle $ is bounded by $\dot{\mathrm{stem}}(\dot{q}(k(\alpha_*)))_{H\restriction k(\alpha_*)}$ and thus  there are $n,m\in I$ such that $k(\dot{\eta}_n)_{H\restriction k(\alpha_*)}=k(\dot{\eta}_m)_{H\restriction k(\alpha_*)}$. This contradicts   injectivity of $\langle k(\dot{\eta}_n)_{H\restriction k(\alpha_n)}\mid n<\omega\rangle.$

\smallskip

Thus, assume that $\vec{\alpha}$ is infinite.

\begin{claim}
    There is $p_*\in G$ such that $\vec{\alpha}\setminus \vec{\mu}\s \dom(\iota(p_*))$.% for all $n<\omega$.
\end{claim}
\begin{proof}[Proof of claim] %\ale{Revise}
Since $\vec{\alpha}\setminus \vec{\mu}\in V[G]$ there is $\zeta<\kappa$ with $\vec{\alpha}\setminus \vec{\mu}\in N[G\restriction\zeta]$. A chain condition argument with $\mathbb{P}_\zeta$ allows us to find a set $B\s [\kappa, \iota(\kappa))$, $|B|<\kappa$  such that $\one \forces_{\mathbb{P}_\zeta}\dot{\vec{\alpha}}\setminus \vec\mu\s \check{B}$ and $B\cap \vec\mu=\emptyset$. By Lemma~\ref{CoveringNowhereLemma} there must be a condition $p_*\in G$ with $B\s \dom(\iota(p_*))$. Thus, $p_*$ is as wished. % \kappa%Then, $\alpha_n\in \dom(\iota(p_*))$ for all $n<\omega$. 
\end{proof}
%\ale{Can't we assume that this is forced by $q$?}Let $q_*\leq q$ be in $H$ such that \textcolor{blue}{$q_*\restriction k(\alpha_n)\forces_{j(\mathbb{P})\restriction k(\alpha_n)}k(\dot{\eta}_{n,p_*})=k(\dot{\eta}_n)$} for all $n<\omega$. 
Let $s\leq q,j(p_*)$ be in $H$. Then (since $\vec\mu$ is finite) for some ${n_*}<\omega$ $$s\restriction k(\alpha_{n_*})\forces_{j(\mathbb{P})\restriction k(\alpha_{n_*})} s(\alpha_{n_*})\leq^* j(p_*)(k(\alpha_{n_*}))$$ and $$s\restriction k(\alpha_{n_*})\forces_{j(\mathbb{P})\restriction k(\alpha_{n_*})} s(\alpha_{n_*})\leq^* q(k(\alpha_{n_*})).$$
Combining this with \eqref{eq1} and \eqref{eq2} above we have $$s\restriction k(\alpha_{n_*})\forces_{j(\mathbb{P})\restriction k(\alpha_{n_*})} k(\dot{\eta}_{n_*})\sqsubseteq \dot{\mathrm{stem}}(j(p_*)(k(\alpha_{n_*}))),$$
and
 $$s\restriction k(\alpha_{n^*})\forces_{j(\mathbb{P})\restriction k(\alpha_n)}\text{$k(\dot{\eta}_{n_*})\nsubseteq \mathrm{stem}(j(p_*)(k(\alpha_{n_*})))$},$$
 which is impossible.% as $s\in H.$

%which implies that  $k(\dot{\eta}_{n_*})_{H\restriction k(\alpha_{n_*})}\sq \mathrm{stem}(j(p_*)(k(\alpha_{n_*})))_{H\restriction k(\alpha_{n_*})}$. This contradicts item~(b) above, so that $\vec{\alpha}$ cannot be infinite.

% so we obtain the desired contradiction whenever $|A|\geq\aleph_0$.% is infinite.
\smallskip

All in all we have shown that $\bar{H}$ is finite and thus so are $A_\iota$ and the $\bar{H}_\alpha$'s. Finally, we argue that $\bigcup \bar{H}_\alpha$ is the maximal element of $\bar{H}_\alpha:$

\begin{claim}
Every two members of $\bar{H}_\alpha$ are $\sqsubseteq$-compatible.
\end{claim}
\begin{proof}[Proof of claim]
Suppose $k(\dot{\eta})_{H\restriction\alpha}, k(\dot{\sigma})_{H\restriction\alpha}\in \bar{H}_\alpha$ and let $r, r'\in k^{-1}(H)$ witnessing this. Then there is $H\ni s\leq k(r),k(r')$ such that $$s\restriction k(\alpha)\forces_{j(\mathbb{P})\restriction k(\alpha)}\text{$``k(r(\alpha))$ and $k(r'(\alpha))$ are compatible''},$$ which yields $ k(\dot{\eta})_{H\restriction k(\alpha)}\sq k(\dot{\sigma})_{H\restriction k(\alpha)}$ or the other way around.
\end{proof}
This is the end of the proof.
\end{proof}

By the previous lemma we can associate to each finite iteration $\iota\colon V\rightarrow N$ (factoring $j\restriction V$) a finite list of names for stems $\vec{\eta}:=\langle \dot{\eta}_\alpha\mid \alpha\in A_\iota\rangle$ where $k(\dot{\eta}_\alpha)_{H\restriction k(\alpha)}$ is the maximal element of $\bar{H}_\alpha$.  We may assume, in addition, that $\dot{\eta}_\alpha$ is forced by the weakest condition of $\iota(\mathbb{P}_\kappa)_\alpha$ to be a potential stem at coordinate $\alpha$. This makes the following to be a sound definition: \begin{definition}\label{definitionriota}
Let $\iota\colon V\rightarrow N$ be a finite iteration factoring $j\restriction V.$ 

We  denote by $r_\iota$ the condition with $\supp r_\iota= A_\iota$ such that $r(\alpha)$ is a $\iota(\mathbb{P}_\kappa)_\alpha$-name for a maximal tree\footnote{Namely, $r(\alpha)=\langle \dot{\eta}_\alpha, \dot{T}_\alpha\rangle$ being $\dot{T}_\alpha$ a $\mathbb{P}_\alpha$-name for the tree in $\kappa^{<\omega}$ all of whose members either end-extend $\dot{\eta}_\alpha$ or are a restriction of it.} with stem $\dot{\eta}_\alpha$ and $A_\iota$ being as in Lemma~\ref{lem;maximal-stem}.
\end{definition}

 \begin{remark}
Note that $r_\iota\in \iota(\mathbb{P}_\kappa)$ is a condition such that $k(r_\iota)\in H$. Indeed, working in $M[H]$, $k(r_\iota)$ is the unique condition with (finite) support $k(A_\iota)$ such that for each $\alpha\in A_\iota$, $k(r_\iota)(k(\alpha))=\langle k(\dot{\eta}_\alpha)_{H\restriction k(\alpha)},k(\dot{T}_\alpha)_{H\restriction k(\alpha)}\rangle$. Let $p\in H$ be a condition forcing this. Then, it must be that $p\leq k(r_\iota)$ for otherwise $p$ will not be able to decide that piece of the generic sequence induced by $H$. By upwards closure of $H$ we deduce that $k(r_\iota)\in H.$
\end{remark}

%to define a condition $r_\iota$ with $\supp r_\iota = A_\iota$ such that $r(\alpha)$ is a $\iota(\mathbb{P})_\alpha$-name for a maximal tree with stem $\dot{\eta}_\alpha$. 
%For every condition $p\in G$, $\iota(p)$ is compatible with $r$ (since $j(p), k(r) \in H$), and thus we can look at the condition $q=\iota(p)^+ \wedge r$, which is defined recursively by $q\restriction \alpha \Vdash q(\alpha) = \iota(p)^+(\alpha) \wedge r(\alpha)$.
%\begin{lemma}
%Let $p\in G$ be a condition and let $q \leq^* \iota(p) \wedge r$, with $\supp q = \supp p \cup \supp r$ and $q$ is a finite modification of $\iota(p) \cup r$. 
%Let $r' \in k^{-1}(H)$. Then $r'$ is  compatible with $q$. 
%\end{lemma}
%\begin{proof}
%Work in $M$. For each coordinate of $k(q) \leq^* j(p) \wedge k(r)$,  
%\end{proof}

\smallskip

The following  gives better control on the support of $r_\iota$: %first claim  of Lemma~\ref{lem;maximal-stem}:

\begin{lemma}\label{AsubsetGamma}
Let $\iota\colon V\rightarrow N$ be a finite iteration factoring $j\restriction V.$

Using the notations of Lemmas \ref{lem;prikry property in iterated ultrapower} and \ref{lem;maximal-stem}, $A_\iota \subseteq \Gamma_\iota$. In particular, the condition  $r_\iota$ complies with the requirements of Lemma~\ref{lem;prikry property in iterated ultrapower}.
\end{lemma}
\begin{proof}
Assume otherwise and let $\alpha < \iota(\kappa)$ be such that $\alpha \in A_\iota$ but $\alpha \notin \Gamma_\iota$. 
Let $\tau$ be a $\iota(\mathbb{P}_\kappa)_\alpha$-name for the ordinal $\max \dot\eta_\alpha$. We may assume  that this name is forced by the trivial condition to be strictly below $\alpha$.
%By the definition, $\dot\tau$ is a $\iota(\mathbb{P})\restriction \alpha$-name for an ordinal below $\alpha$. 
Let $D$ be the set of all conditions $q\in \iota(\mathbb{P}_\kappa)$ with $\alpha\in \supp q$ such that 
\[\one \Vdash_{\iota(\mathbb{P}_\kappa)\restriction\alpha} \min(\mathrm{Lev}_1(T^{q(\alpha)}))> \tau.\]
%where $\succ q(\alpha)$ is the large set in the first level of the tree for $q(\alpha)$. 
An easy application of the mixing lemma shows that $D$ is $\leq^*$-dense open in $\iota(\mathbb{P}_\kappa)$. The  goal is to find a condition $p \in G$ with $\iota(p) \in D$.

Recall that for $q,q'\in \iota(\mathbb{P}_\kappa)$ we write $q'\leq^*_{\Gamma_\iota} q$ if $q'\leq^* q$ and $$\{\beta<\iota(\kappa)\mid q'\restriction\beta\nVdash_{\iota(\mathbb{P}_\kappa)\restriction\beta}q'(\beta)= q(\beta)\}\s {\Gamma_\iota}.$$ %Define $D:=\{q\in \iota(\mathbb{P}_\kappa)\mid \forall q'\leq^*_{\Gamma_\iota} q\, (q'\in D_0)\}$. It follows that $D$ is $\leq^*$-dense  open in $\iota(\mathbb{P}_\kappa)$. {Also, if $q\in \iota(\mathbb{P}_\kappa)$, $\alpha\in \supp q$ and $D\ni q'\leq^*_{\Gamma_\iota} q$  then $q\in D_0$.}\footnote{Because we are assuming that $\alpha\notin \Gamma_{\iota}.$}

%For every condition $q$ with $\supp q \cap \Gamma = \emptyset$ there is $q'\leq^* q$  such that for every modification $q''$ of $q'$ at coordinates $\Gamma$ we have $q'' \leq q$ and  $q'' \in D_0$. This is done by going over all possible extensions in coordinates $\Gamma \cap \alpha$. Let $D$ be the set of such $q'$. %\ale{I'll try to expand this with some details.}Once again, $D$ is $\leq^*$-dense open. \textcolor{blue}{Note that if $q$ is a condition, and there is an extension of it at coordinates $\Gamma$ in $D$ then $q\in D$.} \ale{Clarify} 

As $k(r_\iota)\in H$, we can take a condition $p_0\in G$ forcing it. Note that $p_0$ forces in particular that $\iota(p)$ is compatible with $r_\iota$ for every $p \in G$.

{Using Lemma \ref{lem;prikry property in iterated ultrapower}}, the set
$$E:=\{p\leq p_0\mid (\iota(p)\perp r_\iota)\,\vee \exists q\in D\, (q\leq^*_{\Gamma_\iota}\iota(p)\wedge r_\iota)\}$$
is open and $\leq^*$-dense  below $p_0$. % with the following property: for every $p\in E$ there is $D\ni q'\leq^*_{\Gamma_\iota}\iota(p)\wedge r$. {Thus, if $\alpha\in \supp\iota(p)$ then $\iota(p)\in D_0$.} 

Note also that $E':=\{p\in \mathbb{P}_\kappa\mid \alpha\in \iota(\supp(p))\}$ is $\leq^*$-dense:  %We note that $\alpha \in \supp \iota(p)$ for $\leq^*$-densely many $p\in E$: 
by Claim~\ref{claim:covering-almost-all-ordinals}, as $\alpha \notin {\Gamma_\iota}$ (and in particular $\alpha$ is not one of the critical points of the iteration), there is a nowhere dense set $T$ such that $\alpha \in \iota(T)$. Thus, if $\supp p \supseteq T$, $\alpha \in \dom \iota(p)$. Clearly, the collection of conditions $p$ with support $\supp p$ containing a given nowhere stationary set $T$ is $\leq^*$-dense and  open.

\smallskip

 {Since $E$ and $E'$ are dense open below $p_0\in G$  they have dense intersection. Let $p\in G\cap E\cap E'$. Then, there is $p\in E'$ such that $q\leq^*_{\Gamma_\iota} \iota(p)\wedge r_\iota$. Since $\alpha\notin \Gamma_\iota$ it follows that $q\restriction\alpha\forces ``q(\alpha)=\iota(p)(\alpha)"$ and since $\one\forces ``\min(\dot{T}^{q(\alpha)})>\tau"$ then the same is true about the tree of $\iota(p)(\alpha)$. }
 Now use that $\alpha\in A_\iota$ to find $q\in (\dot{\mathbb{Q}}_\alpha)_{H\restriction k(\alpha)}$ with $q\leq^* k(r_\iota(\alpha))_{H\restriction k(\alpha)}$ such that %$r\in k^{-1}(H)$ such that  
 $q<j(p)(k(\alpha))_{H\restriction k(\alpha)}$. %$$r \restriction \alpha \Vdash_{\iota(\mathbb{P}_\kappa)\restriction\alpha} r(\alpha) \leq^* \iota(p)^+(\alpha) \cat \dot{\eta}_\alpha.$ 
On one hand, 
$$r_\iota\restriction\alpha\Vdash_{\iota(\mathbb{P}_\kappa)\restriction\alpha}\max(\mathrm{stem}(r_\iota(\alpha)))\leq\tau.$$ On the other hand, since $\iota(p)\in D_0$, $$r_\iota\restriction\alpha\Vdash_{\iota(\mathbb{P}_\kappa)\restriction\alpha}\max(\mathrm{stem}(r_\iota(\alpha)))>\tau.$$ This is a contradiction.
\end{proof}

\smallskip

We need a further technical lemma which  gives us almost a complete description of the conditions in a finite iteration that are  sent to $H$.
\begin{lemma}\label{lem;conditions-in-k-inv-H}
Let $\iota\colon V\rightarrow N$ be a finite iteration factoring $j\restriction V$ via $k$. % $k \circ \iota = j \restriction V$, $r$ be the condition associated to $\iota$ by the discussion above and $\Gamma = \Gamma_\iota$. 
For every $q \in k^{-1}(H)$, there is $p \in G$ and $q'\leq q$ such that  $q' \leq^*_{\Gamma_\iota} 
\iota(p) \wedge r_\iota$. 

In particular, for every $\alpha \notin \Gamma_{\iota}$, $(\iota(p) \wedge r_\iota) \restriction \alpha \Vdash \iota(p)(\alpha) \leq q(\alpha)$.

%for all $\alpha$, $k(q')(\alpha)^{H \restriction \alpha} \leq q(\alpha)^{H \restriction \alpha}$.   
%\ale{Note for myself: Fix this. If $q\in k^{-1}(H)$ then $k(q')\leq q$ does not make sense.}
\end{lemma}%\ale{I do not quite follow this proof. I think our $q'$ should be defined in the following way. First find something $\leq^*_{\Gamma_\iota}$ in $D$. Next, modify the resulting condition just in $k(\Gamma_\iota)$ by taking the stem of $r_\iota(\alpha)$ and pluging the least tree below the tree of $\bar{q}(\alpha)$. I think this should be the $q'$. However, why for $\alpha\in\dom(\iota(p))$, $q'(\alpha)=\iota(p)(\alpha)\leq^*\bar{q}(\alpha)$??}
\begin{proof}
Since $\supp q$ is a nowhere stationary set, there is $p\in G$ such that $\iota(\supp p)$ covers $\supp  q \setminus \Gamma_{\iota}$ and we may assume that $p \leq q \restriction \kappa$ and that $p\forces_{\mathbb{P}}``\forall p'\in \dot{G}\; \text{$(\iota(p')$ is compatible with $r_\iota)$''}$.

We would like to use Lemma \ref{lem;prikry property in iterated ultrapower}: First, the set $D\s \iota(\mathbb{P}_\kappa)$
of all conditions that decide the statement $q \in \iota(\dot{G})$ is a $\leq^*$-dense open set, by the Prikry Property. Clearly, $D$ is the set of all conditions which are either $\leq$-incompatible with $q$ or $\leq$-stronger than $q$. 

Let $q' \leq^*_{\Gamma_\iota} \iota(p) \wedge r_\iota$ be a condition in $D$. So, either $q'$ is incompatible with $q$ or stronger than $q$. Let us show the first case cannot occur. 
Indeed, $k(q), j(p)$ and $k(r_{\iota})$ are all in the generic filter $H$ and, in particular,  compatible. Since $j(\supp p) \cup k(\Gamma_{\iota}) \supseteq k(\supp q)$, there are only finitely many coordinates $\alpha$ in which the stem of $k(q(\alpha))_{H\restriction k(\alpha)}$ is longer than the stem of $j(p)(k(\alpha))$, and they belong to $\Gamma_\iota$. So, in those coordinates, the stem of $k(q(\alpha))_{H\restriction k(\alpha)}$ is an initial segment of the stem of $k(r_\iota(\alpha))_{H \restriction k(\alpha)}$. We conclude that the strongest lower bound of $k(q), j(p)$ and $k(r_\iota)$ is a direct extension of $k(\iota(p) \wedge r_{\iota})$. In particular, it is compatible with $k(q')$, which is also a direct extension of the above condition with the same support.  

So, we conclude that $q' \leq q$. Since ${q'}(\alpha) = \iota(p)(\alpha)$ for all $\alpha \notin \Gamma_\iota$, we conclude that $\iota(p)(\alpha) \leq q(\alpha)$ can fail only for $\alpha \in \Gamma_\iota$. \end{proof}

The Lemma \ref{lem;conditions-in-k-inv-H} shows that the set $k^{-1}(H)$ is contained in the filter of all conditions weaker than an $\leq^*_{\Gamma_\iota}$ extension of $\iota(p) \wedge r_\iota$ for $p\in G$. We do not know whether it is the same filter (this is essentially Conjecture \ref{conj:k-inv-H-is-generic}), and we even do not know whether $k^{-1}(H)$ is a filter at all. The problem is that given two conditions $p_0, p_1 \in k^{-1}(H)$, it is unclear whether there is a third condition $q \leq p_0, p_1$ in $k^{-1}(H)$. So, even though for every coordinate we have meets, and we may assume that $p_0, p_1$ differ only at $\Gamma$, we still cannot construct a condition in $k^{-1}(H)$ stronger than both of them. 

The next lemma culminates the discussion of this section. Indeed, from it we  will be able to infer our key coding lemma  (Lemma~\ref{lemma:representing-ultrafilters}).

\begin{lemma}\label{MainLemma2}
%Suppose $V = \mathcal{K}[S]$. 
%Let $j \colon V[G] \to M[H]$ be an ultrapower elementary embedding with  $\crit(j)=\kappa$ and
For each $\epsilon < j(\kappa)$ %\footnote{Recall that $j \restriction V$ is obtained from a normal iteration using normal measures.} 
 there is a finite iteration $\iota \colon V \to N$ factoring $j\restriction V$ such that  $\epsilon\in \range(k)$ %\footnote{Here $k$ stands for the map witnessing $k \circ \iota = j\restriction K$.}
and for which the following holds: 

For every $\mathbb{P}_\kappa$-nice name $\dot{X}$ for a subset of $\kappa$,$$\text{$M[H] \models \epsilon \in j(\dot{X})_H\; \Longleftrightarrow\; \exists q \in \iota(\mathbb{P}_\kappa)\, (k(q) \in H\, \wedge\, q \Vdash_{\iota(\mathbb{P_\kappa})}  k^{-1}(\epsilon) \in \iota(\dot{X}))$}.$$

Moreover, letting $r_\iota$ be the condition of Definition~\ref{definitionriota}, we may take $q$ so that, for some $p\in G$, $q\leq^*_{\Gamma_\iota}\iota(p) \wedge r_\iota$ and $\supp q= \supp (\iota(p) \wedge r_\iota)$.% $q\leq^*_{\Gamma_\iota}\iota(p)$.%that differs from it only at coordinates $\Gamma_\iota$.
\end{lemma} 

%\smallskip

%Let us move on and prove Lemma~\ref{MainLemma2}:%This is exactly what Lemma~\ref{lemma:representing-ultrafilters} claims.

%Given the Lemma, the proof of Lemma \ref{lemma:representing-ultrafilters} is complete. Indeed, for every two choices of $q$ as in the "Moreover" part (not necessarily ones that are sent to $H$ by $k$) are compatible. Thus, one can define a $\kappa$-complete ultrafilter on $\kappa$ by the definition of Lemma \ref{lemma:representing-ultrafilters}. While the definition given in the statement of Lemma \ref{lemma:representing-ultrafilters} allows us to use more conditions, they have to be pairwise compatible. This ultrafilter has to agree with $U$, as for each $\dot{X}$ there is $q$ such that $q$ decides whether $k^{-1}(\epsilon) \in \iota(\dot{X})$ and $k(q) \in H$. \ale{What is exactly the filter defined out of 6.17? Is $\{X\in \mathcal{P}(\kappa)_{V[G]}\mid \exists q\in k^{-1}(H)\exists p\in G\, (q\leq^*_{\Gamma_\iota} \iota(p)\wedge r\; q\forces k^{-1}(\epsilon)\in \iota(X)\}$? If so, how do you amalgamate $q$ and $q'$ to some $q''$ that satisfy the same requirements?}
%\ale{What is the connection between this and lemma 6.7? Is this the culmination of the proof of 6.7?}
\begin{proof} %\ale{I need to understand the previous lemmas to understand this one.}
Suppose that the lemma is false. Fix a finite iteration $\iota\colon V\rightarrow N$ %with critical points $\vec\mu$ and 
with $\epsilon\in \rng(k)$, and set $\bar{\epsilon}=k^{-1}(\epsilon)$. For each $\mathbb{P}_\kappa$-name $\dot{X}$ for a subset of $\kappa$, 
\[\begin{matrix}
    
D^+_{\dot{X},\iota} & := & \{s\in \iota(\mathbb{P}_\kappa) \mid s \Vdash_{\iota(\mathbb{P}_\kappa)} \bar{\epsilon} \in \iota(\dot{X}) \} \\ 
D^-_{\dot{X},\iota} & := & \{s\in \iota(\mathbb{P}_\kappa) \mid s \Vdash_{\iota(\mathbb{P}_\kappa)} \bar{\epsilon} \notin \iota(\dot{X}) \} \\ 
D_{\dot{X},\iota} & := & D^+_{\dot{X},\iota} \cup D^-_{\dot{X},\iota}\end{matrix}\] 
By the Prikry lemma this is a $\leq^*$-dense open set and $D_{\dot{X},\iota}\in N$. 

\smallskip

%\textcolor{blue}{Let $r$ is a condition in $N$ which is sent to $H$. We may take it to be maximal in the sense of the remark after Lemma \ref{lem;maximal-stem}.}\ale{This depends on something which is unclear to me.} 

Denote by $\Gamma_\iota$ the set from Lemma~\ref{lem;prikry property in iterated ultrapower}. Likewise,  $r_\iota$ denotes the condition arising from Definition \ref{definitionriota}; namely, $r_\iota$ is the weakest condition with support $A_\iota$ and stem $\dot\eta_\alpha$ at each  $\alpha \in A_\iota$.  By Lemma~\ref{AsubsetGamma}, $A_\iota \subseteq \Gamma_\iota$, hence  $\supp r_\iota\s  \Gamma_\iota$. 

Let $p_0$ be any condition in $G$ forcing that $\iota(q)$ is compatible with $r_\iota$, for every $q \in G$.
Invoking Lemma \ref{lem;prikry property in iterated ultrapower} for the direct extension order $\leq^*$ of $\mathbb{P}_\kappa$, the $\leq^*$-dense open set $D_{\dot{X},\iota}$ and the conditions $r_\iota$ and $p_0$ we obtain { a condition $p\in G$, $p \leq p_0$ such that $\iota(p) \wedge r_\iota$ admits a $\leq^*_{\Gamma_\iota}$-extension $s\in D_{\dot{X},\iota}$ with $\supp s=\supp (\iota(p)\wedge r_\iota)$. %of $\iota(p) \cup r \restriction \Gamma_\iota$ 
}%\ale{Look at Lemma 6.5. What's the $p$ we need to choose to apply the lemma?}\yair{This is $p_0$.}
%Fix such a direct extension $s$ of $\iota(p)\wedge r_\iota$ such that $s\in D_{\dot{X},\iota}$. 
Note that any two conditions $s$'s as above are compatible, hence the decision on $``\bar\epsilon \in \iota(\dot{X})$'' is independent of their choice.% of such $s$. 

If for every $\dot{X}_G\in\mathcal{U}$ there was a condition $s\in D_{\dot{X},\iota}$ with $s\leq^*_{\Gamma_\iota}\iota(p)\wedge r_\iota$ (for some $p\in G$) and  $k(s) \in H$ then the conclusion of Lemma~\ref{MainLemma2} would hold, thus contradicting our departing assumption. Hence, there must be some $\dot{X}_G\in\mathcal{U}$ such that if $s\in D_{\dot{X},\iota}$ and  $s\leq^*_{\Gamma_\iota}\iota(p)\wedge r_\iota$ then $k(s)\notin H.$

%decides correctly the truth value of $\bar\epsilon \in \iota(\dot{X})$ then the conclusion of Lemma~\ref{MainLemma2} holds. In particular, if $k(s) \in H$ then we are done.

%Consequently our initial assumption implies that this is not the case and so for every normal finite iteration $\iota$ there is a $\mathbb{P}_\kappa$-name $\dot{X}$ for a subset of $\kappa$ for which no choice of $s$ is mapped by $k$ inside $H$.

Fix one such condition $s$. Since $k(s) \notin H$, and 
$$\{\bar{s}\in j(\mathbb{P}_\kappa)\mid \bar{s}\leq k(s)\,\vee\, \bar{s}\perp k(s)\}$$
is dense, we can find $\bar{s}\in H$ such that $\bar{s}\perp k(s)$. % and $\bar{s}\forces_{j(\mathbb{P}_\kappa)}\epsilon\in j(\dot{X}).$
%there is a condition in $H$  incompatible with it. Let $u$ be such a condition. 

Since $M$ is a direct limit of a normal iteration, applying Lemma~\ref{lem;representing-elements-in-direct-limit}, we have that $\bar{s}$ originates from another finite iteration $\iota' \colon V \to N'$; to wit, if $k'\colon N'\rightarrow M$ is the factor embedding between $j$ and $\iota'$ then $\bar{s}\in \range(k').$ Finally, by Lemma~\ref{lem;conditions-in-k-inv-H}, there is  $s' \leq^*_{\Gamma_{\iota'}} \iota'(p') \wedge r_{\iota'}$ for some $p'\in G$, %where $r'$ is a condition of Lemma \ref{lem;maximal-stem}, 
such that $k'(s') \leq \bar{s}$. Clearly, $k'(s')$ is also incompatible with $k(s)$. 
\begin{claim}
The following are true:
\begin{enumerate}
    \item $\{\beta\mid k'(s')(\beta)^{H \restriction \beta} \perp k(s)(\beta)^{H \restriction \beta}\}\s k(\Gamma_\iota) \cap k'(\Gamma_{\iota'}).$
    \item For each $\beta$ in the above set, $$\stem (k'(r_{\iota'})(\beta)^{H\restriction\beta})\neq \stem (k(r_\iota)(\beta)^{H\restriction\beta}).$$
\end{enumerate}
\end{claim}
\begin{proof}
 Let us stipulate that %$q' := k'(s')(\beta)^{H \restriction \beta}$, $q := k(s)(\beta)^{H \restriction \beta}$, 
$r := r_\iota, r' = r_{\iota'}$, $\Gamma := \Gamma_\iota$ and  $\Gamma' = \Gamma_{\iota'}$.

\smallskip

(1) Let $\beta$ be such that $k'(s')(\beta)^{H \restriction \beta} \perp k(s)(\beta)^{H \restriction \beta}$.  Clearly, $\beta \geq \kappa$.  

Since any two conditions in  $j(\mathbb{Q})^{H\restriction\beta}_\beta$ with the same stem are compatible it must be the case that $k(s)(\beta)^{H\restriction\beta}$ and $k(s')(\beta)^{H\restriction\beta}$ have different stems. To simplify notations, let us put
$q := k(s)(\beta)^{H \restriction \beta}$ and $q' := k'(s')(\beta)^{H \restriction \beta}$.

\smallskip

\underline{\textbf{First,  $\beta \in k(\Gamma)$:}} Otherwise, $\beta\notin k(\Gamma)$ and so $q = j(p)(\beta)$, which belongs to the generic filter $H_\beta \subseteq j(\mathbb{Q})_\beta$. Since $q'\in H_\beta$ (because $k'(s')\in H$) we infer that $q'$ and $q$ are compatible, hence obtaining a contradiction.

\smallskip

\underline{\textbf{Second, $\beta \in k'(\Gamma')$:}} Otherwise,  $q' = j(p')(\beta)^{H\restriction\beta}$ and thanks to the above discussion $\beta\in k(\Gamma).$ Since the condition $\iota(p')$ is compatible with $r$, $$\text{$(j(p') \wedge k(r))(\beta)^{H \restriction \beta}$ is well-defined and extends $j(p')(\beta)^{H \restriction \beta}$.}$$
In particular, the remainder between the stems
$$\stem((j(p') \wedge k(r))(\beta)^{H \restriction \beta})\setminus \stem(j(p')(\beta)^{H \restriction \beta})$$
is included in the measure one set carried by $j(p')(\beta)^{H\restriction\beta}$. 

Since $\beta\in k(\Gamma)$ maximality of (the stem of) $r$ guarantees that
$$\stem((j(p') \wedge k(r))(\beta)^{H \restriction \beta})=\stem(k(r)(\beta)^{H\restriction\beta})$$
and thus 
$$\stem(k(r)(\beta)^{H\restriction\beta})\setminus \stem(j(p')(\beta)^{H \restriction \beta})$$
is included in the measure one set of $j(p')(\beta)^{H\restriction\beta}$.

\smallskip

On a different front, $q$ is a  $\leq^*$-extension of $(j(p) \wedge k(r))(\beta)^{H\restriction\beta}$ so 
\[M[H \restriction \beta]\models \stem q= \stem (j(p) \wedge k(r))(\beta) = \stem (k(r)(\beta)).\]
This implies that both $q$ and $j(p')(\beta)^{M[H\restriction\beta]}$ are compatible. A contradiction.

%where the last equality follows from the maximality of the stem of $r$, together with the hypothesis that $\beta\notin k'(\Gamma')$. Thus, $q$ extends $(j(p') \wedge k(r))(\beta)$, which extends $q'$.  Thereby, $q$ and $q'$ are compatible, hence yielding a contradiction. 

\smallskip

(2) By the previous arguments, the stems of $k(s)(\beta)^{H\restriction\beta}$ and $k'(s')(\beta)^{H\restriction\beta}$ must be different. Also, since $\beta\in k(\Gamma)\cap k'(\Gamma')$,  these  stems respectively coincide with those of $k(r)(\beta)^{H\restriction\beta}$ and  $k'(r')(\beta)^{H\restriction\beta}.$ 

The meet of  $q', j(p)(\beta), j(p')(\beta)$ and $k(r)(\beta)$ (call it $u$) belongs to $H_\beta$. If the stem of $q'$ was strictly shorter than the stem of $k(r)(\beta)$ then the stem of $u$ would be exactly the stem of $k(r)(\beta)$ and thus $u$ and $q$ would have the same stem;  as a result, they would be compatible. So, $\stem q$ is strictly longer than the stem of $k(r)(\beta)$. As the stem of $k'(r')(\beta)$ is an extension of the stem of $q$, we obtain the desired inequality.
%$(j(p)\wedge k(r))(\beta)^{H\restriction\beta}$ and $(j(p')\wedge k(r'))(\beta)$ (and indeed, taking $p,p'$ to be stronger will not change this stem), the claim follows.
\end{proof}

\smallskip

For $X\in\mathcal{U}$ and $\iota\colon V\rightarrow N$ a finite iteration with $\epsilon\in \range(k)$ we  say that \emph{$X$ is good for $\iota$} if there is $s\in D_{\dot{X},\iota}$ with $s\leq^*_{\Gamma_\iota}\iota(p)\wedge r_\iota$ (some $p\in G$) such that $k(s)\in H$.\footnote{Note that if $s\in D_{\dot{X},\iota}$ and $k(s)\in H$ then $k(s)\forces_{j(\mathbb{P})}\epsilon\in j(\dot{X}).$} Alternatively, \emph{$X$ is bad} for $\iota$ if it is not good. \footnote{Formally, we should define the notions of good for an iteration $\iota$ using a \emph{name} $\dot{X}$ such that $X = \dot{X}^G$, but it is easy to verify that if $\dot{X}$ is good for 
 $\iota$ and $\dot{X}'$ is forced to be equal to $\dot{X}$ then $\dot{X}'$ is good for $\iota$ as well.}

The above argument shows that if $X$ is bad for $\iota$ then there is another finite iteration $\iota'$ for which $X$ is good and this discrepancy entails $k(\Gamma_\iota)\cap k'(\Gamma_{\iota'})\neq \emptyset$. Moreover, at least one of the  $\beta$'s in this intersection witnesses a disagreement between the length of the stems of $k(r_{\iota})$ and $k'(r_{\iota'})$.

\smallskip

Working in $V[G]$ let us define  $\langle (\iota_\alpha, X_\alpha)\mid \alpha<(2^{\aleph_0})^+\rangle$ as follows:
\begin{enumerate}
    \item $\iota_\alpha\colon V\rightarrow N_\alpha$ is a finite iteration %, $k_\alpha\colon N_\alpha\rightarrow M$ is the factor map between $j$ and $\iota_\alpha$ and  
    with $\epsilon\in \range(k_\alpha)$;
    \item $\langle X_\alpha\mid \alpha<(2^{\aleph_0})^+\rangle\s \mathcal{U}$ is a $\s$-decreasing sequence;
    \item for all $\alpha<(2^{\aleph_0})^+$, 
    \begin{enumerate}
        \item 
    $\bigcap_{\beta<\alpha} X_\beta$ is good for $\iota_\alpha$;
    \item $X_\alpha$ is bad for $\iota_\beta$, for all $\beta\leq \alpha$.
    \end{enumerate}
\end{enumerate}
 In particular, by (3)(a), for each $\beta<\alpha<(2^{\aleph_0})^+$, $k_\alpha(\Gamma_{\iota_\alpha})\cap k_\beta(\Gamma_{\iota_\beta})\neq \emptyset.$
 
 \smallskip

Let $\iota_0$ be any finite iteration witnessing (1). Since we are assuming that the lemma is false there is  $X_0\in\mathcal{U}$ which is bad for $\iota_0$. So, suppose that $\langle (\iota_\beta, X_\beta)\mid \beta<\alpha\rangle$ was defined. By $\kappa$-completeness, $X^0_\alpha:=\bigcap_{\beta<\alpha}X_\beta\in \mathcal{U}$. Arguing as before, there is a finite iteration $\iota_\alpha$ for which $X^0_\alpha$ is good. Also, since we are assuming that the conclusion of the lemma is false, there is $X_\alpha\s X^0_\alpha$ that is bad for $\iota_\alpha.$ Without loss of generality, $X_\alpha \in \mathcal U$. Finally,  $X_\alpha$ is bad for all $\iota_\beta$, $\beta\leq \alpha$: Otherwise, there is $s\in D_{\dot{X}_\alpha,\iota_\beta}$ such that $k_\beta(s)\in H$. Note that $s$ must be in $D^+_{\dot{X}_\alpha,\iota_\beta}$, as $\dot{X}_\alpha$ is in $\mathcal U$. However, $D^+_{\dot{X}_\alpha,\iota_\beta}\s D^+_{\dot{X}_\beta,\iota_\beta}$, which would  contradict that $X_\beta$ is bad for $\iota_\beta.$

\smallskip

%\medskip

%a sequence of $\mathbb{P}_\kappa$-names $\dot{X}_\alpha $ for subsets of $\kappa$ such that, for each $\alpha\leq \beta$, $\one\Vdash_{\mathbb{P}_\kappa} \dot{X}_\alpha \supseteq \dot{X}_\beta$, there is a finite iteration $\iota_\alpha \colon V \to N_\alpha$,  a condition $s'_\alpha$ in  $k_\alpha^{-1}(H)$ such that $s'_\alpha \Vdash \bar\epsilon \in \iota_\alpha(\dot{X}_\alpha)$, but also a condition $s_\alpha \leq^* \iota_\alpha(p) \wedge r_\alpha$, \textcolor{blue}{as in the lemma, such that $s_\alpha \Vdash \epsilon \notin \dot{X}_{\alpha + 1}$.}\ale{As in Lemma 6.7?}
%At limit $\alpha$ we define $\dot{X}_\alpha$ to be the intersection of all the $\dot{X}_\beta$ for $\beta < \alpha$. Using the $\sigma$-completeness of the measure $\mathcal{U}$ the condition $s'_\alpha$ exists.

%Let us iterate this process up to $(2^{\aleph_0})^+$.

Let us now look at the sets $\Gamma_{\alpha} := \Gamma_{\iota_{\alpha}}$. By the argument above, for every $\alpha < \beta$, every disagreement point between (the evaluations of) $k_\alpha(r_\alpha)$ and $k_\beta(r_\beta)$ is a member of $k_\alpha(\Gamma_\alpha) \cap k_{\beta}(\Gamma_\beta)$. Let us apply the $\Delta$-system lemma to $\{k_\alpha(\Gamma_\alpha)\mid \alpha<(2^{\aleph_0})^+\}$. This way we obtain a set $I$ of cardinality $(2^{\aleph_0})^+$ and a root $\Delta$ such that for all $\alpha,\beta\in I$, $k_\alpha(\Gamma_\alpha) \cap k_\beta(\Gamma_\beta)=\Delta.$ %\ale{I do not understand this part of the argument.}
Fix a coloring $c\colon [(2^{\aleph_0})^+]^2\rightarrow \Delta$ sending $\langle \alpha, \beta\rangle$ to the least coordinate $\zeta$ in the root $\Delta$ which exhibits a difference between the stems of $k_\alpha(r_\alpha)^{H \restriction \zeta}$ and $k_\beta(r_\beta)^{H \restriction \zeta}$. By the Erd\H{o}s-Rado theorem this coloring admits a homogeneous set $J\subseteq I$  of order type $\omega_1+1$. So, let $\gamma\in \Delta$ be such that $c``J=\{\gamma\}$; namely, $\gamma$ is the common least coordinate witnessing a disagreement between the stems of $k_\alpha(r_\alpha)(\gamma)^{H \restriction \gamma}$ and $k_\beta(r_\beta)(\gamma)^{H \restriction \gamma}$, for all $\langle\alpha,\beta\rangle\in [J]^2$. Since for each such $\alpha, \beta$ the lengths of $\stem (k_\alpha(r_\alpha)(\gamma)^{H\restriction\gamma})$ and $\stem (k_\beta(r_\beta)(\gamma)^{H\restriction\gamma})$ are finite and different we derive a contradiction with our departing assumption. %\yair{I think that there is a way to get that without the Erdos-Rado. Start with $\aleph_1$ many indices. Use the Delta-system to get a root and also to stabilize the length of the stems on the root. Then, we can't have disagreement.}
%\ale{Whatever you wish. I think your argument with the Erdos Rado theorem is quite elegant.}
\end{proof}

Given Lemma~\ref{MainLemma2} the proof of our coding lemma is  easy: 
\begin{proof}[Proof of Lemma~\ref{lemma:representing-ultrafilters}]
Let $\mathcal{U}\in V[G]$ be a $\kappa$-complete ultrafilter on $\kappa$ and let $j\colon V[G]\rightarrow M[H]$ be the induced ultrapower. %By our anti-large cardinal hypothesis, $j\restriction K$ is an iteration of measures. By Lemma \ref{lem;lifting-S}, this latter embedding lifts to an iteration of normal measures in $V$, in a way that it is determined by $j(S)$. In particular, we can find a finite sub-iteration $\iota\colon V\rightarrow N$ of $j$ using measures from $V$ such that $\epsilon:=[\id]_{\mathcal{U}}$ belongs to $\range(k)$, the range of the factor map between $\iota$ and $j$ (see Lemma~\ref{lem;representing-elements-in-direct-limit}). 
By our setup considerations in page~\pageref{Section;proof of coding} there is a finite iteration $\iota\colon V\rightarrow N$ such that $k\circ\iota=j\restriction V$ and $\epsilon:=[\id]_{\mathcal{U}}$ belongs to $\range(k)$. Invoke Lemma~\ref{MainLemma2} with respect to  these inputs and define $\mathcal{V}$ to be the set of all $X\in \mathcal{P}(\kappa)^{V[G]}$ such that there are $q\in k^{-1}(H)$ and $p\in G$ such that 
 $$q\leq^*_{\Gamma_\iota}(\iota(p)\wedge r_\iota)\;\text{ and }\; q\forces_{\iota(\mathbb{P}_\kappa)}k^{-1}(\epsilon)\in \iota(\dot{X}).$$
Note that any two conditions $q, q'$ witnessing the above admit an explicit $\leq^*_{\Gamma_\iota}$-extension. In particular, $\mathcal{V}$ is a filter. Also, by the moreover part of Lemma~\ref{MainLemma2}, $\mathcal{V}$ contains the $\kappa$-complete ultrafilter $\mathcal{U}$. Thus, $\mathcal{V}=\mathcal{U}.$ \label{proofofcoding}
\end{proof}

\subsection{A conjecture}

Before addressing the consistency of the $\omega$-gluing property (Theorem~\ref{MainTheorem}) let us pose an attractive conjecture, which is a strengthening of what we show in Lemma~~\ref{lem;conditions-in-k-inv-H}:
\begin{conjecture}\label{conj:k-inv-H-is-generic}
Let $j \colon V[G] \to M[H]$ be an ultrapower embedding with critical point $\kappa$ such that $j \restriction V$ is a normal iteration using normal measures. 

Then, there is a finite subiteration iteration $\iota \colon V \to N$ factoring $j\restriction V$, $[\id]\in \range(k)$, with critical points $\vec\mu$, and a condition $r \in \iota(\mathbb{P}_\kappa)$ such that $\supp r = \vec\mu$ and for all $s \leq^* r$ in $N$ with $\supp s=\vec{\mu}$ then $k(s) \in H$.% where $k$ is the factor embedding. %elementary embedding satisfying $k\circ \iota = j$. }
\end{conjecture}
The meaning of this is the following: There is a finite subiteration $\iota$ of $j$ for which there is a $\iota(\mathbb{P}_\kappa)$-generic filter  that can be defined from $\iota``G$ together with finitely-many additional information, and all its conditions are sent to $H$ by $k$. Moreover, this filter is $\leq^*$-generic. The conjecture asserts something similar to  what was obtained in Lemma~\ref{lemma:extending-S} when we force with $\mathbb{S}_\kappa$ over $\mathcal{K}.$

If our conjecture holds then, in particular,  the above-mentioned filter will decide all statements of the form $``k^{-1}([\id]) \in \iota(\dot{X})$'' for a $\mathbb{P}_\kappa$-name $\dot{X}$  of a subset of $\kappa$. Thus, such generic will determine the measure $\mathcal{U}$ defining  $j$. 

Unfortunately we do not know how to prove the conjecture, but we got pretty close to it, by showing that its corollary  holds; namely, there is a finite iteration that determines $\mathcal{U}$ (Lemma~\ref{MainLemma2} above).

\subsection{A model for the \texorpdfstring{$\omega$}{omega}-gluing property from \texorpdfstring{$o(\kappa)=\omega_1$}{o(kappa)=omega1}}
We are finally in conditions to derive the main result of Section~\ref{section; improving}.
\begin{theorem}[$\mathrm{V}=\mathcal{K}$]\label{MainTheorem}
Let $\kappa$ be a measurable cardinal with $o(\kappa)=\omega_1$ and assume that there are no other measurables with Mitchell order $\geq \omega_1$. 

Then, there is a forcing extension in which $\kappa$ has the $\omega$-gluing property.
\end{theorem}
\begin{proof}
Let ${\mathcal{U}}=\langle {U}(\mu,\zeta)\mid \mu\leq \kappa,\,\zeta<o^{{\mathcal{U}}}(\mu)\rangle$ be the coherent sequence of normal measures over $\kappa$ in $\mathcal{K}$. Let $S\s \mathbb{S}$ be $\mathcal{K}$-generic  and let $\ell := \bigcup S$ be the obtained generic fast function. Working in $\mathcal{K}[S]$, let us define a non-stationary supported iteration of Priky-type forcings $\langle \mathbb{P}_\alpha, \dot{\mathbb{Q}}_\beta\mid \beta<\alpha\leq \kappa\rangle$ as follows. Suppose that $\mathbb{P}_\alpha\in \mathcal{K}[S]$ has been defined. If $\alpha$ is not measurable (in $\mathcal{K}$) then we let $\dot{\mathbb{Q}}_\alpha$ a $\mathbb{P}_\alpha$-name for the trivial forcing. So, suppose $\alpha$ is measurable in $\mathcal{K}$. Then, $\mathbb{S}/\{\langle \alpha, \ell(\alpha)\rangle \}$ factors as $\mathbb{S}_\alpha \times \mathbb{S}\setminus (\alpha+1)$ and $\mathbb{S}\setminus (\alpha+1)$ is $\mathbb{S}_\alpha$-forced to be much more distributive than the rank of $\mathbb{P}_\alpha$. In particular, $\mathbb{P}_\alpha\in \mathcal{K}[S\restriction\alpha]$.\footnote{As customary, $S\restriction\alpha$ is the natural projection of $S$ to a $\mathbb{S}\restriction\alpha$-generic.} 
The forthcoming definition will ensure that $\mathbb{Q}_\alpha$ is non-trivial (for most $\alpha$) in the generic extension by $\mathbb{S}_\alpha \ast \mathbb{P}_\alpha$.

Work in $\mathcal{K}[S\restriction\alpha]$. We say that a tuple $\langle \vec \rho, \vec \zeta, f, g\rangle$ \emph{codes a measure} on $\alpha$ if:
\begin{enumerate}
    \item $\vec\rho=\langle \rho_0,\dots, \rho_{n-1}\rangle$  where $\rho_i \colon \alpha^i \to (\alpha+1)$ and $\rho_0(\langle\rangle):=\alpha$,
    \item $\vec\zeta=\langle \zeta_0,\dots, \zeta_{n-1}\rangle$ is a sequence of %length $n$ of 
    countable ordinals,
    \item $f \colon \alpha^n \to  \mathbb{P}_\alpha$, %\ale{In which model do these functions live? In $\mathcal{K}[S\restriction\alpha]$ or in $\mathcal{K}$? I think in $\mathcal{K}[S\restriction\alpha]$. From this depends where the code lives.}
    \item $g \colon \alpha^n \to \alpha$,
    \item For each $i < n$, $\rho_i(\langle \mu_j \mid j < i\rangle) = \mu_i$ and $\zeta_i < o^{N_i}(\mu_i)$, where $$\iota_{i,i+1}\colon N_i \to N_{i+1} \cong \Ult(N_i, \iota_i({\mathcal{U}})(\mu_i, \zeta_i)),$$ being  $\iota_{i + 1} := \iota_{i, i + 1} \circ \iota_i$, $N_0 := V$ and $\iota_0 := \id$, 
\item Let $\iota^*_n$ be denote the lifting\footnote{Such lifting exists by Lemma~\ref{lem;lifting-S} and it is definable in $\mathcal{K}[S\restriction\alpha]$. Indeed, such lifting is determined by $\iota_n$, the generic $S\restriction\alpha$ and some sequence $\langle a_i\mid i<n\rangle \in \prod_{i<n}H(\mu_i^+)$. In fact, any elementary embedding $j\colon \mathcal{K}[S\restriction\alpha]\rightarrow M$ for which $\iota=j\restriction \mathcal{K}$ is a normal iteration is a lifting of $\iota.$} of $\iota_n$ to $\mathcal{K}[S\restriction\alpha]$ and set $$\text{$r = \iota^*_n(f)(\langle \mu_i \mid i < n\rangle),$  $\epsilon = \iota^*_n(g)(\langle \mu_i \mid i < n\rangle)$,}$$ and%\ale{I think taking here the lifting is important, because in the coding lemma the witnessing iteration $\iota$ is from $\mathcal{K}[S\restriction\alpha]$. }
\[\begin{matrix} \dot{U} &= &\{\dot{X}_{S \restriction \alpha \ast \dot{G}} \subseteq \alpha \mid & \exists p \in \dot{G} \; \exists s\in \iota^*_n(\mathbb{P}_\alpha), \\ & & & s \leq^*_{\Gamma_{\iota_n}} \iota^*_n(p) \wedge r, s\in N_n, \\ & & & N_n \models s \Vdash_{\iota^*_n(\mathbb{P}_\alpha)} \epsilon \in \iota^*_n(\dot{X})\}.\end{matrix}\]
Then, it is the case that
$$\one\forces^{\mathcal{K}[S\restriction\alpha]}_{\mathbb{P}_\alpha}\text{$``\dot{U}$ is an $\alpha$-complete ultrafilter on $\alpha$''.}$$
\end{enumerate}

 In the above-described scenario we will say that $\dot{U}$ is \emph{coded} by $\langle \vec\rho, \vec\zeta, f, g\rangle$. Note that a code $\langle \vec \rho, \vec \zeta, f, g\rangle$ for a measure on $\alpha$ belongs to $H(\alpha^+)^\mathcal{K}$. %\ale{In principle codes belong to $H(\alpha^+)^{V[S\restriction\alpha]}$, so why are they elements of $H(\alpha^+)^{V}$? We will need $\alpha^+$-ccness  or some sort of fusion argument.} 
 In particular, $\langle \vec \rho, \vec \zeta, f, g\rangle$ is a legitimate value for $s(\alpha)$ provided $s\in \mathbb{S}$ (see our definition in page~\pageref{fastfunctionforcing}). % - this being true by $\alpha^+$-ccness of $\mathbb{S}\restriction\alpha$ (see \cite[Proposition~8.4]{CumHandBook}).
 %\footnote{Note, however, that being a code is a property in $\mathcal{K}[S\restriction\alpha\ast \mathbb{P}_\alpha]$ not in $\mathcal{K}$ precisely because of requirement }  
 In particular, any code $\langle \vec \rho, \vec \zeta, f, g\rangle$ is a legitimate value for $s(\alpha)$ for any condition $s\in \mathbb{S}$ (see Definition~\ref{fastfunctionforcing}). 

\smallskip

Let us define $\mathbb{Q}_\alpha$: %\footnote{Recall that this definition is in the generic extension by $\mathbb{S}\restriction(\alpha + 1)\ast \mathbb{P}\restriction\alpha$.} 
If $\mathbb{P}_\alpha$ forces (over ${\mathcal{K}[S\restriction\alpha]})$ that $$\text{``$\dot{\ell}(\alpha)$ is a countable sequence of codes for measures $\mathcal{V}=\langle \dot{U}_n\mid n<\omega\rangle$ on $\alpha$''}$$ we let $\dot{\mathbb{Q}}_\alpha$ be a $\mathbb{P}_\alpha$-name for the $\mathcal{V}$-tree Prikry forcing $\dot{\mathbb{T}}(\mathcal{V})$ of Definition~\ref{UtreePrikry}. Otherwise, $\dot{\mathbb{Q}}_\alpha$ is simply  a $\mathbb{P}_\alpha$-name for the trivial poset.% Otherwise, let $\mathbb{Q}_\alpha$ be the tree Prikry forcing of Section \ref{section:consistent-omega-gluing} defined with respect to  the measures coded by $F(\alpha)$.

\smallskip

Let us show that, in  $\mathcal{K}[S\ast G]$, $\kappa$ has the $\omega$-gluing property. Indeed, by the \emph{Coding Lemma} (Lemma \ref{lemma:representing-ultrafilters}) in the generic extension $\mathcal{K}[S\ast G]$ every $\kappa$-complete ultrafilter on $\kappa$ has a code. Fix a countable sequence of measures in the generic extension, $\langle U_n \mid n < \omega\rangle$, and let $\vec{c}=\langle \langle \vec{\rho}_n,\vec{\zeta}_n,f_n,g_n\rangle \mid n < \omega\rangle$ be a sequence of codes for them. 
Let $\zeta_*$ be a countable ordinal larger than all the countable ordinals  mentioned in all  the $\vec{\zeta}_n$'s. This choice is possible for we were assuming that $\kappa$ is the unique measurable with $o(\kappa)\geq \omega_1$.

\smallskip

Let $\iota \colon \mathcal{K} \to \Ult(\mathcal{K}, \mathcal{U}(\kappa,\zeta_*)) =: M$. By Lemma~\ref{LemmaAboutStationary} regarded with respect to $\mathbb{S}$, $\iota$ lifts to ${\iota} \colon \mathcal{K}[S] \to M[\iota(S)]$ in a way that $\iota(\ell)(\kappa) = \vec{c}.$ By Claim~\ref{claim: closed under kappa}, $M[\iota(S)]$ is closed under $\kappa$-sequences in $\mathcal{K}[S].$ Hence, $\iota(\mathbb{P}_\kappa)$ factors as $$\mathbb{P}_\kappa\ast \dot{\mathbb{Q}}_\kappa^{M[\iota(S)]}\ast\mathbb{R}.$$ We would like to show that $\mathbb{P}_\kappa$ forces (over $M[\iota(S)]$) that $\iota(\ell)(\kappa)$ (i.e., $\vec{c}$) is a countable sequence of codes for measures over $\kappa$. This will imply that  $\dot{\mathbb{Q}}_\kappa$ is a name for the corresponding tree Prikry forcing $\dot{\mathbb{T}}(\langle U_n\mid n<\omega\rangle)$, which will lead us to a situation similar to that of the proof of Theorem~\ref{omegagluingnonoptimal}. 

\begin{claim}\label{claim: closed under kappa}%\ale{Please, check it.}
    $M[\iota(S)]$ is closed under $\kappa$-sequences in $\mathcal{K}[S]$.
\end{claim}
\begin{proof}[Proof of claim]
First, it is enough to show that $M[S]$ is closed under $\kappa$-sequences in $\mathcal{K}[S]$.  Since $M[S]$ is a model of the Axiom of Choice we shall simply check that any set of  ordinals (in $\mathcal{K}[S]$) of cardinality $\kappa$ belongs to $M[S]$. Let $X\s \Ord$ be in $\mathcal{K}[S]$ with size $\kappa$. By Claim~\ref{kappabounding} there is $Y\in \mathcal{K}$ such that $X\s Y$ and $|Y|^{\mathcal{K}}\leq \kappa.$ By $\kappa$-closure of $M$ in $\mathcal{K}$ it follows that $Y\in M$. Note that $X$ can be coded via $Y$ and $c:=\{\alpha<\kappa\mid x_\alpha\in Y\}\in \mathcal{K}[S]$, being $\langle x_\alpha\mid \alpha<\kappa\rangle$  an injective enumeration of $X$ (in $\mathcal{K}[S]$). Also note that $c\in M[S]$ as both $\mathcal{K}[S]$ and $M[S]$ have the same subsets of $\kappa$. All in all, $X$ can be decoded inside $M[S]$ and thus it is a member of this model.
 \end{proof}

\begin{claim}
In $M[\iota(S)\ast G]$, $\vec{c}$ is a sequence of codes for measures on $\kappa$.
\end{claim}
\begin{proof}[Proof of claim]
    Fix a code $c_* = \langle \vec{\rho}_n,\vec{\zeta}_n,f_n,g_n\rangle$. By our choice of $\zeta_*$ (i.e., above $\sup_{n<\omega}(\max\vec{\zeta}_n)$), each measure which is mentioned in the process of decoding $c_*$ exists in $M$ and in its iterations. Since $M[\iota(S)]$ is closed under $\kappa$ sequences in $\mathcal{K}[S]$ then the computations of the models $N_i$ and the embeddings $\iota_i$ will be the same both in $\mathcal{K}[S]$ and in $M[\iota(S)]$. Given the embedding and the models, the definition of $U_n$ is clearly absolute between $\mathcal{K}[S\ast G]$  and $M[\iota(S)\ast G]$. Thus, $c_*$  decodes to a measure   in $M[\iota(S)\ast G]$.
\end{proof}

 %Let us verify that, in $M$, each $c_n=\langle \vec{\rho}_n,\vec{\zeta}_n,f_n,g_n\rangle$ is still a code for $U_n$. 
 %Fix a code $c_* = c_n$. By our choice of $\zeta_*$ (i.e., above $\sup_{n<\omega}\vec{\zeta}_n$), each measure which is mentioned in the process of decoding $c_*$ exists in $M$ and in its iterations. Since $M$ is closed under $\kappa$ sequences, the computations of the models $N_i$ and the embeddings $\iota_i$ are the same. Given the embedding and the models, the definition of $U_n$ is clearly absolute.

Finally, we argue similarly to Theorem~\ref{omegagluingnonoptimal}. We opt to provide some details as there are some few differences compared to the argument in \S\ref{section:consistent-omega-gluing}. For instance, one   %\footnote{Note that the GCH holds  in $\mathcal{K}[S]$ by Proposition~\ref{PropertiesofS}. 
%This was needed in the computations appearing in the proof of  Theorem~\ref{omegagluingnonoptimal}.} 
of the advantages of having non-stationary support is that we do not need to bear on GCH anymore. Indeed,  let $\mathcal{W}$ be all sets $X\s \kappa^\omega$ for which there is $p\in G$ and  a $\mathbb{P}_\kappa$-name $\dot{T}$  such that 
$$p\cup\{\langle \varnothing,\dot{T}\rangle\}\cup \iota(p)\setminus (\kappa+1)\forces_{\iota(\mathbb{P}_\kappa)}\dot{b}_\kappa\in \iota(\dot{X}).$$
\begin{claim}
    $\mathcal{W}$ is a $\kappa$-complete measure on $\kappa^\omega$.
\end{claim}
\begin{proof}[Proof of claim]
It is evident that $\mathcal{W}$ is a filter on $\kappa^\omega$. Suppose that $X\s \kappa^\omega$ and let $D:=D_X$ be a $\iota(\mathbb{P}_{\kappa})\setminus(\kappa+1)$-name for the set $$\{q\in \iota(\mathbb{P}_\kappa)/(\kappa+1)\mid q\parallel_{\iota(\mathbb{P}_\kappa)/(\kappa+1)}\dot{b}_\kappa\in \iota(\dot{X})\}.$$
By the Prikry property for non-stationary-supported iterations (see p.\pageref{GitikIterationsccc}) this is a $\leq^*$-dense open set. Let $f\colon \kappa\rightarrow V$ be a function representing $D$; namely, $\iota(f)(\kappa)=D$. Without loss of generality, $f(\alpha)$ is a $\mathbb{P}_{\alpha+1}$-name for a $\leq^*$-dense open subset of $\mathbb{P}_\kappa/(\alpha+1)$.\footnote{Formally speaking the above relation holds for all $\alpha\in C\cap\dom(f)$ for a club $C\s\kappa$. We opted to write ``for all $\alpha\in\dom(f)$'' to avoid overcomplicated notations.} Using a fusion argument similar to that of Lemma~\ref{lem;prikry property in iterated ultrapower} (see also \cite[Lemma~2.3]{BenUng}) one shows that $$\{p\in\mathbb{P}_\kappa\mid \exists E\subseteq \kappa \text{ a club }\forall \alpha\in E\, (p\restriction(\alpha+1)\forces_{\mathbb{P}_{\alpha+1}}p\setminus (\alpha+1)\in f(\alpha))\}$$ is dense in $\mathbb{P}_\kappa$. In particular there is $p\in G$ in that set. By elementarity, $$\iota(p)\restriction\kappa+1\forces_{\iota(\mathbb{P}_\kappa)_{\kappa+1}}\iota(p)\setminus (\kappa+1)\in D.$$
 Notice that $\kappa\notin \supp(\iota(p))$ precisely because this is nowhere stationary. Hence, using the Prikry property in $V[G]$, we can find a tree $T$ such that $\langle \varnothing, T\rangle$ forces either  $``\iota(p)\setminus (\kappa+1)\forces \dot{b}_\kappa\in\iota(\dot{X})$'' or  $``\iota(p)\setminus (\kappa+1)\forces \dot{b}_\kappa\notin\iota(\dot{X})$''. Suppose, for instance, that the former is true. Then, working back in $V$, we let $q\in G$ with $q\leq p$ such that $q$ forces all the above information. Thus, 
 $$q\cup\{\langle \varnothing, \dot{T}\rangle\}\cup \iota(q)\setminus \kappa+1\forces \dot{b}_\kappa\in \iota(\dot{X})$$
 This shows that $X\in \mathcal{W}.$ Similarly, if we would have assumed that the second option holds  then we would have concluded that $\kappa^\omega\setminus X\in\mathcal{W}.$

 The argument for $\kappa$-completeness is similar. Suppose that $\langle X_\alpha\mid \alpha<\delta\rangle$ is a sequence  of sets in $\mathcal{W}$, for some $\delta<\kappa.$ Now argue as we did before but this time with respect to the $\leq^*$-dense open set  \begin{equation*}\{q\in \iota(\mathbb{P}_\kappa)/(\kappa+1)\mid q\forces _{\iota(\mathbb{P}_\kappa)/(\kappa+1)}\dot{b}_\kappa\in \bigcap_{\alpha<\delta}\iota(\dot{X}_\alpha)\}.\footnote{The fact that it is dense comes from the fact that $\langle \iota(\mathbb{P}_\kappa)/(\kappa+1),\leq^*\rangle$ is $\kappa^+$-closed.}\qedhere\end{equation*}
\end{proof}
Arguing as in Claim~\ref{Wglues}  %(and combining it with Lemma~\ref{lemma:omega-gluing-by-a-measure}) 
one show that $\mathcal{W}$ glues $\langle U_n\mid n<\omega\rangle$. From altogether we conclude that $\kappa$ has the $\omega$-gluing property in $V[G]$.
%there is a direct extension of $\iota(p)$ forcing the canonical name of the generic of $\tilde\iota(\mathbb{Q})_\kappa$ to be in $\tilde\iota(\dot{X})$. % \ale{I suggest to be more explicit here. What do you think?}\yair{I think that it is OK as it is here, but feel free to add more details.}
\end{proof}

\section{A lower bound for the \texorpdfstring{$\omega$}{omega}-gluing property} \label{section:lower-bound}
In this last section we provide a lower bound for the consistency strength of the $\omega$-gluing property. Since the proof relies on some key features of the core model $\mathcal{K}$ our readers might want to revisit the material in \S\ref{sectionInnerModel}; particularly, \emph{Mitchell's Core Model Theorem} (Theorem~\ref{CoreModelTheorem}).
\begin{theorem}\label{lowerboundgluing}
Assume that $\kappa$ is a measurable cardinal having the $\omega$-gluing property. If there is no inner model of $``\exists \alpha\,(o(\alpha) = \alpha)$'' then $o^\mathcal{\mathcal{K}}(\kappa) \geq \omega_1$.
\end{theorem} 
\begin{proof}
Assume that there is no inner model of $``\exists \alpha\,(o(\alpha) = \alpha)$''. Through the proof $\langle \mathcal{U}_{\kappa,\xi} \mid \xi < o^\mathcal{\mathcal{K}}(\kappa)\rangle$ will denote the enumeration of the  normal measures in $\mathcal{K}$ indexed by their corresponding Mitchell orders.  We aim to produce --by induction on $\beta < \omega_1$-- a sequence $\langle \mathcal{V}_\beta \mid \beta < \omega_1\rangle$ of $\kappa$-complete $\mathcal{K}$-normal measures on $\kappa$ such that % each $\mathcal{V}_\beta \cap \mathcal{K}$ is a $\mathcal{K}$-normal measure and the sequence 
$\langle \mathcal{V}_\beta \cap \mathcal{K} \mid \beta < \omega_1\rangle$ is Mitchell-order increasing.  By maximality of the core model (i.e. Theorem~\ref{CoreModelTheorem}(1)),  if $\mathcal{V}_\beta$ is $\mathcal{K}$-normal and $\kappa$-complete then $\mathcal{V}_\beta\cap \mathcal{K}\in \mathcal{K}$. In particular, every such $\mathcal{V}_\beta$ will admit an index $\xi_{\beta}<o^\mathcal{K}(\kappa)$ such that $\mathcal{V}_\beta\cap \mathcal{K}=\mathcal{U}_{\kappa,\xi_\beta}$.  The sequence $\langle \mathcal{V}_\beta\cap \mathcal{K}\mid \beta<\omega_1\rangle$ will thus produce an evidence for $o^\mathcal{K}(\kappa)\geq \omega_1$. Roughly speaking, the measures $\mathcal{V}_\beta$ will be  the outcome of gluing (in $V$) the sequence where each of the previous measures are repeated $\omega$-many times.

\smallskip

 %\textcolor{blue}{Thus, for each $\beta$, there is some $\xi_\beta$ such that $\mathcal{V}_\beta \cap \mathcal{K} = U_{\kappa, \xi_\beta}$.} %\ale{But we did not construct yet the measures $\mathcal{V}_\beta$...} 

Let $\mathcal{V}_0$ be a normal $\kappa$-complete measure in $V$. Clearly, $\mathcal{V}_0\cap \mathcal{K}$ is $\kappa$-complete and $\mathcal{K}$-normal so that $\mathcal{V}_0\cap \mathcal{K}\in \mathcal{K}$. Assume that we have built $\langle \mathcal{V}_\beta \mid \beta < \alpha\rangle$. 
 
 If $\alpha$ takes the form $\alpha=\gamma+1$ then we apply the $\omega$-gluing property to the constant sequence $\vec{\mathcal{V}} = \langle \mathcal{V}_\gamma \mid n < \omega\rangle$. Otherwise, we fix $\langle \beta_n \mid n < \omega\rangle$ a cofinal sequence in $\alpha$ and apply the $\omega$-gluing property to  $$\vec{\mathcal{V}} = \langle \mathcal{V}_{\beta_{r(m)}} \mid m < \omega\rangle,$$ where $r \colon \omega \to \omega$ is a function such that $|r^{-1}(\{n\})|=\aleph_0$  for all $n<\omega$. %In both cases we get a $\kappa$-complete measure $\mathcal{V}'_\alpha$.% be the obtained measure, using the $\omega$-gluing property. \ale{I understand that we apply again the $\omega$-gluing property. Is that right?}
 
 \smallskip

 Let $\bar{\mathcal{V}}_\alpha$ be the obtained gluing measure and  $j \colon V \to M$ be its ultrapower map. % using $\mathcal{V}_{\alpha}'$.  
By virtue of our anti-large cardinal hypothesis, %since there is no inner model with $o(\kappa) = \kappa^{++}$, by 
Theorem~\ref{CoreModelTheorem}(2) implies that $j \restriction \mathcal{K} \colon \mathcal{K} \to \mathcal{K}^M$ is a normal iteration of measures in $\mathcal{K}$. 

Denote this iteration  by $\langle \iota_{\gamma,\delta} \mid \gamma \leq \delta < \delta_*\rangle$ and 
$$\text{$\iota_{\gamma,\gamma+1}\colon \mathcal{K}_\gamma\rightarrow \mathcal{K}_{\gamma+1}$ and $\mu_\gamma:=\crit(\iota_{\gamma,\gamma+1})$.}$$
We may assume  that $\mu_{\gamma}\leq i_\gamma(\kappa)$ for $\gamma\leq \delta_*$: Indeed, %\iota_{\gamma, \gamma+1}$ in the iteration is $\l>eq \iota_\gamma(\kappa)$:  
if for some index $\gamma$, $\mu_\gamma>\iota_\gamma(\kappa)$ then the latter must be $j(\kappa)$, hence above $\eta_n$ for all $n<\omega$.  So, if that happens, we truncate the iteration at that stage $\gamma$.
 %Namely, for each $\gamma < \delta_*$, $\iota_{\gamma,\gamma+1}$ is an internal ultrapower using a normal measure in $\mathcal{K}_\gamma$, which is the co-domain of $\iota_\gamma$. Moreover the sequence of critical points $\langle \crit \iota_{\gamma,\gamma+1} \mid \gamma < \delta_*\rangle$ is strictly increasing. 

 \smallskip

Since $\bar{\mathcal{V}}_{\alpha}$ glues $\vec{\mathcal{V}}$ there is, by the discussion right after Definition~\ref{def:gluing-property},  an increasing sequence of ordinals $\langle \eta_n \mid n < \omega\rangle$ such that
\[\vec{\mathcal{V}}(n) = \{X \subseteq \kappa \mid \eta_n \in j(X)\}.\] %\ale{Is this a normal measure in $\mathcal{K}$?}
Let $\eta_\omega := \sup_{n < \omega} \eta_n$ and define\[\mathcal{V}_\alpha := \{X \subseteq \kappa \mid \eta_\omega \in j(X)\}.\] %\ale{Missing $j$?}
By our induction hypothesis $\vec{\mathcal{V}}(n)$ is a normal measure, hence $\vec{\mathcal{V}}(n)\cap \mathcal{K}$  contains the club filter $\mathrm{Club}_\kappa^\mathcal{K}$, and thus $\eta_n\in \bigcap j``\mathrm{Club}_\kappa^\mathcal{K}$ for all $n<\omega$. %Clearly, the same holds true for $\eta_\omega$, for this being the limit of the sequence $\langle \eta_n\mid n<\omega\rangle$.
%\begin{claim}
%For every club $C \in \mathcal{K}$, and for every $n$, $\eta_n \in j(C)$.
%\end{claim}
%\begin{proof}
%By the inductive hypothesis, $\vec{\mathcal{V}}(n) \cap \mathcal{K}$ is normal and therefore contains all clubs.  
%\end{proof}

%\ale{Maybe $f\colon \kappa^m \rightarrow \kappa$? This is what is used below.}

\smallskip

For $m<\omega$ and a function $f \colon \kappa^m \to \kappa$ in $\mathcal{K}$, the set of its closure points $\{\alpha<\kappa\mid  f``\alpha^m\s \alpha\}$ belongs to $\mathrm{Club}^\mathcal{K}_\kappa$. In particular,  all the $\eta_n$'s and $\eta_\omega$ are closure points of $j(f)$, for every such function $f\in \mathcal{K}$. 

\begin{claim}\label{gammacriticalpoints}
For $n < \omega$, there is $\gamma_n < \delta_*$ with $\iota_{\gamma_n}(\kappa) = \crit(\iota_{\gamma_n, \gamma_{n+1}}) = \eta_n$.
\end{claim}
\begin{proof}
{Let $\gamma_n:=\gamma < \delta_*$ be the maximal ordinal such that  $\mu_{\bar{\gamma}}< \eta_n$ for $\bar{\gamma} < \gamma$.} %\ale{Why this ordinal exists? Is it because otherwise for some $\gamma<\delta_*$, $\iota_{\gamma}(\kappa)=j(\kappa)$?} 

Recall that every member of $\mathcal{K}_\gamma$ is represented as  $$\iota_{\gamma}(f)(\mu_{\gamma_0}, \dots, \mu_{\gamma_{m-1}}),$$ where  $m<\omega$, $f \colon \kappa^m \to \mathcal{K}$ is a function in $\mathcal{K}$ and $\gamma_0,\dots,\gamma_{m-1}<\gamma$.\footnote{This is a standard fact about iterated ultrapowers. For details see \cite[\S19]{Kan} and, particularly, \cite[Lemma~19.6]{Kan}.} % $\alpha_0,\dots, \alpha_{m-1}$ are critical points of steps in the iteration before step $\gamma$.  
In particular, any ordinal ${<}\iota_{\gamma}(\kappa)$ is represented by a function $f\colon \kappa^m\rightarrow \kappa$ in $\mathcal{K}$. %as above such that $\mathrm{range}\ f \subseteq \kappa$.

\smallskip 

Let us show that $\eta_n=\iota_\gamma(\kappa)=\mu_{\gamma}$: For each $\alpha<\iota_\gamma(\kappa)$ there is a function $f\colon \kappa^m\rightarrow \kappa$ and $\gamma_0,\dots, \gamma_{m-1}<\gamma$ with  $\alpha=\iota_\gamma(f)(\mu_{\gamma_0},\dots, \mu_{\gamma_{m-1}})$. Clearly, $$\iota_\gamma(f)(\mu_{\gamma_0},\dots, \mu_{\gamma_{m-1}})\leq j(f)(\mu_{\gamma_0},\dots, \mu_{\gamma_{m-1}}).$$ Also,  $\eta_n$ is a closure point of $j(f)$ and  $\mu_{\gamma_{m-1}}<\eta_n$, so that
$$\alpha\leq j(f)(\mu_{\gamma_0},\dots, \mu_{\gamma_{m-1}})<\eta_n.$$ All in all, $\iota_\gamma(\kappa)\leq \eta_n$. 

By maximality of $\gamma$,  $\eta_n\leq \mu_{\gamma}$ %In effect,  $\mu_{\gamma+1}\leq \iota_{\gamma+1}(\kappa)=\iota_{\gamma}(\kappa)$, as $\iota_\gamma(\kappa)$ is inaccessible in $\mathcal{K}_\gamma$.  Thus, $\mu_{\gamma+1}\leq \eta_n$. Note that this inequality cannot be strict by maximality of $\gamma$, so that $\mu_{\gamma+1}=\eta_n$. 
and from altogether,  $\iota_{\gamma_n}(\kappa)=\mu_{\gamma_n}=\eta_n.$ %we conclude
%\begin{equation*}
%	\iota_{\gamma_n}(\kappa)=\mu_{\gamma_n}=\eta_n.\qedhere
%\end{equation*} 
%$\eta_n<\iota_\gamma(\kappa)$ then $\eta_n<\iota$
%Since $\iota_{\gamma}(f)(\alpha_0, \dots, \alpha_{m-1}) \leq j(f)(\alpha_0,\dots, \alpha_{m-1}) < \eta_n$, as $\alpha_k < \eta_n$ for all $k$, we conclude that $\iota_\gamma(\kappa) \leq \eta_n$. 
%This inequality cannot be strict, as it would entail that the next critical point must be below $\eta_n$. Thus, $\iota_\gamma(\kappa) = \eta_n$. Similarly, the critical point of the next step in the iteration cannot be strictly below $\eta_n$ (as it would contradict maximality) and nor strictly above (as $j(\kappa)$  would be $\eta_n$), so the conclusion of the lemma follows.
\end{proof}

Let $\gamma_{\omega} := \sup_{n < \omega} \gamma_n$.
\begin{claim}\label{ClaimOnetaOmega}
$\eta_\omega = 
\iota_{\gamma_\omega}(\kappa) =
\crit (\iota_{\gamma_\omega, \gamma_\omega + 1})$. 
\end{claim}
\begin{proof}
On one hand, since $\iota_{\gamma_n}(\kappa) = \eta_n$ and this latter is the critical point of  $\iota_{\gamma_n,\gamma_{n+1}}$,  $\iota_{\gamma_\omega}(\kappa) = \eta_\omega$. On the other hand, $\gamma_\omega<\delta^*$ and $\crit(\iota_{\gamma_{\omega},\gamma_{\omega+1}})=\eta_\omega$. In effect, the latter assertion will follow from normality of the iteration once we show that $\gamma_\omega<\delta^*$. Towards a contradiction, suppose that $\delta_*=\gamma_\omega$. Then, $j(\kappa)=\iota_{\gamma_\omega}(\kappa)=\eta_\omega$. However,  $M$ is closed under $\omega$-sequences and so $\mathrm{cof}^V(j(\kappa))>\omega$. Thus, in particular, $j(\kappa)$ cannot be $\eta_\omega$.
 %by the normality of the iteration,  $\crit \iota_{\gamma_\omega, \gamma_\omega + 1} \geq \eta_\omega$. If the inequality was strict (or $\delta_* = \gamma_\omega$), then $j(\kappa) = \eta_\omega$, which is impossible since $M$ is closed under $\omega$-sequences, and $\cf \eta_\omega = \omega$. 
\end{proof}
Let us analyse the measures from the iteration of $\mathcal{K}$ which are used during the steps $\gamma_n$ and $\gamma_\omega$. By our initial hypothesis ``There is no inner model for $\exists\alpha\,o(\alpha)=\alpha$'', $o^{\mathcal{K}}(\kappa) < \kappa$ and, in particular, ${<}(\kappa^{+})^\mathcal{K}$. Therefore, the normal measures on $\kappa$ in $\mathcal{K}$ are discrete; namely, there is a sequence of pairwise disjoint sets $\langle A_\xi \mid \xi < o^\mathcal{K}(\kappa)\rangle$ %\ale{Why? Doen't it follows just from $\kappa$-completeness and because the measures are distinct?}\yair{This is essentially the argument, but note that you're using here the assumption that the measures can be enumerated in a sequence of length $\leq \kappa$.} 
such that \[A_\xi \in U_{\kappa,\xi}\setminus \bigcup_{\zeta \neq \xi} U_{\kappa, \zeta}.\] 

Let us look at the measure which is applied at step $\gamma_n$ of the iteration. As  $\crit (\iota_{\gamma_n,\gamma_{n+1}}) = \iota_{\gamma_n}(\kappa)$ this is one of the measures on $\iota_{\gamma_n}(\kappa)$ lying in $\mathcal{K}_{\gamma_n}$. Recall that $\langle U_{\kappa,\xi} \mid \xi < o^\mathcal{K}(\kappa)\rangle$ denotes the sequence of normal measures on $\kappa$ in $\mathcal{K}$ and that they are indexed according to their Mitchell order. 

Since $\crit (\iota_\gamma) =\kappa > o^\mathcal{K}(\kappa)$ we conclude that 
\begin{equation*}\label{orderislow}
   \tag{$\diamondsuit$} \iota_{\gamma}(\langle U_{\kappa,\xi} \mid \xi < o^\mathcal{K}(\kappa)\rangle)=
\langle \iota_{\gamma}(U_{\kappa,\xi}) \mid \xi < o^\mathcal{K}(\kappa)\rangle
\end{equation*}
are the only normal measures on $\iota_\gamma(\kappa)$ lying in $\mathcal{K}_\gamma$. 
In particular, the measure that is iterated at stage $\gamma_n$ 
must have the form $\iota_{\gamma_n}(U_{\kappa, \zeta_n}) $ for $\zeta_n < o^\mathcal{K}(\kappa)$. 

\begin{claim}
	$\mathcal{V}_\alpha$ is $\mathcal{K}$-normal. 
\end{claim}
\begin{proof}
	Let $f\colon \kappa\rightarrow \kappa$ be a function in $\mathcal{K}$ with $\{\alpha<\kappa\mid f(\alpha)<\alpha\}\in\mathcal{V}_\alpha\cap \mathcal{K}$. By definition of $\mathcal{V}_\alpha$ and since $f\in \mathcal{K}$,  $\iota_{\delta_*}(f)(\eta_\omega)<\eta_\omega$. Also, Claim~\ref{ClaimOnetaOmega} yields $\iota_{\gamma_{\omega}+1}(f)(\eta_\omega)<\eta_\omega$, so that $\{\alpha<\eta_\omega\mid i_{\gamma_\omega}(f)(\alpha)<\alpha\}$ belongs to $\iota_{\gamma_\omega}(U_{\kappa,\zeta_\omega})$, the normal measure iterated at stage $\gamma_\omega$.  Thus, there is $X\in\iota_{\gamma_\omega}(U_{\kappa,\zeta_\omega})$ and $\theta<\eta_\omega=\iota_{\gamma_\omega}(\kappa)$ such that $\iota_{\gamma_\omega}(f)``X=\{\theta\}$. By elementarity, there is $X\in U_{\kappa,\zeta_\omega}$ and $\theta<\kappa$ such that $f``X=\{\theta\}$. Moreover, $X\in \mathcal{V}_\alpha\cap \mathcal{K}$: 
	\begin{equation*}
		X\in U_{\kappa,\zeta_\omega}\Leftrightarrow \iota_{\gamma_\omega}(X)\in \iota_{\gamma_\omega}( U_{\kappa,\zeta_\omega})\Leftrightarrow\eta_\omega\in \iota_{\gamma_{\omega},\gamma_\omega+1}(X)\Leftrightarrow \eta_\omega\in \iota_{\delta_*}(X).\qedhere
	\end{equation*}
	%Since $X\in\mathcal{W}_{\gamma_\omega}$ and the rest of the iteration has critical point above $\eta_\omega$ it follows that $\eta_\omega\in \iota_{\gamma_\omega,\delta_*}(X)$
\end{proof}
\begin{claim}\label{Calculatingzeta}
Let $\zeta_n$ be as above. Then, 
$$\zeta_n=
\begin{cases}
\xi_\gamma, &\text{if $\alpha=\gamma+1$;}\\
\xi_{\beta_{r(n)}}, & \text{if $\alpha\in \mathrm{acc}(\omega_1)$.}
\end{cases}
$$
%\begin{itemize}
 %   \item If $\alpha$ is a successor ordinal and $\alpha = \gamma + 1$ then $\zeta_n = \xi_\gamma$.  
 %   \item If $\alpha$ is a limit ordinal then $\zeta_n = \xi_{\beta_{r(n)}}$. 
%\end{itemize}
%Thus, the measure applied at stage $\gamma_n$ has Mitchell-order 
\end{claim}
\begin{proof}
%Let us show that this is the only possibility. Indeed, 
By discreteness of the measures $U_{\kappa,\xi}$'s, $\zeta_n$ is the unique ordinal $\zeta$ such that $\iota_{\gamma_n}(A_\zeta) \in \iota_{\gamma_n}(U_{\kappa,\zeta_n})$. By Claim~\ref{gammacriticalpoints} and normality of the iteration this is equivalent to say that
\[\eta_n \in \iota_{\gamma_n,\gamma_{n+1}}(\iota_{\gamma_n}(A_\zeta))\;\Leftrightarrow\; \eta_n\in j(A_\zeta).\]
%which is equivalent to 
%\[\eta_n \in j(A_\zeta)\]
%since the iteration is normal, and $\eta_n < \crit \iota_{\gamma_n + 1, \delta_*}$, so $\eta_n \in \iota_{\gamma + 1}(A_\zeta) \iff \eta_n \in \iota_{\delta_*}(A_\zeta) = j(A_\zeta)$.

Finally, %let us note that by the assumption on $\mathcal{V}'_\alpha$, 
this is equivalent to $A_\zeta \in \vec{\mathcal{V}}(n)\cap \mathcal{K}$. But $
\vec{\mathcal{V}}(n)\cap \mathcal{K}$ is exactly $U_{\kappa,\xi_\gamma}$ in the successor case and $U_{\kappa, \xi_{\beta_{r(n)}}}$ in the limit case, so the claim follows. 
\end{proof}
%Recall that $\langle A_\xi \mid \xi < o^\mathcal{K}(\kappa)\rangle$ is a sequence of pairwise disjoint sets, such that $A_\zeta$ is large with respect to $U_{\kappa,\xi}$ if and only if $\xi = \zeta$. 

%By the same argument, since the critical points of the iteration is above the length of the sequence of measures, there is a unique ordinal $\xi_n$ such that $\eta_n \in \iota_{0,\gamma_n + 1}(A_{\xi_n})$. \ale{I do not understand this point...}

%As the critical point of the rest of the iteration is strictly above $\eta_n$, we have that $\eta_n \in j(A_\xi)$. In particular, $\vec{\mathcal{V}}(n) \cap \mathcal{K} = U_{\kappa,\xi_n}$. \ale{Why?}

Let us look at $\mathcal{V}_{\alpha} \cap \mathcal{K}$. Since $\mathcal{V}_\alpha$ is $\kappa$-complete and $\mathcal{K}$-normal there is $\zeta<o^\mathcal{K}(\kappa)$ such that $\mathcal{V}_\alpha\cap \mathcal{K}=U_{\kappa,\zeta}$. We next show that $\zeta > \zeta_n$ for all $n<\omega$. From this we shall be able to infer that $\mathcal{V}_\beta\cap \mathcal{K}\unlhd \mathcal{V}_\alpha\cap \mathcal{K}$.% for all $\beta<\alpha$. \ale{Clarify}

\begin{claim}
$o^{\mathcal{K}^M}(\eta_\omega) = \zeta$.	
\end{claim}
\begin{proof}[Proof of claim]
Arguing as before, the measure used at stage $\gamma_\omega$ of the iteration is of the form $\iota_{\gamma_\omega}(U_{\kappa,\zeta_\omega})$ for some $\zeta_\omega<o^\mathcal{K}(\kappa)$. In particular, {$o^{\mathcal{K}_{\gamma_{\omega}+1}}(\eta_\omega)=\zeta_\omega$ and so, since $\eta_\omega<\crit(\iota_{{\gamma_\omega}+1,\delta_*})$, $o^{\mathcal{K}^M}(\eta_\omega)=\zeta_\omega.$} %\ale{Why true?} 
%Indeed, \[o^{\mathcal{K}^M}(\eta_\omega) =  o^{\mathcal{K}^M}(\iota_{\gamma_\omega, \delta_*}(\eta_\omega)) = \iota_{\gamma_\omega, \delta_*}(o^{\mathcal{K}_{\gamma_\omega}}(\eta_\omega))\], since $\iota_{\gamma_\omega, \delta_*}(\eta_\omega) = \eta_\omega$ and $\iota_{\gamma_\omega, \delta_*}(\zeta_\omega) = \zeta_\omega$.

	%Again, by discreteness, $\zeta$ is the unique ordinal such that $A_\zeta\in U_{\kappa,\zeta}$. Thus, by elementarity, 
	
	Let us now show that $\zeta_\omega=\zeta$. Once again, by discreteness, $\zeta_{\omega}$ is the unique ordinal such that $A_{\zeta_\omega}\in U_{\kappa,\zeta_\omega}$. By elementarity this is equivalent to $\iota_{\gamma_\omega}(A_{\zeta_\omega})\in \iota_{\gamma_\omega}(U_{\kappa,\zeta_\omega})$, which is equivalent to $\eta_\omega\in j(A_{\zeta_\omega})$. All in all, we have that $A_{\zeta_\omega}\in \mathcal{V}_\alpha\cap \mathcal{K}=U_{\kappa,\zeta}$. Thus, by our choice on $A_{\zeta_\omega}$, $\zeta=\zeta_\omega$.
\end{proof}
%\textcolor{blue}{By normality of the iteration, $o^{\mathcal{K}^M}(\eta_\omega) = \zeta$ (as after the step $\gamma_\omega$ the Mitchell order is $\zeta < \eta_\omega$ and the critical point of the rest of the embedding is strictly above $\eta_\omega$). } \ale{This is what was originally written}

\begin{claim}
For each $n<\omega$ there is a set $B_n\in[\omega]^{\omega}$ such that $\zeta_m=\zeta_n$ for all $m\in B_n$.
\end{claim}
\begin{proof}[Proof of claim]
Fix $n<\omega$. If $\alpha$ was a successor ordinal then we can let $B_n=\omega$ for in that case $\zeta_m=\xi_\gamma$ for every $m<\omega$. Otherwise, if $\alpha$ is limit, put $B_n:=r^{-1}\{r(n)\}$. By our choice upon $r$, $B_n\in[\omega]^{\omega}$. Also, by Claim~\ref{Calculatingzeta} $\zeta_m=\xi_{\beta_{r(m)}}=\xi_{\beta_{r(n)}}=\zeta_n$. 
\end{proof}

\begin{claim}
	For each $n<\omega$, $o^{\mathcal{K}^M}(\eta_\omega)>\zeta_n$.
\end{claim}
\begin{proof}[Proof of claim]
Fix $n<\omega$ and let $B_n\in[\omega]^{\omega}$ be as in the previous claim. Since $M$ is closed under $\omega$-sequences we have that $\langle \eta_m\mid m\in B_n\rangle \in M$. Thus, working in  $M$ we can define the following filter: \[\mathcal{F}_n := \{X \subseteq \eta_\omega \mid \exists k\,  \forall m \in (B_n\setminus k)\; \eta_m \in X\}.\]
We claim that $\mathcal{F}_n\cap \mathcal{K}^M=\iota_{\gamma_\omega}(\mathcal{U}_{\kappa,\zeta_n})$. Note that if this is the case $\mathcal{F}_n\cap \mathcal{K}^M$ is a measure on $\eta_\omega$ in $\mathcal{K}^M$ of Mitchell order $\zeta_n$ and thus $o^{\mathcal{K}^M}(\eta_\omega)\geq \zeta_n+1$.

Let $X\in \iota_{\gamma_\omega}(\mathcal{U}_{\kappa,\zeta_n})$. Note that $X$ takes the form $\iota_{\gamma_k,\gamma_\omega}(\bar{X}_k)$ for a tail of $k<\omega$ and $\bar{X}_k\in \mathcal{K}_k$. Let $\bar{k}$ be the least of such indices and $m\in B_n\setminus \bar{k}$. Since $X=\iota_{\gamma_m,\gamma_\omega}(\bar{X}_m)$ then, by elementarity, $\bar{X}_m\in \iota_{\gamma_m}(\mathcal{U}_{\kappa,\zeta_n})=\iota_{\gamma_m}(\mathcal{U}_{\kappa,\zeta_m}).$\footnote{Note that this latter equality follows by our choice of $B_n$.} Since this latter is the measure in $\mathcal{K}_{\gamma_m}$ used at stage $\gamma_m$ we have that $\eta_m\in\iota_{\gamma_m,\gamma_{m+1}}(\bar{X}_m)$, which is equivalent to $\eta_m\in \iota_{\gamma_m,\gamma_\omega}(\bar{X}_m)=X.$

Conversely, let $X\in \mathcal{K}^M$, $X\s \eta_\omega$ be such that $X\notin \iota_{\gamma_\omega}(\mathcal{U}_{\kappa,\zeta_n})$. Since $\crit(\iota_{\gamma_{\omega},\gamma_\omega+1})=\eta_\omega$ we have that $X\in \mathcal{K}_{\gamma_\omega}$. Fix $k<\omega$. Arguing as before find $m\in B_n\setminus k$ such that $\iota_{\gamma_m,\gamma_\omega}(\bar{X}_m)=X$. By elementarity, $\bar{X}_m\notin \iota_{\gamma_m}(\mathcal{U}_{\kappa,\zeta_m})$, hence $\eta_m\notin \iota_{\gamma_m,\gamma_{m+1}}(\bar{X}_m)$, and thus $\eta_m\notin X$. Thereby, $X\notin \mathcal{F}_n$.
\end{proof}
Thus, the above arguments show that $o^{\mathcal{K}^M}(\eta_\omega)=\zeta>\zeta_n$ for all $n<\omega$.
 %Fix an infinite $B \subseteq \omega$ such that $\zeta_n = \zeta_*$ for all $n \in B$. The set $B$ can be either $\omega$ if $\alpha$ is a successor ordinal or $r^{-1}(\{n\})$ if $\alpha$ is a limit ordinal.
%The sequence $\langle \eta_n \mid n < \omega\rangle$ belongs to $M$, by its closure under $\omega$-sequences. In particular, working in  $M$ we can define the following filter: \[\mathcal{F} = \{X \subseteq \eta_\omega \mid \exists n\,  \forall m \in (B\setminus n)\; \eta_m \in X\}.\]
%Let $X \in \iota_{\gamma_\omega}(U_{\kappa,\zeta_m})$. Then, there is $m \in B$ and $X' \subseteq \eta_m$, such that $X = \iota_{\gamma_m, \gamma_\omega}(X')$. Therefore, for all $n \geq m$ in $B$, $\eta_n \in X$, and in particular $X \in \mathcal{F}$. Similarly, if $X \notin \iota_{\gamma_\omega}(U_{\kappa,\xi_n}$ then there are $m$ and $X'$ as before, but $X' \notin \iota_{\gamma_m}(U_{\kappa, \zeta_m})$. Then, $\eta_n \notin X$ for all $n\geq m$.
%We conclude that $\mathcal{F} \cap \mathcal{K}^M = \iota_{\gamma_{\omega}}(U_{\kappa,\xi_n})$, and in particular, this measure \emph{belongs} to $\mathcal{K}^M$. Therefore, $o^{\mathcal{K}^M}(\eta_\omega) > \zeta_n$.

\begin{claim}
	$\mathcal{V}_\beta\cap \mathcal{K}\unlhd \mathcal{V}_\alpha\cap \mathcal{K}$ for all $\beta<\alpha$.
\end{claim}
\begin{proof}[Proof of claim]
If $\alpha=\gamma+1$ then $\mathcal{V}_\gamma\cap \mathcal{K}=\mathcal{U}_{\kappa,\xi_\gamma}=\mathcal{U}_{\kappa,\zeta_n}\unlhd \mathcal{U}_{\kappa,\zeta}=\mathcal{V}_\alpha\cap \mathcal{K}$.

Suppose that $\alpha$ is a limit ordinal and let $\beta<\alpha$. Choose $\beta_n\in (\beta, \alpha)$ and observe that $\beta_n=\beta_{r(m)}$ for some (infinitely-many) $m<\omega$. In particular, $\mathcal{V}_\beta\cap \mathcal{K}\unlhd  \mathcal{V}_{\beta_n}\cap \mathcal{K}=\mathcal{V}_{\beta_{r(m)}}\cap \mathcal{K}=\mathcal{U}_{\kappa,\zeta_m}\unlhd \mathcal{U}_{\kappa,\zeta}=\mathcal{V}_\alpha\cap \mathcal{K}.$
\end{proof}
This completes the argument for the inductive step of the construction and as result the proof of the theorem.
\end{proof}
\begin{remark}
    Observant readers will have  noticed that the previous proof is    more informative than what we said in Theorem~\ref{lowerboundgluing}. Indeed, under the same anti-large-cardinal assumptions,  if $\kappa$ has the $\lambda$-gluing property for some infinite cardinal $\lambda<\kappa$ then $o^\mathcal{K}(\kappa)\geq \lambda^+.$ Note that, even in this more general setting, equation \eqref{orderislow} above stands valid. This was crucial in identifying which are the measures in $\mathcal{K}$ corresponding to the generators $\eta_n$'s. %\ale{Am I right?}
\end{remark}

\section{Open problems}\label{OpenProblems}
We would like to conclude the paper mentioning a few  open problems.

\smallskip

%In \S\ref{section:consistent-bounds-on-gluing} we gave a forcing argument showing that (consistently) a $\kappa^{+}$-$\Pi^1_1$-subcompact cardinal $\kappa$  can fail to have the $(2^{\kappa})^+$-gluing property. Intuitively, we expect a better result:
\begin{question}
    Can the least cardinal $\kappa$ which is $\kappa^{+}$-$\Pi^1_1$-subcompact have the $2^{2^{\kappa}}$-gluing property while $2^{2^\kappa}>\kappa^{++}$?\footnote{Note that if $2^{2^\kappa}=\kappa^{++}$ then Lemma \ref{lemma: pi11havelotsofgluings}.}    
\end{question}

While we suspect that the consistency strength of the gluing property is low, it is rather unclear which large cardinals weaker than strongly compact outright imply gluing. For example, a standard core model argument shows that if there is no inner model with a Woodin cardinal then in the core model the least strong cardinal does not have the $\omega$-gluing property.\footnote{See Remark~\ref{StrongsDoNotHaveGluing}.%Indeed, if $j \colon \mathcal{K} \to \mathcal{K}^M$ is an ultrapower embedding and $M$ is closed under $\kappa$-sequences, then it must be a finite iteration of extenders from $\mathcal{K}$. In particular, there are at most finitely many $\alpha < j(\kappa)$ which belong to $\bigcap \{j(C) \mid C\in\mathrm{Cub}_\kappa\}$.
} %\ale{Could you elaborate more on this footnote? Why $\kappa$-closure of $M$ implies that $j$ is finite? Is it because $\mathcal{K}^M=\mathcal{K}\cap M?$}
\begin{question}
    What large cardinals do have the $\omega$-gluing property? More generally, which of them do have the $\lambda$-gluing property for every cardinal $\lambda$?
\end{question}
In light of the striking structural consequences of Goldberg's \emph{Ultrapower Axiom} (UA)\cite{Goldberg} the following speculation seems natural:
\begin{question}
   Is there a characterization, under $\mathrm{UA}$, of those measurable cardinals having the  $\lambda$-gluing property?
\end{question}
In \S\ref{section; improving} and \S\ref{section:lower-bound} we show that the  consistency strength of the $\omega$-gluing proper\-ty is exactly $``\exists\kappa\,(o(\kappa)=\omega_1)$''. This raises an obvious question:
\begin{question}
   What is the  consistency strength of ``There is a measurable cardinal $\kappa$ with the $\lambda$-gluing property'' for $\lambda\geq \kappa$?  
\end{question}
Our conjecture is that starting with a strong cardinal one should be able to produce a model with a measurable cardinal having the $\lambda$-gluing property for all cardinals $\lambda$. This combined with Gitik's theorem from \S1 would yield the exact consistency strength of this compactness principle. 

\smallskip

Another interesting question from the technical perspective is:

\begin{question}
  Suppose that $\kappa$ has the $\lambda$-gluing property. Is there a forcing poset destroying this property? In general, using forcing, can we tune the exact amount of gluing that $\kappa$ carries?
\end{question}
Lastly, in \S\ref{section; improving}, we gave a characterization of the $\kappa$-complete ultrafilters on $\kappa$ in a certain forcing extension of $\mathcal{K}$. By recent intriguing results of Gitik and Kaplan \cite{GitikKaplan-nonstationary2022}, we suspect that the assumption $\mathrm{V} = \mathcal{K}$ can be dropped (or replaced by a suitable $\GCH$ hypothesis). 

\begin{question}
Let $\mathbb{P}$ be a non-stationary support iteration of tree Prikry forcings over an arbitrary ground model $V$. Is there a full characterization of the $\kappa$-complete ultrafilters on $\kappa$ in the generic extension, similar to Lemma~~\ref{lemma:representing-ultrafilters}?
\end{question}
\subsection*{Acknowledgements}
Hayut's research was supported by the ISF grant 1967/21. Poveda wishes to thank the Einstein Institute of Mathematics at the Hebrew University and the Center of Mathematical Sciences and Applications at Harvard University for their generous support during the  duration of this project. Both authors are very grateful to the anonymous referee for his/her extensive list of comments and suggestions. In addition, they are indebted to Gabriel Goldberg for pointing us that \cite[Lemma 4.7]{HP} was wrong and for prompting us to prove Theorem 4.12.

\bibliographystyle{alpha}
\bibliography{biblio}

\end{document}